\documentclass{article}
\usepackage{latexsym}
\usepackage{amssymb}
\usepackage{amsmath}
\usepackage{color}
\usepackage{graphicx}
\usepackage{amsthm}
\usepackage{indentfirst}
\usepackage{mathrsfs}
\usepackage{dsfont}
\usepackage{enumitem}
\usepackage{multirow}
\usepackage{extpfeil,extarrows}
\usepackage{galois}
\usepackage{float}
\usepackage{longtable}
\usepackage{pdfpages}
\usepackage{appendix}
\allowdisplaybreaks

\newtheorem{theorem}{\color{black}\indent Theorem}[section]

\newtheorem{lemma}{\color{black}\indent Lemma}[section]
\newtheorem{remark}{\color{black}\indent Remark}
\newtheorem{definition}{\color{black}\indent Definition}[section]
\newtheorem{cor}{\color{black}\indent Corollary}[section]
\newtheorem{pro}{\color{black}\indent Property}[section]
\newtheorem{example}{\color{black}\indent Example}

\begin{document}

\title{Quasiperiodic Poincar\'{e} Persistence at High Degeneracy\footnote{The first author was supported by China Postdoctoral Science Foundation (2021M701396, 2022T150262), NSFC(12201243). The second author was supported by National Basic Research Program of China (2013CB834100), NSFC (11571065,11171132), JilinDRC (2017c028-1). The third author was supported by NSFC (11201173,12371191), Science and Technology Developing Plan of Jilin Province (20180101220JC). }}
\author{ Weichao Qian$^a$\thanks{E-mail address: qian\_wc@163.com}, ~ Yong Li$^{a,b}$\thanks{E-mail address: liyongmath@163.com}~\footnote{Corresponding author}, ~ Xue Yang$^{a,b}$\thanks{E-mail address: xueyang@jlu.edu.cn}
\\
{$^a$College of Mathematics, Jilin University, P. R. China}
\\
{$^b$School of Mathematics and Statistics \&}
\\
{Center for Mathematics and Interdisciplinary Sciences, }
\\
{Northeast Normal University, P. R. China}
}
\date{}

\maketitle

\begin{abstract}
For Hamiltonian systems with high-order degenerate perturbation, we study the persistence of resonant invariant tori, where the resonant tori might be elliptic, hyperbolic or mixed types. As a consequence, we prove a quasiperiodic Poincar\'{e} theorem at high degeneracy. This answers a long standing conjecture on the persistence of resonant invariant tori in general situations.

\end{abstract}

\section{Introduction}\label{introduction}

This paper concerns the persistence of resonant invariant tori for the following Hamiltonian system
\begin{eqnarray}\label{005}
H(\theta, I)=H_0 (I) + \varepsilon P(\theta,I,\varepsilon),
\end{eqnarray}
where  $\theta \in T^d = R^d/ Z^d$, $I$ $\in$ $G$ ($G$ is a bounded closed region in $R^d$), $H_0 (I)$ and $ P(\theta, I, \varepsilon) =  P_0(I, \theta,0)+ \sum\limits_{1\leq \iota \in Z_+}\frac{\varepsilon^{\iota }}{\iota!} P_\iota(I, \theta,0)$ are real analytic functions
on a complex neighborhood of the bounded closed region $T^d \times G$ and $\varepsilon > 0$ is a small parameter. Here the so-called resonant invariant tori mean the frequency $\omega(I) = \frac{\partial H_0}{\partial I}$ is resonant for some $I$, i.e., there exists at least one $k \in {Z^d \setminus \{0\}}$ such that $\langle k, \omega \rangle = 0$.

 The celebrated KAM theory due to Kolmogorov, Arnold and Moser asserts that, if an integrable system, $H_0(I)$ in (\ref{005}), is nondegenerate, i.e. $\det \partial_{I}^2 H_0 \neq 0$, then, for the perturbed system $H (\theta,I) = H_0(I)+ \varepsilon P(\theta,I,\varepsilon),$  most of nonresonant invariant tori still survive (\cite{Arnold,Kolmogorov,Moser}). For some recent developments and applications related to KAM theory, refer to \cite{Guardia,han,Kaloshin1,Meyer,Palacian,Palacian1,Qian}. However, in the presence of resonance, the persistence problem becomes very complicated. Let us do a brief recall. The periodic case can go back to
 the work of Poincar\'{e} in nineteenth century, which does not involve the small divisor problem(\cite{Poincare}). There has been a long standing conjecture about resonant tori under a convexity assumption on $H_0$
 (\cite{Broer,Cong,Livia,Gentile,Kappeler}), as written by Kappeler and P\"{o}schel in \cite{Kappeler}:
 \begin{itemize}
   \item[] \emph{For $m = 1 $ in particular, such a torus is foliated into identical closed orbits. Bernstein $\&$ Katok $(\emph{\cite{Bernstein}})$ showed that in a convex system at least $d$ of them survive any sufficiently small perturbation. $\cdots$ For the intermediate cases with $1<m<d-1$, only partial results are known $\cdots$. The long standing conjecture is that at least $d- m +1$, and generically $2^{d-m}$, invariant $m-$tori always survive in a nondegenerate system $\cdots$. That is, their number should be equal to the number of critical points of smooth functions on the torus $T^{d-m}$.}
 \end{itemize}
 In above description, $m$ and $d$ are dimensions of the lower-dimensional invariant tori and the degree of freedom, respectively.

 The first breakthrough of the conjecture mentioned above was due to Tresch\"{e}v (\cite{Treshchev}) for the persistence of hyperbolic resonant tori in 1989, 35 years after the establishment of KAM theory, and such tori are called Tresch\"{e}v's tori today. For the persistence of general resonant tori, we refer readers to \cite{Cong,Li,Xu2,xu3}. In fact, for Hamiltonian system $(\ref{005})$, when $P_0(I, \theta,0)$ in $\varepsilon P(\theta, I, \varepsilon) = \varepsilon P_0(I, \theta,0)+ \sum\limits_{1\leq \iota \in Z_+}\varepsilon^{\iota +1} P_\iota(I, \theta,0)$ is nondegenerate (we will explain what `non-degenerate` means later), the proof of the conjecture mentioned above has been completed, see \cite{Cheng1,Cong,Li,Treshchev}. However,
 \begin{itemize}
   \item[] \emph{What happens to the conjecture if $P_0(I, \theta, 0)$ is degenerate?}
 \end{itemize}
In the present paper we will touch this essential problem.

In order to state our main result, first, let us introduce some notations. We say that a frequency vector $\omega = {\partial_{I} H_0} $ is nonresonant for some $I$, if $\langle k, \omega \rangle \neq 0$ for any $k \in {Z^d \setminus \{0\}}$. Furthermore, if there is a subgroup $g$ of $Z^d$ such that $\langle k , \omega\rangle = 0$ for all $k \in g$ and $\langle k , \omega \rangle \neq 0$ for all $k \in {Z^d / g}$, then $\omega$ is called \textbf{multiplicity $m_0$ resonant frequency} ($g-$resonant frequency), where $g$ is generated by independent $d-$dimensional integer vectors $\tau _1, \ldots , \tau _ {m_0}$. For a given subgroup $g$, the manifold
\begin{eqnarray*}
\widetilde{\Lambda}(g, G)=\{I \in G : \langle k , \omega (I) \rangle =0, k\in g \}
\end{eqnarray*}
is called $g-$resonant surface. By group theory, there are integer vectors ${\tau _1'}, \cdots, {\tau _m'}$ $\in$ $Z ^ d$, such that $Z^d$ is generated by ${\tau _1}, \cdots, {\tau _{m_0}},{\tau _1'}, \cdots, {\tau _m'}$, and~ $\det K_0 = 1 $, where $K_0 = (K_* , K^{'})$,~ $K_* = (\tau _1', \cdots, \tau _m')$, $K^{'}=({\tau _1}, \cdots, {\tau _{m_0}}) $ are ~$d\times d$ matrix, $d\times m $ matrix, $d\times m_0$ matrix, respectively,  and $K_{*}$ generates the quotient group $Z^d / g$, while $K^{'}$ generates the group $g$ (\cite{Treshchev}). If $\det {K^{'}}^T \partial_{I}^2 {H_0}{K^{'}} \neq 0$ and $\det \partial_{I}^2 {H_0} \neq 0$ for $I \in \widetilde{\Lambda}(g,G)$, Hamiltonian system (\ref{005}) is called \textbf{$g-$nondegenerate}.

The motion equation of the unperturbed Hamiltonian system $H_0(I)$ in (\ref{005}) is
\begin{eqnarray*}
\left\{
  \begin{array}{ll}
    \dot{\theta} = \omega(I),  \\
    \dot{I} = 0.
  \end{array}
\right.
\end{eqnarray*}
Denote $p=(y, v),$ $q= (x,u)$, where $y = (p_1, \cdots ,p_m)^T,$ $v = (p_{m+1}, \cdots ,p_d)^T,$ $x = (q_1, \cdots ,q_m)^T,$ $u = (q_{m+1}, \cdots ,q_d)^T.$ When $\omega(I)$ is $g-$resonant, under the following sympletic transformation
\begin{eqnarray}\label{ST1}
\phi_g: (p,q) \rightarrow (I, \theta),
\end{eqnarray}
 where $K_0 ^T \theta = q, I - I_0 = K_0 p,$
the equation of motion becomes
\begin{eqnarray*}
\left\{
  \begin{array}{llll}
    \dot{x} = K_*^T \omega(I),  \\
    \dot{u} = 0,\\
\dot{y} = 0,\\
\dot{v} = 0,
  \end{array}
\right.
\end{eqnarray*}
where $K_0$ and $K_*$ are mentioned as above. (We place the verification that $\phi_g$ is sympletic on Appendix \ref{A}.) We call such $(y, v, u)$ the relative critical point.

With transformation $\phi_g$, Hamiltonian system (\ref{005}) could be transformed to
\begin{eqnarray}
H(x,y,u,v) = H\circ \phi_g = \tilde{H}_0(y,v)+\varepsilon \tilde{P}(x,y,u,v,\varepsilon),
\end{eqnarray}
where $$\tilde{P}(x,y,u,v,\varepsilon) = P((K_0 ^T)^{-1} \left(
          \begin{array}{c}
            x \\
            u \\
          \end{array}
        \right), I_0 + K_0 \left(
             \begin{array}{c}
               y \\
               v \\
             \end{array}
           \right),\varepsilon) = \sum\limits_{\iota}\frac{\varepsilon^{\iota }}{\iota!} \tilde{P}_\iota(x,y,u,v,0).$$
(For the normal form in detail, refer to section \ref{074}.) Let $[\tilde{P}](y,u,v,\varepsilon)=\\ \int_{T^m} \tilde{P}(x, y,u,v,\varepsilon) dx = \sum\limits_{\iota}\frac{\varepsilon^{\iota }}{\iota!} [\tilde{P}_\iota](y_0,u_0,v_0,0).$ When $\det \partial_{I}^2 {H_0} \neq 0$ and no eigenvalue of $\partial_u ^2[\tilde{P}_0] {K^{'}}^T \partial_{I}^2 {H_0}{K^{'}}$ is positive or zero, Tresch\"{e}v $(\cite{Treshchev})$ dealt with the persistence of resonant tori. When $\det \partial_u^2 [\tilde{P}_0] \neq0$, for $g-$nondegenerate Hamiltonian system $(\ref{005})$, Cong, K\"{u}pper, Li and You (\cite{Cong}) dealt with the persistence of resonant invariant tori.  Li and Yi$(\cite{Li})$ further removed the $g-$nondegenerate condition. When $\det \partial_u^2 [\tilde{P}_0]=0$, what happens to the persistence of resonant tori becomes very complicated. The conjecture says that the number of the survival resonant tori is at least $m_0+1$ and generically $2^{m_0}$ for nondegenerate systems. We call perturbation $\tilde{P}(x,y,u,v ,\varepsilon)$ \textbf{$\kappa-$order nondegenerate}, if $\det \partial_u^2 [\tilde{P}_\iota]( y_0,u_0,v_0,0) =0$ for $0\leq\iota\leq \kappa-1$ and $\det \partial_u^2 [\tilde{P}_{\kappa}]( y_0,u_0,v_0,0) \neq 0$, where $(y_0, u_0, v_0)$ is the critical point of $[\tilde{P}]$. Obviously, above results only deal with the persistence of resonant tori for Hamiltonian (\ref{005}) with $0$-order nondegenerate perturbation. In the present paper, we prove that $2^{m_0}$ families of invariant torus survive for Hamiltonian system (\ref{005}) with $\kappa-$order nondegenerate perturbation, where $\kappa$ is a given integer.

Now we are in a position to state our main results. We call $P(I, \theta, \varepsilon)$ in $(\ref{005})$ $\kappa-$order nondegenerate, if there is a symplectic transformation $\phi_g$ as in $(\ref{ST1})$ such that $\tilde{P}(x,y,u,v,\varepsilon) = P \circ \phi_g$ is $\kappa-$oder nondegenerate. First, we show results about a simple case, a $g-$nondegenerate Hamiltonian system with $\kappa$-order nondegenerate perturbation.

\begin{theorem} \label{cor2}
Let $g-$nondegenerate Hamiltonian system (\ref{005}) with $\kappa-$order nondegenerate perturbation $P(I, \theta,\varepsilon)$ be real analytic on the complex neighborhood of $T^d\times G$. We have:
\begin{itemize}
  \item [$\bf{i)}$]There exists a $\varepsilon_0 >0 $ and a family of Cantor sets $\widetilde{\Lambda}_\varepsilon(g,G) \subset \widetilde{\Lambda}(g,G)$, $0<\varepsilon < \varepsilon_0$, such that for each $I \in \widetilde{\Lambda}_\varepsilon(g,G)$, system (\ref{005}) admits $2^{m_0}$ families of invariant torus, possessing  hyperbolic, elliptic or mixed types, associated to nondegenerate relative equilibria. All such perturbed tori corresponding to a same $I \in \widetilde{\Lambda}_\varepsilon(g,G) $ are symplectically conjugated to the standard quasiperiodic $m-$tori $T^m$ with the Diophantine frequency vector $\omega_* = K_*^T \partial_{I} H_0(I)$. Moreover, the relative Lebesgue measure $|\widetilde{\Lambda}(g,G) \setminus \widetilde{\Lambda}_{\varepsilon}(g,G)|$ tends to $0$ as $\varepsilon \rightarrow 0$.
  \item [$\bf{ii)}$] Consider $g-$nondegenerate Hamiltonian system $(\ref{005})$ with $\kappa$-order nondegenerate perturbation $P(I,\theta, \varepsilon)$ on $\sum = \{I: H_0(I) = c\}$.  Assume
  \begin{enumerate}
   \item [\bf{(S1).}]$
rank \left(
       \begin{array}{cc}
         K_0^T \partial_{I}^2 H_0(I) K_0 & \bar{\omega}_* \\
         \bar{\omega}_*^T & 0 \\
       \end{array}
     \right) = m+ m_0+1
$, where $\bar{\omega}_* = \left(
                                                                         \begin{array}{c}
                                                                           \omega_* \\
                                                                           0 \\
                                                                         \end{array}
                                                                       \right)
\in R^{m+ m_0}$, $\omega_* = K_*^T \partial_{I} H_0(I).$
 \end{enumerate} Then there exists a $\varepsilon_0 >0 $ and a family of Cantor sets $\tilde{\Xi}_\varepsilon \subset \tilde{\Xi}=\{I \in G: H_0(I) = c, \langle k, \omega\rangle = 0,  k\in g\}$, $0<\varepsilon < \varepsilon_0$, such that for each $I \in \tilde{\Xi}_\varepsilon$, on a given energy-level manifold system (\ref{005}) admits $2^{m_0}$ families of invariant torus, possessing  hyperbolic, elliptic or mixed types, associated to nondegenerate relative equilibria. The frequencies $\breve{\omega}$ of the persistent tori satisfy that $\breve{\omega} =t \omega_*$, where $t\rightarrow 1$ as $\varepsilon\rightarrow 0$. Moreover, the relative Lebesgue measure $|\tilde{\Xi} \setminus \tilde{\Xi}_\varepsilon |$ tends to $0$ as $\varepsilon \rightarrow 0$.
\end{itemize}
\end{theorem}

\begin{remark}
Here a map defined on a Cantor set is said to be smooth in Whitney's sense if it has a smooth Whitney extension. For details, see $\emph{\cite{Poschel}}$.
\end{remark}

Since $[\tilde{P}](y,u,v, \varepsilon)$ is $T^{m_0}-$periodic in $u$, there are at least $m_0 + 1$ critical points for $[\tilde{P}]( y,u,v,\varepsilon)$ for given $y_0$, $v_0$ and $\varepsilon_0$(\cite{Milnor}). Note that $[\tilde{P}]$ is $\kappa$-order nondegenerate, $\det \partial_u^2 [\tilde{P}_{\kappa}]( y_0,u_0,v_0,0) \neq 0$, where $(y_0,u_0, v_0)$ is relative critical point, which means that such perturbations are generic according to Morse theory $(\cite{Hirsch,Milnor})$. Therefore, Theorem \ref{cor2} shows the persistence of resonant tori for a $g-$nondegenerate Hamiltonian system $(\ref{005})$ with a generic perturbation in the sense of the $\kappa$-order nondegeneracy, where $\kappa$ is a given positive integer. Hence this positively verifies the conjecture mentioned above in a general situation for $g$-nondegenerate Hamiltonian system (\ref{005}).

The $\kappa-$order nondegenerate perturbation in the present paper is different from the case given by Tresch\"{e}v $(\cite{Treshchev})$, where the corresponding Hamiltonian is the following:
\begin{eqnarray}\label{Eq13}
H(x,y,\varepsilon) = H_0(y)+ \varepsilon H_1(y)+ \cdots + \varepsilon^k H_k(y) + \varepsilon^{k+1} H_{k+1}(x,y,\varepsilon).
\end{eqnarray}
If there is some condition on the $0-$order Taylor coefficient of the average of $H_{k+1}$ in $(\ref{Eq13})$, he obtained the persistence of resonant tori (hyperbolic), and for some recent developments of such system, refer to {\cite{han,Qian2,Xu2,xu3}}. Actually, for the nearly integrable Hamiltonian system with a resonant integrable part and a $\kappa$-order nondegenerate perturbation, with finite KAM steps Hamiltonian system (\ref{005}) can be reduced to the following system:
\begin{eqnarray*}
H(x,y,u,v,\varepsilon) &=& H_0(y,v)+ \varepsilon H_1(y,u,v)+ \cdots + \varepsilon^{\kappa} H_{\kappa}(y,u,v) \\
&~&+ \varepsilon^{\kappa+1} H_{\kappa+1}(x,y,v,u,\varepsilon),
\end{eqnarray*}
where $y$ and $v$ come from $I$ of the original system $(\ref{005})$, $x$ and $u$ come from $\theta$ of the original system $(\ref{005})$. For detail definitions and the process of reduction, refer to Section \emph{\ref{074}}. Moreover, $\kappa-$order nondegenerate perturbation ensures the relative equilibria of $\varepsilon H_1(y,u,v)+ \cdots + \varepsilon^\kappa H_{\kappa+1}(y,u,v)$ is nondegenerate, which means there are $2^{m_0}$ relative critical points according to Morse theory.

Next, we will give a more general case, in which we remove the $g-$nondegeneracy and study the partial preservation of frequency and partial preservation of  ratios of frequencies. Let us do some assumptions for Hamiltonian system (\ref{005}) first:
\begin{enumerate}
   \item [\bf{(S2)}.] For $H_0(I)$ in (\ref{005}), $\omega_*(I) = K_*^T \partial_{I} H_0(I)$ satisfies R\"{u}ssmann non-degenerate condition, i.e., for some $N>0$, $rank \{\partial_I^\alpha \omega_*(I), |\alpha| < N  \} = m$ for every $I\in \widetilde{\Lambda}(g, G)$, where $\partial_{I}^\alpha \omega_*(I) = \frac{\partial^{|\alpha|}\omega_*}{\partial I_1^{\alpha_1}\cdots \partial I_d^{\alpha_d}}$, $\alpha = (\alpha_1,\cdots,\alpha_d)\in Z_+^d, |\alpha|= |\alpha_1|+\cdots+|\alpha_d|\leq N$$;$
   \item [\bf{(S3)}.]$ rank ~( K_0^T \partial_{I}^2  H_0  K_0 ) = n+ m_0$, $0\leq n\leq m$, and $rank((K')^T \partial_{I}^2 H_0 K_*, \\ (K')^T \partial_{I}^2 H_0 K') = m_0$, where $H_0(I)$ comes from (\ref{005}), $K_0$ and $K_*$ are defined as above$;$
   \item [\bf{(S4)}.]  $rank \left(
       \begin{array}{cc}
         K_0^T \partial_{I}^2 H K_0 & \bar{\omega}_* \\
         \bar{\omega}_*^T & 0 \\
       \end{array}
     \right) = n + m_0+1$, $0\leq n\leq m$, where $\bar{\omega}_* =\left(
                                                                                                   \begin{array}{c}
                                                                                                     \omega_* \\
                                                                                                     0 \\
                                                                                                   \end{array}
                                                                                                 \right)
      \in R^{m+ m_0}$, $\omega_* = K_*^T \partial_{I} H_0(I)\in R^{m}$, $H_0(I)$ comes from $(\ref{005})$, $K_0$ and $K_*$ are defined as above.
 \end{enumerate}

Now, let us state these more general results.
\begin{theorem}\label{dingli11}
Let Hamiltonian system $(\ref{005})$  with a $\kappa$-order nondegenerate perturbation $P(I, \theta, \varepsilon)$ be real analytic
on the complex neighborhood of $T^d \times G$. We have:
\begin{itemize}
  \item [$\bf{i)}$] Assume $\bf{(S2)}$ and $\bf{(S3)}$ hold. Then there exists a $\varepsilon_0 >0 $ and a family of Cantor sets $\widetilde{\Lambda}_\varepsilon(g,G) \subset \widetilde{\Lambda}(g,G)$, $0<\varepsilon < \varepsilon_0$, such that for each $I \in \widetilde{\Lambda}_\varepsilon(g,G)$, system (\ref{005}) admits at least $2^{m_0}$ families of invariant torus, possessing  hyperbolic, elliptic or mixed types, associated to nondegenerate relative equilibria. And $n$ coordinates of the frequency $\breve{\omega}$ on the persistent tori coincide with $n$ coordinates of  $\omega_*$. Moreover, the relative Lebesgue measure $|\widetilde{\Lambda}(g,G) \setminus \widetilde{\Lambda}_{\varepsilon}(g,G)|$ tends to $0$ as $\varepsilon \rightarrow 0$.
  \item [$\bf{ii)}$]Consider Hamiltonian system $(\ref{005})$ with a $\kappa$-order nondegenerate perturbation $P(I, \theta, \varepsilon)$ on $\sum = \{I: H_0(I) = c\}$. Assume $\bf{(S2)}$, $\bf{(S3)}$ and $\bf{(S4)}$ hold on $\sum$. Let $\tilde{\Xi}=\{I\in G: H_0(I) = c, \langle k, \omega\rangle = 0,  k\in g\}.$ Then there exists a $\varepsilon_0 >0 $ and a family of Cantor sets $\tilde{\Xi}_\varepsilon \subset \tilde{\Xi}$, $0<\varepsilon < \varepsilon_0$, such that for each $I \in \tilde{\Xi}_\varepsilon$, on a given energy-level manifold, system (\ref{005}) admits at least $2^{m_0}$ families of invariant torus, possessing  hyperbolic, elliptic or mixed types, associated to nondegenerate relative equilibria. And $n$ coordinates of the frequency $\breve{\omega}$ on the persistent tori coincide with $n$ coordinates of $t \omega_*$, where $t \rightarrow 1$ as $\varepsilon \rightarrow 0$. Moreover, the relative Lebesgue measure $|\tilde{\Xi} \setminus \tilde{\Xi}_\varepsilon |$ tends to $0$ as $\varepsilon \rightarrow 0$.
\end{itemize}

\end{theorem}

\begin{remark}\label{Eq11}
Consider the following Hamiltonian system
\begin{eqnarray}
\nonumber H(\tilde{x}, \tilde{y}) &=& \langle \tilde{\omega}, \tilde{y} \rangle+ \frac{\varepsilon}{2} \langle \tilde{y}, M
\tilde{y} \rangle + \varepsilon^3 \cos(2x_1- x_2)\\
\label{EQ40}&& + \varepsilon^2 \cos(2x_1- x_2)\sin (- x_1) e^{-y_1 - 2y_2},
\end{eqnarray}
where $\tilde{x}= (x_1, x_2)^T$, $\tilde{y}= (y_1, y_2)^T$, $\tilde{\omega} = (\omega_1, 2\omega_1)^T$,  $x_1, x_2 \in T^1$, $y_1$, $y_2$ $ \in R^1$, $\omega_1 \in R\setminus \{0\}$ and ${M = \left(
                                                  \begin{array}{cc}
                                                    \frac{1}{4} & 0 \\
                                                    0 & \frac{1}{4} \\
                                                  \end{array}
                                                \right)}$. Let $P = \varepsilon^3 \cos(2x_1- x_2)+ \varepsilon^2 \cos(2x_1- x_2)\sin (- x_1) e^{-y_1 - 2y_2}.$ Denote $\tilde{\phi}_g:$ $\left(
                                          \begin{array}{c}
                                            y_1 \\
                                            y_2 \\
                                          \end{array}
                                        \right) = \left(
                                                    \begin{array}{cc}
                                                      -1 & 2 \\
                                                      0 & -1 \\
                                                    \end{array}
                                                  \right)\left(
                                                           \begin{array}{c}
                                                             y \\
                                                             v \\
                                                           \end{array}
                                                         \right)
$, $\left(
      \begin{array}{c}
        x_1 \\
        x_2 \\
      \end{array}
    \right) = \\ \left(
                \begin{array}{cc}
                  -1 & 0 \\
                  -2 & -1 \\
                \end{array}
              \right) \left(
                        \begin{array}{c}
                          x \\
                          u \\
                        \end{array}
                      \right).
$ Obviously, previous works do not apply to this system, since $\tilde{P}(x, u, y, v) = P \circ \tilde{\phi}_g$ is $2-$order nondegenerate perturbation. Actually, under transformation $\tilde{\phi}_g$, $(\ref{EQ40})$ is changed to
\begin{eqnarray*}
H(x,y,u,v) &=& - \omega_1 y + \frac{\varepsilon}{2}\langle \left(
                                                             \begin{array}{c}
                                                               y \\
                                                               v \\
                                                             \end{array}
                                                           \right), \left(
                                                                      \begin{array}{cc}
                                                                        \frac{1}{4} & -\frac{1}{2} \\
                                                                        -\frac{1}{2} & \frac{5}{4} \\
                                                                      \end{array}
                                                                    \right)\left(
                                                             \begin{array}{c}
                                                               y \\
                                                               v \\
                                                             \end{array}
                                                           \right)
\rangle\\
&~& + \varepsilon^3 \cos u + \varepsilon^2~ \cos u ~ \sin x ~e^y,
\end{eqnarray*}
which implies that our Theorem \ref{dingli11} works. Moreover, with our results there are 2 families of resonant torus for system $(\ref{EQ40})$.  For details, refer to Section $\ref{example}.$

\end{remark}

\begin{remark}
Condition $\bf{(S2)}$ ensures the existence of the resonant tori for perturbed system.
\end{remark}

\begin{remark} If $n=m$ and $(K')^T \partial_{I}^2 H_0 K'$ is nondegenerate, condition $\bf{(S3)}$ is $g-$nondegenerate condition mentioned in \cite{Cong,Treshchev}, which ensures the preservation of frequency in the process of KAM iteration. When $n=m$, condition $\bf{(S3)}$ is the condition mentioned in \cite{Li}. Obviously, condition $\bf{(S3)}$ is weaker than all of them if $n<m$. Combining conditions $\bf{(S2)}$ and $\bf{(S3)}$, in the process of KAM iteration, we could show the partial preservation of frequencies, which is determined by $(K_{*}^T \partial_{I}^2 {H_0}K_{*}, K_*^T \partial_{I}^2 H K')$. The details will be shown in Section $\ref{normal form}$.
\end{remark}

\begin{remark}
Under the isoenergetic nondegenerate condition: $$\det\left(
                                                    \begin{array}{cc}
                                                      \partial_{I}^2 H_0 & \partial_{I} H_0 \\
                                                      (\partial_{I} H_0)^T & 0 \\
                                                    \end{array}
                                                  \right)\neq 0,
$$
for Hamiltonian system $(\ref{005})$, Arnold $($\emph{\cite{Arnold1}}$)$ proved that on each energy-level manifold, the invariant tori form majority, which means that the Lebesgue measure of the complement of their union is small and depends on the perturbation. Conditions $\bf{(S2)}$, $\bf{(S3)}$ and $\bf{(S4)}$ are isoenergetic nondegenerate conditions for resonant tori, where $\bf{(S3)}$ and $\bf{(S4)}$ are closely related to the preservation of ratios of frequencies on a given energy-level manifold. As is well-known, the Kolmogorov nondegenerate condition and the classical isoenergetic nondegenerate condition are independent $($\emph{\cite{Sevryuk}}$)$. Our conditions do not violate this fact and reveal a further fact on partial preservation of ratios of frequencies: $\bf{(S3)}$ is also essential for the preservation of energy.
\end{remark}

\begin{remark}
When $n =m$ in condition $\bf{(S3)}$, $\bf{(S2)}$ holds automatically. If $n=m$ in condition $\bf{(S3)}$ and perturbation $P(I, \theta, \varepsilon)$ in Hamiltonian system $(\ref{005})$ is $0$-order nondegenerate, part $\bf{i)}$ of \textbf{Theorem } \emph{\ref{dingli11}} is the result of \emph{\cite{Li}}.
\end{remark}

Finally, we give the following corollary  according to Theorem \ref{dingli11}.
\begin{cor}\label{cor1}
Let Hamiltonian system $(\ref{005})$  with a $\kappa$-order nondegenerate perturbation $P(I, \theta, \varepsilon)$ be real analytic
on the complex neighborhood of $T^d \times G$.  Assume $\bf{(S2)}$, $\bf{(S4)}$ and
 \begin{enumerate}
   \item[{\bf (S5).}] $K_0^T \partial_{I}^2 H_0 K_0$ has a $(m_0+n)\times (m_0+n)$ nonsingular minor, $0\leq n \leq m$, and $\det {K^{'}}^T \partial_{I}^2 {H_0}{K^{'}} \neq 0$.
 \end{enumerate}
Then the conclusions of Theorem \ref{dingli11} also hold.

\end{cor}

\begin{remark}
$\bf{(S5)}$ is equivalent to the following $\bf{(S5')}$$:$
\begin{enumerate}
   \item[$\bf{ (S5').}$] $rank (K_*^T \partial_{I}^2 H K' {K'}^T \partial_{I}^2 H K' {K'}^T\partial_{I}^2 H K_* + K_*^T \partial_{I}^2 HK_* )= n,$ $n < m$, and $\det {K^{'}}^T \partial_{I}^2 {H_0}{K^{'}} \neq 0$ for $I \in \widetilde{\Lambda}(g,G)$,
 \end{enumerate}
 which follows from the following fact$:$
 \begin{eqnarray*}
 \left(
   \begin{array}{cc}
     I_r & 0 \\
     -DB^{-1} & I_{m-r} \\
   \end{array}
 \right) \left(
           \begin{array}{cc}
             B & C \\
             D & E \\
           \end{array}
         \right)
 = \left(
                     \begin{array}{cc}
                       B & C \\
                       0 & -DB^{-1}C + E \\
                     \end{array}
                   \right),
  \end{eqnarray*}
  where $B$ is nonsingular.
\end{remark}

The classical Birkhoff normal form theory provides a formal integrability to harmonic oscillators with perturbation. But it does not work for the persistence of resonant tori studied in present paper, due to the nonlinearity of the unperturbed system and the degeneracy of $[\tilde{P}_0] (y,u,v,0)$. To overcome these difficulties, besides using Tresch\"{e}v's reduction, we propose a quasilinear normal form program by introducing quasilinear KAM iteration, which is used for searching high nondegeneracy and keeping critical points that are related to certain quasiperiodicity of the perturbation. In particular, our KAM iteration is more suitable for problems with worse normal forms. Hence, this approach provides a thorough way to study the persistence of resonant invariant tori under high degenerate perturbations.

The paper is organized as follows. In Section \ref{normal form}, we give an abstract Hamiltonian system and show the persistence of invariant tori. In this section, we introduce modificatory KAM step, which is interesting in itself.  With the results of the abstract Hamiltonian system we finish the proof of Theorem \ref{dingli11} in Section \ref{074}. Finally, in Section $\ref{example}$, we also give two examples to show the complexity resulting from the high degeneracy of the perturbation.

\section{Abstract Hamiltonian systems}\label{normal form}\setcounter{equation}{0}
Throughout the paper, unless specified explanation, we shall use the same symbol $|\cdot|$ to denote an equivalent (finite dimensional) vector norm and its induced matrix norm, absolute value of functions, and measure of sets, etc., and use $|\cdot|_D$ to denote the supremum norm of functions on a domain $D$. Also, for any two complex column vectors $\xi$, $\zeta$ of the same dimension, $\langle \xi, \zeta \rangle$ always stands for $\xi^T \zeta$, i.e., the transpose of $\xi$ times $\zeta$. For the sake of brevity, we shall not specify smoothness orders for functions having obvious orders of smoothness indicated by their derivatives taking. All constants below are positive and independent of the iteration process. Moreover, all Hamiltonian functions in the sequel are associated to the standard symplectic structure.

Let $z=(u, v) \in R ^{2m_0}$. To prove Theorem \ref{dingli11}, consider the following real analytic Hamiltonian system with more general normal form
\begin{eqnarray}\label{model33}
 H(x,y,z,\lambda,\varepsilon) &=& N(y,z, \lambda, \varepsilon) +  \varepsilon^2 P(x,y,z, \lambda, \varepsilon),\\
\nonumber N (y,z,\lambda, \varepsilon)&=& \langle \omega(\lambda), y\rangle + \frac{\varepsilon}{2}\langle \left(
                                                                     \begin{array}{c}
                                                                       y \\
                                                                       z \\
                                                                     \end{array}
                                                                   \right)
, M(\lambda) \left(
               \begin{array}{c}
                 y \\
                 z \\
               \end{array}
             \right)
\rangle + \varepsilon h (y,z,\lambda,\varepsilon),
\end{eqnarray}
defined on
\begin{center}
$D(r,s)=\{(x,y,z):|Im~x|<r,~ |y|<s, ~|z|<s\}$,
\end{center}
where $x \in T^m$, $y \in R^m$, $\lambda \in \Lambda$, $M$, a symmetric matrix, depends smoothly on $\lambda$, $h = O(|\left(
    \begin{array}{c}
      y \\
      z \\
    \end{array}
  \right) |^3
)$ is smooth. Here, $\Lambda$ is a bounded closed region in $R^{m}$. Thorough the paper, all $\lambda-$dependence are of class $C^{l_0}$ for some integer $l_0 \geq d$. Rewrite
\begin{eqnarray*}
M = \left(
      \begin{array}{cc}
        M_{11} & M_{12} \\
        M_{21} & M_{22} \\
      \end{array}
    \right),
\end{eqnarray*}
where $M_{11},$ $M_{12}$, $M_{21},$ $M_{22}$ are $m\times m,$ $m\times 2m_0$, $2m_0\times m,$ $2m_0\times 2 m_0$ matrices, respectively.

\subsection{A General Theorem}

To show the persistence of invariant tori for Hamiltonian (\ref{model33}), assume:

\begin{itemize}
\item[{\bf (A1)}]$rank~ \{\frac{\partial^\alpha \omega}{\partial \lambda^ \alpha}: 0\leq|\alpha| \leq m -1 \}  = m $ for all $\lambda \in {\Lambda}$.
\item[{\bf (A2)}]For given $n$, $0\leq n \leq m,$ $rank (M) = n+ 2m_0$ and $rank (M_{21}, M_{22}) = 2m_0$ for all $\lambda\in \Lambda,$ where $M= (m_{ij})_{(m + 2m_0)\times (m+ 2m_0)}.$
\item[{\bf (A3)}]For given $n$, $0\leq n \leq m,$
\begin{center}
   $ rank \left(
         \begin{array}{cc}
           M(\lambda) & \bar{\omega}_1(\lambda) \\
           {\bar{\omega}_1 }^T(\lambda) & 0 \\
         \end{array}
       \right)
       = n+ 2m_0 +1,$
\end{center}
 where $\bar{\omega}_1 = \left(
                 \begin{array}{c}
                   \omega \\
                   0 \\
                 \end{array}
               \right)\in R^{m+ 2m_0}
$, $\omega\in R^{m}$, $M= (m_{ij})_{(m + 2m_0)\times (m+ 2m_0)}.$
\end{itemize}

\begin{remark}
 We call $\bf{(A2)}$ and $\bf{(A3)}$ sub-isoenergetically nondegenerate conditions for the persistence of lower dimensional invariant tori. Specifically, when $n = m$ and $m_0 = 0$, they are isoenergetically nondegenerate condition introduced by Arnold $(\emph{\cite{Arnold}})$. When $m_0 =0$, they are similar to the isoenergetically nondegenerate condition contained in \emph{\cite{LLL,Sevryuk}}. When $M$ is a block diagonal matrix, refer to $\emph{\cite{Qian1}}$ for a similar condition.
\end{remark}

We state our results for (\ref{model33}) as follows.

\begin{theorem}\label{shengluede}
Let $H(x,y,z,\lambda)$ in $(\ref{model33})$ be real analytic
on the complex neighborhood of $T^d \times G$.
\begin{itemize}
  \item [$\bf{i)}$] Assume $\bf{(A1)}$ and $\bf{(A2)}$ hold on $\Lambda$. Then there exists a $\varepsilon_0 >0 $ and a family of Cantor sets $\Lambda_\varepsilon \subset \Lambda$, $0<\varepsilon < \varepsilon_0$, such that for each $\lambda \in \Lambda_\varepsilon$, system $(\ref{model33})$ admits a family of invariant tori. And $n$ coordinates of the frequency $\breve{\omega}$ on the persistent tori coincide with $n$ coordinates of $\omega$, which are determined by those rows of $(M_{11}, M_{12})$ that are linearly independent. Moreover, the relative Lebesgue measure $|{\Lambda} \setminus {\Lambda}_{\varepsilon}|$ tends to $0$ as $\varepsilon \rightarrow 0$.
  \item [$\bf{ii)}$]Assume $\bf{(A1)}$, $\bf{(A2)}$ and $\bf{(A3)}$ hold on $\tilde{\Xi} = \{\lambda \in \Lambda: N(y,z,\lambda) = c\}$. Then there exists a $\varepsilon_0 >0 $ and a family of Cantor sets $\tilde{\Xi}_\varepsilon \subset \tilde{\Xi}$, $0<\varepsilon < \varepsilon_0$, such that for each $\lambda \in \tilde{\Xi}_\varepsilon$, on a given energy-level manifold, system $(\ref{model33})$ admits a family of invariant tori. And $n$ coordinates of the frequency $\breve{\omega}$ on the persistent tori coincide with $n$ coordinates of $t \omega$, which are determined by those rows of $(M_{11}, M_{12})$ that are linearly independent, where $t \rightarrow1$ as $\varepsilon \rightarrow 0$. Moreover, the relative Lebesgue measure $|\tilde{\Xi} \setminus \tilde{\Xi}_\varepsilon |$ tends to $0$ as $\varepsilon \rightarrow 0$.
\end{itemize}

\end{theorem}

The proof of Theorem \ref{shengluede} will proceed by quasilinear KAM iteration process, which consists of infinitely many KAM steps. Due to the existence of small parameter $\varepsilon$ in term $\frac{\varepsilon}{2} \langle \left(
                                                                                                                          \begin{array}{c}
                                                                                                                            y \\
                                                                                                                            z \\
                                                                                                                          \end{array}
                                                                                                                        \right),
M \left(
    \begin{array}{c}
      y \\
      z \\
    \end{array}
  \right)
\rangle + \varepsilon h(y,z,\lambda,\varepsilon)$, we weaken nondegenerate condition for the persistence of lower dimensional invariant tori, which is interesting in itself. For the case that there is no small parameter in normal direction, refer to \cite{Bourgain,Eliasson,Kuksin,Li1,Mel1,Moser1,Poschel,You}. Next, we show the detail of our KAM steps.

\subsection{KAM step}\label{KAM}

We show first the $0-$th KAM step. For the sake of induction, let
\begin{eqnarray*}
 r_0 = r,~s_0= s,~ \Lambda_0 = \Lambda,~ H_0 = H, ~ N_0 = N, ~P_0 = P,~ M_0 = M, ~h_0 = h,
\end{eqnarray*}
 where $0 <r, s \leq 1$, and denote
\begin{eqnarray*}
 M^* &=&\max\limits_{\substack{|l|\leq l_0, |j|\leq 2, \\ (y, z)\in D(r_0, s_0), \lambda\in \Lambda_0}} |\partial_\lambda^l \partial_{(y,z)}^j h_0(y,z,\lambda)|_{D(r_0, s_0)\times \Lambda_0}.
\end{eqnarray*}
For $j \in Z_+^m$, define
\begin{eqnarray*}
a_j &=& 1- sgn {(|j|-1)} = \left\{
                         \begin{array}{lll}
                           2, & \hbox{$|j| = 0$,} \\
                           1, & \hbox{$|j| = 1$},\\
 0, & \hbox{$|j| \geq 2$.}
                         \end{array}
                       \right.
\end{eqnarray*}
Denote the complex neighborhood of $\Lambda_0$ by $\tilde{\Lambda}_0 = \{\lambda\in \mathds{C}^m, |\lambda - \Lambda_0| \leq \varrho_0\}$ for given constant $\varrho_0$. Let $\varepsilon = \delta$, $\gamma_0 = \varepsilon^{\frac{1 - 3\iota}{3(l_0 +9)}}$, $s_0 = \varepsilon^{\frac{1}{3}}$, $\mu_0 = \varepsilon^\iota$, $\iota \in (0,\frac{1}{3})$ and $\eta_0 = \frac{1}{4} \varrho_0$.  Therefore, by Cauchy's estimate,
\begin{eqnarray}\label{714}
|\partial_\lambda^l P_0|_{D(r_0, s_0) \times {\tilde{\Lambda}}_0} \leq c \frac{\delta\gamma_0^{l_0+9} s_0^{2} \mu_0}{\eta_0^{l_0}}
\end{eqnarray}
for all $l\in Z_+^m$, $|l| \leq l_0,$ where $c>0$ is a constant.

Next we characterize the iteration scheme for Hamiltonian (\ref{model33}) in one KAM step, say, from the $\nu-$th KAM step to the $(\nu + 1)-$th step. Recall $M_\nu = \left(
              \begin{array}{cc}
                M_{11,\nu} & M_{12,\nu} \\
                M_{21,\nu} & M_{22,\nu} \\
              \end{array}
            \right)
$, for given $k \in Z^m$, denote
\begin{eqnarray*}
 \breve{L}_{k0, \nu} &=& \sqrt{-1}\langle k, \omega_\nu\rangle,\\
 \breve{L}_{k1, \nu} &=& \left(
                   \begin{array}{cc}
                     \breve{L}_{k0,\nu} I_m & - \delta M_{12, \nu} J \\
                     0 & \breve{L}_{k0, \nu} I_{2m_0} - \delta M_{22, \nu} J \\
                   \end{array}
                 \right),\\
 \breve{L}_{k2, \nu} &=& \left(
             \begin{array}{ccc}
               I_m \otimes \breve{L}_{k0, \nu} I_m & (\delta J M_{21, \nu})^T \otimes I_m  & 0 \\
               0 & \breve{a}_{22, \nu}  & - I_{2m_0} \otimes (2\delta M_{12,\nu}J) \\
               0 & 0 & \breve{a}_{33, \nu} \\
             \end{array}
           \right),
 \end{eqnarray*}
 where $\breve{a}_{22, \nu} = I_{2m_0}\otimes \breve{L}_{k0, \nu} I_m - (\delta M_{22, \nu} J) \otimes I_m,$ $\breve{a}_{33, \nu} = \breve{L}_{k0, \nu} I_{4m_0^2} - (\delta M_{22, \nu} J)\otimes I_{2m_0} - I_{2m_0}\otimes (\delta M_{22, \nu} J).$ For given matrix $A$, $A^*$ represents conjugate transpose of $A$. Let
  \begin{eqnarray*}
  \Lambda_\nu &=& \{\lambda\in \Lambda_{\nu-1}: |\breve{L}_{k0, \nu}| > \frac{\gamma_\nu}{|k|^\tau}, \breve{L}_{k1, \nu}^* \breve{L}_{k1, \nu} > \frac{\gamma_\nu}{|k|^\tau}  I_{m + 2m_0},  \\
&~&\breve{L}_{k2, \nu}^* \breve{L}_{k2, \nu} > \frac{\gamma_\nu}{|k|^\tau}  I_{m^2 + 2m m_0+ 4m_0^2},  ~for~all~0< |k|\leq K_\nu\},\\
\tilde{\Lambda}_\nu &=& \{\lambda\in \mathds{C}^m, |\lambda - \Lambda_\nu| \leq 4\eta_\nu\}, ~~\eta_\nu = \mu_{\nu-1}^{\frac{1}{6 l_0^2}}.
 \end{eqnarray*}

 Now, suppose that after $\nu$ KAM steps, we have arrived at the following real analytic Hamiltonian system
\begin{eqnarray}\label{N2}
 H_\nu(x,y,z) &=&N_\nu(y,z) +  P_\nu(x,y,z,\varepsilon),\\
\nonumber N_\nu(y,z) &=& \langle \omega_\nu(\lambda), y\rangle + \frac{\delta}{2}\langle \left(
                                                                     \begin{array}{c}
                                                                       y \\
                                                                       z \\
                                                                     \end{array}
                                                                   \right)
, M_\nu(\lambda) \left(
               \begin{array}{c}
                 y \\
                 z \\
               \end{array}
             \right)
\rangle + \delta h_\nu(y,z,\lambda, \varepsilon),\\
\nonumber |\partial_\lambda^l  P_\nu|_{D(r_\nu, s_\nu) \times\tilde{\Lambda}_\nu} &\leq&\frac{\delta \gamma_\nu^{l_0+9} s_\nu^{2} \mu_\nu}{\eta_\nu^{l_0}}, ~|l| \leq l_0,
\end{eqnarray}
where $M_\nu (\lambda)= (m_{ij})_{(m + 2m_0)\times (m+ 2m_0)}$ satisfies that $rank (M_\nu) = n+ 2m_0$ and $rank (M_{21,\nu}, M_{22, \nu}) = 2m_0$ for positive integer $n\in [0, m]$ and $\lambda\in \Lambda_\nu$, $h_\nu=O(|\left(
    \begin{array}{c}
      y \\
      z \\
    \end{array}
  \right) |^3
).$ For convenience, we shall omit the index for all quantities of the $\nu-$th KAM step and use $'+'$ to index all quantities in the ${(\nu+ 1)}-$th KAM step. To simplify the notions, we shall suspend the $\lambda-$dependence in most terms of this section.  By considering both averaging and translation, we shall find a symplectic transformation $\Phi _ {+}$,
which, on a small phase domain $D(r_+, s_+)$  and a smaller parameter domain $\Lambda_+$, transforms Hamiltonian (\ref{N2}) into the following form:
 \begin{center}
  $H_{+} = H{\circ} {\Phi _+}= {N_+} +{P_+}$,
 \end{center}
where on $D(r_+, s_+)\times \tilde{\Lambda}_+,$ $N_+$ and $P_+$ enjoy similar properties as $N$ and $P$, respectively.

Define
\begin{eqnarray*}
s_+ &=& \frac{1}{8} \alpha s,~~~~~\mu_+ = 64 c_0 \mu^{\frac{13}{12}}, ~~~~~r_+ = r - \frac{r_0}{2^{\nu+1}}, ~~~~~\gamma_+ = \gamma - \frac{\gamma_0}{2^{\nu+1}},\\
\eta_+ &=& \mu^{\frac{1}{6 l_0^2}}, ~~~K_+ = ([\log\frac{1}{\mu}]+ 1)^{3{\eta}},~~~\Gamma(r- r_+) = \sum_{0<|k|\leq K_+} |k|^{\chi} e^{-|k|\frac{r -r_+}{8}},\\
\Lambda_+ &=& \{\lambda\in \Lambda: |\breve{L}_{k0}| > \frac{\gamma}{|k|^\tau}, \breve{L}_{k1}^* \breve{L}_{k1} > \frac{\gamma}{|k|^\tau}  I_{m + 2m_0},  \\
&~&\breve{L}_{k2}^* \breve{L}_{k2} > \frac{\gamma}{|k|^\tau}  I_{m^2 + 2m m_0+ 4m_0^2},  ~for~all~0< |k|\leq K_+\},\\
\tilde{\Lambda}_+ &=& \{\lambda\in \mathds{C}^m, |\lambda - \Lambda_+| \leq 4\eta_+\},\\
\hat{D}(\lambda) &=& D(r_++ \frac{7}{8}(r - r_+), \lambda), ~~~~~D(\lambda) = \{y\in \mathds{C}^n:|y|< \lambda\},\\
D_{\frac{i}{8} \alpha} &=&  D(r_+ + \frac{i -1 }{8}(r - r_+), \frac{i}{8}\alpha s), ~~i = 1,2,\cdots, 8,
\end{eqnarray*}
where $\alpha = \mu^{\frac{1}{3}},$ $\chi =3\chi_1 = 3 (m^2 + 2m m_0+ 4m_0^2)\big((l_0+5)\tau + 5l_0+ 10 +m^2 + 2m m_0+ 4m_0^2\big),$ $c_0$ is the maximal among all $c's$ mentioned in this paper and depends on $r_0$, $\beta_0$.

\subsubsection{Truncation of the perturbation}
Consider the Taylor-Fourier series of $P$:
\begin{eqnarray*}
P  = \sum_{i\in Z_+^m,j\in Z_+^{2m_0}, k \in Z^m} p_{kij} y^{i}z^{j} e^{\sqrt{-1}\langle k, x\rangle},
\end{eqnarray*}
and let $R$ be the truncation of $P$ with the following form:
\begin{eqnarray*}
R &=& \sum\limits_{|k| \leq K_+} (p_{k00}+ \langle p_{k10},y\rangle+\langle p_{k01},z\rangle + \langle y, p_{k20}y\rangle+ \langle y, p_{k11}z\rangle\\
 &~&+ \langle z, p_{k02}z\rangle) e^{\sqrt{-1} \langle k, x\rangle},
\end{eqnarray*}
where $K_+$ is defined as above.

\begin{lemma}
Assume that
\begin{itemize}
\item[\bf{(H1)}] $K_+ \geq \frac{8(m + l_0)}{r - r_+}$,
\item[\bf{(H2)}] $\int_{K_+}^\infty x^{m + l_0} e^{-x \frac{r - r_+}{8}} d x \leq \mu.$
\end{itemize}
Then there is a constant $c$ such that for all $|l| \leq l_0$, $\lambda \in \Lambda$,
\begin{eqnarray*}
|\partial_\lambda^l  (P- R)|_{D_\alpha \times {\tilde{\Lambda}}} \leq c \frac{\delta\gamma^{l_0 + 9} s^{2} \mu^{2}}{\eta^{l_0}}.
\end{eqnarray*}
\end{lemma}

\begin{proof}
The proof is standard. For detail, refer to, for example, $\bf{Lemma ~3.1}$ of \cite{Li}.

\end{proof}

\subsubsection{Homological equations}\label{Homological_equations}
We want to average out all coefficients of $R$ by constructing a symplectic transformation as the time-1 map $\phi_F^1$ of the flow generated by a Hamiltonian $F$ with the following form:
\begin{eqnarray}
\nonumber F &=& \sum\limits_{0< |k| \leq K_+} \big(f_{k00}+ \langle f_{k10},y\rangle+\langle f_{k01},z\rangle + \langle y, f_{k20}y\rangle+ \langle y, f_{k11}z\rangle\\
\label{709} &~&~~~~~~~~~ + \langle z, f_{k02}z\rangle\big) e^{\sqrt{-1} \langle k, x\rangle},
\end{eqnarray}
where $f_{kij}$, $0\leq |i|+|j|\leq 2$, are scalar, vectors or matrices with obvious dimensions, which are allowed to depend on $y$, $z$ and $\lambda$. Under the time-1 map $\phi_F^1$, Hamiltonian (\ref{N2}) becomes
\begin{eqnarray}
\nonumber H\circ \phi_F^1 &=& (N +  R )\circ \phi_F^1 +  (P - R)\circ \phi_F^1\\
\label{N6} &=&N +  R+ \{N, F\} +\int_0 ^1 \{R_t,F\}\circ \phi_F^t dt + (P - R )\circ \phi_F^1,~~~~
\end{eqnarray}
where $ R_t = (1-t) \{N, F\} + R.$ Let
\begin{eqnarray}\label{706}
\{N, F\} +  R - [R] - R'= 0,
\end{eqnarray}
where
\begin{eqnarray}
\nonumber [R] &=& \int_{T^n} R(x, \cdot) dx,\\
\nonumber R' &=& \partial_z {h} J \partial_z F+ \langle y, M_{12}J \Delta_0 \rangle+ \langle z, M_{22}J \Delta_0 \rangle,\\
\nonumber \Delta_0 &=& \langle y, \partial_z f_{k20} y\rangle + \langle y, \partial_z f_{k11} z\rangle+ \langle z, \partial_z f_{k02}z\rangle,\\
\label{EQ30}\hat{h} &=& \frac{\delta}{2} \langle \left(
                                 \begin{array}{c}
                                   y \\
                                   z \\
                                 \end{array}
                               \right), M \left(
                                            \begin{array}{c}
                                              y \\
                                              z \\
                                            \end{array}
                                          \right)
\rangle + \delta h(y,z, \lambda, \varepsilon).
\end{eqnarray}
Then Hamiltonian $(\ref{N6})$ arrives at
\begin{eqnarray}\label{N4}
\bar{H}_+ = \bar{N}_+(y,z) + \bar{P}_+(x,y,z),
\end{eqnarray}
where $\bar{N}_+ = N +  [R],$ $\bar{P}_+ = {R'} + \int_0 ^1 \{R_t,F\}\circ \phi_F^t dt + (P - R )\circ \phi_F^1.$

Consider the following symplectic translation:
\begin{eqnarray}\label{E_21}
\phi: x \rightarrow x, \left(
                         \begin{array}{c}
                           y \\
                           z \\
                         \end{array}
                       \right)
  \rightarrow \left(
                       \begin{array}{c}
                         y + y_0 \\
                         z+ z_0 \\
                       \end{array}
                     \right),
\end{eqnarray}
where $(y_0, z_0)$ is determined by
\begin{eqnarray}\label{E_1}
\delta\frac{ M}{2} \left(
  \begin{array}{c}
    y_0 \\
    z_0 \\
  \end{array}
\right) +  \delta \left(
                    \begin{array}{c}
                      \partial_y h(y_0, z_0, \lambda) \\
                      \partial_z h(y_0, z_0, \lambda) \\
                    \end{array}
                  \right)
 = - \left(
                         \begin{array}{c}
                           p_{010} \\
                           p_{001} \\
                         \end{array}
                       \right).
\end{eqnarray}
Then Hamiltonian system $(\ref{N4})$ is changed to
\begin{eqnarray*}
H_+ &=& \bar{H}_+ \circ \phi\\
    &=& e_+ + \langle \omega_+, y\rangle +\frac{\delta}{2} \langle\left(
                                           \begin{array}{c}
                                             y \\
                                             z \\
                                           \end{array}
                                         \right), M_+ \left(
                                           \begin{array}{c}
                                             y \\
                                             z \\
                                           \end{array}
                                         \right)\rangle + \delta h_+(y,z,\lambda,\varepsilon)+ P_+,
\end{eqnarray*}
where
\begin{eqnarray}
\nonumber e_+ &=& e + \langle \omega, y_0\rangle+ \frac{\delta}{2} \langle \left(
                                                                        \begin{array}{c}
                                                                          y_0 \\
                                                                          z_0 \\
                                                                        \end{array}
                                                                      \right), M  \left(
                                                                        \begin{array}{c}
                                                                          y_0 \\
                                                                          z_0 \\
                                                                        \end{array}
                                                                      \right)
\rangle +  p_{000} +  \langle\left(
                                              \begin{array}{c}
                                                p_{010} \\
                                                p_{001} \\
                                              \end{array}
                                            \right), \left(
                                                       \begin{array}{c}
                                                         y_0 \\
                                                         z_0 \\
                                                       \end{array}
                                                     \right)
                                            \rangle\\
\nonumber&~&~~~+ \langle \left(
                                                                                      \begin{array}{c}
                                                                                        y_0 \\
                                                                                        z_0 \\
                                                                                      \end{array}
                                                                                    \right), \left(
                                                                                               \begin{array}{cc}
                                                                                                 p_{010} & \frac{1}{2} p_{011} \\
                                                                                                 \frac{1}{2} p_{011}^T  & p_{002} \\
                                                                                               \end{array}
                                                                                             \right)
                                           \left(
                                             \begin{array}{c}
                                               y_0 \\
                                               z_0 \\
                                             \end{array}
                                           \right)
                                            \rangle + \delta h(y_0, z_0,\lambda),\\
\nonumber \omega_+ &=& \omega + \frac{\delta M }{2} \left(
                                                      \begin{array}{c}
                                                        y_0 \\
                                                        z_0 \\
                                                      \end{array}
                                                    \right)+ \delta \left(
                                                                      \begin{array}{c}
                                                                        \partial_y h(y_0, z_0, \lambda) \\
                                                                        \partial_z h(y_0, z_0, \lambda) \\
                                                                      \end{array}
                                                                    \right)+ \left(
                                                                               \begin{array}{c}
                                                                                 p_{010} \\
                                                                                 p_{001} \\
                                                                               \end{array}
                                                                             \right)
                                                                    ,\\
\label{EQ31} M_+&=& M +2 \left(
             \begin{array}{cc}
               p_{020} & \frac{1}{2} p_{011} \\
                \frac{1}{2} p_{011}^T & p_{002}\\
             \end{array}
           \right)+ \partial_{(y,z)}^2 h(y_0, z_0, \lambda),\\
\label{N5} P_+ &=& \bar{P}_+ + \delta \langle \left(
                                                                                      \begin{array}{c}
                                                                                        y \\
                                                                                        z \\
                                                                                      \end{array}
                                                                                    \right), \left(
                                                                                               \begin{array}{cc}
                                                                                                 p_{010} & \frac{1}{2} p_{011} \\
                                                                                                 \frac{1}{2} p_{011}^T  & p_{002} \\
                                                                                               \end{array}
                                                                                             \right)
                                           \left(
                                             \begin{array}{c}
                                               y_0 \\
                                               z_0 \\
                                             \end{array}
                                           \right)
                                            \rangle,\\
\nonumber h_+&=& h(y,z,\lambda)- h(y_0, z_0, \lambda) - \langle \left(
                                                                  \begin{array}{c}
                                                                    \partial_y h(y_0, z_0, \lambda) \\
                                                                     \partial_z h(y_0, z_0, \lambda) \\
                                                                  \end{array}
                                                                \right), \left(
                                                                           \begin{array}{c}
                                                                             y \\
                                                                             z \\
                                                                           \end{array}
                                                                         \right)
\rangle\\
\label{EQ32}&~&~~ - \frac{1}{2} \langle \left(
                                     \begin{array}{c}
                                       y \\
                                       z \\
                                     \end{array}
                                   \right), \partial_{(y,z)}^2h (y_0, z_0, \lambda) \left(
                                                                                      \begin{array}{c}
                                                                                        y \\
                                                                                        z \\
                                                                                      \end{array}
                                                                                    \right)
\rangle.
\end{eqnarray}

\subsubsection{Estimate on the transformation} \label{Estimate}
According to the definition of Poisson bracket on coordinate $(x,y,z)\in T^m \times R^m\times R^{2m_0}$,
\begin{eqnarray*}
\{N, F\} &=& \partial_x N \partial_y F - \partial_y N \partial_x F + \partial_z N J \partial_z F\\
&=&- \partial_y N \partial_x F + \partial_z \hat{h} J \partial_z F,
\end{eqnarray*}
where $J = \left(
             \begin{array}{cc}
               0 & I_{m_0\times m_0} \\
               -I_{m_0\times m_0} & 0 \\
             \end{array}
           \right).
$  Then $(\ref{706})$ is changed to
\begin{eqnarray}\label{Eq14}
- \partial_y N \partial_x F+ \partial_z \hat{h} J \partial_z F + R -[R] = 0.
\end{eqnarray}
Denote $\Delta_1 = \partial_y \hat{h} = \delta (M_{11} y + M_{12} z+ \partial_y h(y,z,\lambda)).$ Directly,
\begin{eqnarray}
\nonumber \partial_y N \partial_x F &=& \sum\limits_{0< |k| \leq K_+} \sqrt{-1}\langle k, \omega +\Delta_1 \rangle \big(f_{k00}+ \langle f_{k10}, y\rangle + \langle f_{k01}, z\rangle \\
\label{EQ1}&~& + \langle y, f_{k20} y\rangle+ \langle y, f_{k11}z\rangle + \langle z, f_{k02}z\rangle\big)e^{\sqrt{-1}\langle k, x\rangle},\\
\nonumber R -[R] &=&  \sum\limits_{0< |k| \leq K_+} \big(p_{k00}+ \langle p_{k10},y\rangle+\langle p_{k01},z\rangle \\
\label{EQ2q}&~&+ \langle y, p_{k20}y\rangle+ \langle y, p_{k11}z\rangle + \langle z, p_{k02}z\rangle\big) e^{\sqrt{-1} \langle k, x\rangle}.
\end{eqnarray}
Substituting $(\ref{EQ1})$ and $(\ref{EQ2q})$ into $(\ref{Eq14})$ yields:
\begin{eqnarray}
\label{5Eq1}\sqrt{-1} \langle k, \omega+ \Delta_1\rangle f_{k00} =  p_{k00},~~~~~~~\\
\label{5Eq2}\sqrt{-1} \langle k, \omega+ \Delta_1\rangle f_{k10} - \delta M_{12} J f_{k01} =  p_{k10} + \delta M_{12} J \partial_z f_{k00},~~~~~~~\\
\label{5Eq3}\sqrt{-1} \langle k, \omega+ \Delta_1\rangle f_{k01} - \delta M_{22} J f_{k01} =  p_{k01} + \delta M_{22} J \partial_z f_{k00},~~~~~~~\\
\label{5Eq4}\sqrt{-1} \langle k, \omega+ \Delta_1\rangle f_{k20} + \delta f_{k11} J M_{21} =  p_{k20} + \delta M_{12} J \partial_z (f_{k10})^T,~~~~~~~\\
\nonumber \sqrt{-1} \langle k, \omega+ \Delta_1\rangle f_{k11} - 2\delta M_{12} J f_{k02} - \delta f_{k11} (M_{22}J)^T ~~~~~~~ \\
\label{5Eq5} ~=  p_{k11}+ \delta M_{12} J \partial_z (f_{k01})^T + (\delta M_{22} J \partial_z (f_{k10})^T)^T,~~~~~~~\\
\nonumber \sqrt{-1} \langle k, \omega+ \Delta_1\rangle f_{k02} - \delta M_{22} J f_{k02} + \delta f_{k02} J M_{22} ~~~~~~~\\
\label{5Eq6}=  p_{k02} + \delta M_{22} J \partial_z (f_{k10})^T.~~~~~~~
\end{eqnarray}
For any matrix $A = (a_{ij})_{p\times q}$, denote $T(A) = (a_{11}, \cdots, a_{p1}, \cdots, a_{1q}, \cdots, a_{pq})^T.$ Let
\begin{eqnarray*}
L_{k0} &=& \sqrt{-1} \langle k, \omega + \Delta_1\rangle,\\
L_{k1} &=& \left(
             \begin{array}{cc}
               L_{k0} I_m & - \delta M_{12} J \\
               0 & L_{k0} I_{2m_0} - \delta M_{22} J  \\
             \end{array}
           \right),\\
L_{k2} &=& \left(
             \begin{array}{ccc}
               I_m \otimes (L_{k0} I_m) & (\delta J M_{21})^T \otimes I_m  & 0 \\
               0 & a_{22}  & - I_{2m_0} \otimes (2\delta M_{12}J) \\
               0 & 0 & a_{33} \\
             \end{array}
           \right),\\
a_{22} &=&I_{2m_0} \otimes (L_{k0} I_m) - (\delta M_{22} J) \otimes I_m,\\
a_{33} &=& L_{k0} I_{4m_0^2} - (\delta M_{22} J)\otimes I_{2m_0} - I_{2m_0}\otimes (\delta M_{22} J).
\end{eqnarray*}
Rewrite $(\ref{5Eq1})$ $-$ $(\ref{5Eq6})$ as follows
\begin{eqnarray}
\label{5Eq7} L_{k0} f_{k00} &=& p_{k00},\\
\label{5Eq8} L_{k1} \left(
         \begin{array}{c}
           f_{k10} \\
           f_{k01} \\
         \end{array}
       \right) &=& \left(
         \begin{array}{c}
           p_{k10} \\
           p_{k01} \\
         \end{array}
       \right) +\delta \left(
         \begin{array}{c}
            M_{12} J \partial_z f_{k00} \\
            M_{22} J \partial_z f_{k00} \\
         \end{array}
       \right),\\
\label{5Eq9} L_{k2} \left(
         \begin{array}{c}
           T (f_{k20})\\
           T (f_{k11})\\
           T(f_{k02}) \\
         \end{array}
       \right) &=& \left(
                   \begin{array}{c}
                     T(p_{k20}) \\
                     T(p_{k11}) \\
                     T(p_{k02}) \\
                   \end{array}
                 \right) + \delta \left(
                               \begin{array}{c}
                                 T(M_{12} J \partial_z (f_{k10})^T) \\
                                 T(\check{M}) \\
                                 T(M_{22} J \partial_z(f_{k01})^T) \\
                               \end{array}
                             \right),~~~~
\end{eqnarray}
where $\check{M} = M_{12} J \partial_z (f_{k01})^T + (M_{22} J \partial_z (f_{k00})^T)^T$.

On pages 17-19,

\begin{lemma}
Assume that
\begin{itemize}
\item[\bf{(H3)}] $\max\limits_{ |j| \leq 2} | \partial_{(y,z)}^j \hat{h}(y,z,\lambda) -  \partial_{(y,z)}^j \hat{h}_0(y,z,\lambda)|_{D(r,s)\times{\Lambda}}\leq \mu_0^{\frac{1}{2}}.$
\end{itemize}
Then there is a constant c such that for all $|l| \leq l_0$,
\begin{eqnarray}
\label{E_18}|\partial_\lambda^l e_+ - \partial_\lambda^l e|_{D(r,s) \times\tilde{\Lambda}} &\leq& c \frac{\gamma^{l_0+9} s \mu}{\eta^{l_0}},\\
\label{E_19}|\partial_\lambda^l M_+ - \partial_\lambda^l M|_{D(r,s) \times\tilde{\Lambda}} &\leq& c \frac{\gamma^{l_0 + 9} \mu}{\eta^{l_0}},\\
\label{E_21}|\partial_\lambda^l \omega_+ - \partial_\lambda^l \omega|_{D(r,s) \times\tilde{\Lambda}} &\leq& c\frac{\delta s (\gamma^{l_0+9} \mu+ s)}{\eta^{l_0}},\\
\label{E_20}|\left(
   \begin{array}{c}
     \partial_\lambda^l y_0 \\
     \partial_\lambda^l z_0 \\
   \end{array}
 \right)
|_{D(r,s) \times\tilde{\Lambda}} &\leq& c \frac{\gamma^{l_0+9} s \mu}{\eta^{l_0}}.
\end{eqnarray}
\end{lemma}
\begin{proof}
Obviously, $|\partial _ \lambda^l p_{000}|_{\tilde{\Lambda}} \leq c  \frac{\delta\gamma^{l_0+9} s^2 \mu}{\eta^{l_0}},$ $|\partial _ \lambda^l p_{010}|_{\tilde{\Lambda}} + |\partial _ \lambda^l p_{001}|_{\tilde{\Lambda}}\leq c \frac{\delta \gamma^{l_0+9} s \mu}{\eta^{l_0}},$ $|\left(
   \begin{array}{cc}
     \partial_{\lambda}^l p_{020} &  \partial_{\lambda}^l p_{011} \\
      \partial_{\lambda}^l p_{011}^T &  \partial_{\lambda}^l p_{002} \\
   \end{array}
 \right)
|_{\tilde{\Lambda}} \leq c\frac{\delta \gamma^{l_0+9} \mu}{\eta^{l_0}}.$
Denote
\begin{eqnarray*}
B= \frac{M}{2}+ \left(
        \begin{array}{cc}
          \int_0^1 \partial_y^2 h({\theta} y,z,\lambda) d{\theta} & \int_0^1 \partial_y\partial_z h( y,{\theta} z,\lambda) d{\theta} \\
          \int_0^1\partial_z \partial_y h({\theta} y,z,\lambda) d{\theta} & \int_0^1 \partial_z^2 h( y,{\theta} z,\lambda) d{\theta} \\
        \end{array}
      \right).
\end{eqnarray*}
Then $(\ref{E_1})$ becomes
\begin{eqnarray}\label{E_17}
\delta B \left(
    \begin{array}{c}
      y \\
      z \\
    \end{array}
  \right) = - \left(
              \begin{array}{c}
                p_{010} \\
                p_{001} \\
              \end{array}
            \right).
\end{eqnarray}
For given matrix $A =(a_{ij})_{n\times n}$, let $||A||_1 = \frac{1}{n}\sum\limits_{i,j=1}^n |a_{ij}(\lambda)|,$ where $|a_{ij}(\lambda)|$ is the absolute value of $a_{ij}(\lambda)$, $\lambda\in \Lambda$.
According to assumption $\bf{(H3)}$ and the definition of $M^*$, we have $||M - M_0||_1 \leq \mu_0^{\frac{1}{2}},$ $||\partial_{(y,z)}^2 h||_1 \leq (M^*+1 )s,$ where $M_0$ is $M_\nu$ for $\nu= 0$. Denote $M_* = ||M_0^{-1}||_1$ for $\lambda\in \Lambda$. Without loss of generality, let $\mu_0$ and $s_0$ be small enough such that $s_0^{\frac{1}{2}} M_* (M^* +1 ) \leq\frac{1}{4}$ and $ \mu_0 M_* \leq\frac{1}{4}. $
Then
\begin{eqnarray*}
||M_0 - B||_1 &\leq& ||M- M_0||_1 + ||B - M||_1\\
&\leq& \mu_0^{\frac{1}{2}} + (M^*+1 )s^2\\
&\leq& \frac{1}{2M_*}.
\end{eqnarray*}
Let $M_0$ be nonsingular. It follows that $B$ is nonsingular and
\begin{eqnarray*}
||B^{-1} ||_1 &=& ||\frac{M_0^{-1}}{I - (M_0 - B) M_0^{-1}}||_1\\
&\leq& \frac{||M_0^{-1}||_1}{||I - (M_0 - B) M_0^{-1}||_1}\\
&\leq& \frac{||M_0^{-1}||_1}{1 - ||(M_0 - B) M_0^{-1}||_1}\\
&\leq& \frac{||M_0^{-1}||_1}{1 - ||(M_0 - B)||_1|| M_0^{-1}||_1}\\
&\leq& \frac{M_*}{1- \frac{1}{2M_*}M_*}\\
&=&2M_*.
\end{eqnarray*}
Here, we use the fact that $||(I - A)^{-1}||_1 \leq \frac{1}{1 - ||A||_1},$ which is obvious if $||I||_1 = 1$ and $||A||_1 <1.$
 Therefore,
 \begin{eqnarray*}
 |\left(
   \begin{array}{c}
     y \\
     z \\
   \end{array}
 \right)|_{D(r,s) \times \tilde{\Lambda}} &=& |\frac{1}{\delta} B^{-1} \left(
                                       \begin{array}{c}
                                         p_{010} \\
                                         p_{001} \\
                                       \end{array}
                                     \right)
 |\\
 &\leq& \frac{m+ 2m_0}{\delta} ||B^{-1}||_1 |\left(
                                         \begin{array}{c}
                                           p_{010} \\
                                           p_{001} \\
                                         \end{array}
                                       \right)
 |\\
 &\leq& c \gamma^{l_0 +9} s \mu.
 \end{eqnarray*}
 Consider the differential with respect to $\lambda$ on both sides of $(\ref{E_17})$
 \begin{eqnarray*}
 \partial_{(y,z)} B  \left(
                                     \begin{array}{c}
                                       \partial_\lambda y \\
                                       \partial_\lambda z \\
                                     \end{array}
                                   \right) \left(
                                             \begin{array}{c}
                                               y \\
                                               z \\
                                             \end{array}
                                           \right)
  + \partial_\lambda B \left(
                                            \begin{array}{c}
                                              y \\
                                              z \\
                                            \end{array}
                                          \right) + B \left(
                                                                       \begin{array}{c}
                                                                        \partial_\lambda  y \\
                                                                        \partial_\lambda  z \\
                                                                       \end{array}
                                                                     \right)
                                            = - \left(
                                                                                                  \begin{array}{c}
                                                                                                    \partial_\lambda P_{010} \\
                                                                                                    \partial_\lambda P_{001} \\
                                                                                                  \end{array}
                                                                                                \right).
 \end{eqnarray*}
 Then
 \begin{eqnarray*}
 |\left(
                                                                       \begin{array}{c}
                                                                        \partial_\lambda  y \\
                                                                        \partial_\lambda  z \\
                                                                       \end{array}
                                                                     \right)|_{D(r,s) \times {\tilde{\Lambda}}} &=& |B^{-1} ( \left(
                                                                                                  \begin{array}{c}
                                                                                                    \partial_\lambda P_{010} \\
                                                                                                    \partial_\lambda P_{001} \\
                                                                                                  \end{array}
                                                                                                \right) + \partial_{(y,z)} B  \left(
                                     \begin{array}{c}
                                       \partial_\lambda y \\
                                       \partial_\lambda z \\
                                     \end{array}
                                   \right)\left(
                                            \begin{array}{c}
                                              y \\
                                              z \\
                                            \end{array}
                                          \right)\\
&&  + \partial_\lambda B \left(
                                            \begin{array}{c}
                                              y \\
                                              z \\
                                            \end{array}
                                          \right))|\\
  &\leq& 2M_* \frac{\gamma^{l_0+9} s \mu}{\eta} + 4M_*^2 (M^* +1) \frac{\gamma^{l_0+9} s \mu}{\eta} |\left(
                                                                        \begin{array}{c}
                                       \partial_\lambda y \\
                                       \partial_\lambda z \\
                                     \end{array}
                                                                     \right)|\\
  && + 4M_*^2 (M^* +1)\frac{ \gamma^{l_0+9} s\mu}{\eta}\\
 &\leq& c \frac{\gamma^{l_0+9} s\mu}{\eta}.
 \end{eqnarray*}
 Inductively, we get $(\ref{E_20})$. According to the definition of $e_+$, $\omega_+$ and $M_+$, $(\ref{E_18})$, $(\ref{E_19})$ and $(\ref{E_21})$ are obvious.

\end{proof}

Recall $ \chi_1 = (m^2 + 2m m_0+ 4m_0^2)\big((l_0+5)\tau + 5l_0+ 10 +m^2 + 2m m_0+ 4m_0^2\big).$
\begin{lemma}\label{Lema1}
Assume that
\begin{itemize}
\item[\bf{(H4)}]$\max\{s, \mu^{\frac{1}{6l_0^2}}\} K_+ ^{\chi_1}= o(\gamma)$.
\end{itemize}
The following hold for all $0< |k| \leq K_+$.

\begin{itemize}
\item[\bf{(1)}] On $D(s)\times {\tilde{\Lambda}}_+$, for $|l|\leq l_0,$ $|\imath|+ |\jmath|\leq 2$,
\begin{eqnarray*}
|\partial_\lambda^l \partial_y^\imath \partial_z^\jmath f_{kij}|_{D(s)\times {\tilde{\Lambda}}_+} &\leq&  c \frac{\delta|k|^{3\chi_1 } s^{2-|i| - |j|}\mu e ^{- |k|r}}{\eta_+^{l_0}};
\end{eqnarray*}
\item[\bf{(2)}] On $\hat{D}(s) \times {\tilde{\Lambda}}_+$,
\begin{eqnarray*}
|\partial_\lambda^l \partial_x^i \partial_{(y,z)}^j F|_{\hat{D}(s) \times {\tilde{\Lambda}}_+} \leq c\frac{\delta s^{a_j} \mu \Gamma(r -r_+)}{\eta_+^{l_0}},~~ |i|< l_0, ~ |j| \leq 2, |l|< l_0.
\end{eqnarray*}
\end{itemize}
\end{lemma}

\begin{proof}
Denote $\omega = \omega(\lambda)$ for $\lambda\in \Lambda$ and $\omega_0 = \omega (\lambda) $ for $\lambda\in \tilde{\Lambda}$. Recall $\eta_+ = \mu^{\frac{1}{6 l_0^2}}.$ For any $\lambda \in \tilde{\Lambda}_+$, $0<|k|\leq K_+$, with assumption $\bf{(H4)}$ we have
\begin{eqnarray}
\nonumber | L_{k0}|_{D(s)\times\tilde{\Lambda}_+} &=& |\sqrt{-1}\langle k, \omega\rangle + \sqrt{-1}\langle k, \Delta\rangle + \sqrt{-1}\langle k, \omega- \omega_0\rangle|\\
\nonumber &\geq& \frac{\gamma}{|k|^\tau} - c \max \{s, \mu^{\frac{1}{6l_0^2}}\} \delta K_+\\
\label{Eq1}& \geq& \frac{\gamma}{2 |k|^\tau},
\end{eqnarray}
and $|\partial_\lambda^l \partial_y^\imath \partial_z^\jmath L_{k0}|_{D(s)\times\tilde{\Lambda}_+} \leq c |k|. $ Applying the above and the following inequalities
\begin{eqnarray*}
|\partial^{l} L_{k0}^{-1}| \leq |L_{k0}^{-1}| \sum_{|l'| = 1}^{|l|} \left(
                                                                       \begin{array}{c}
                                                                         l \\
                                                                          l'\\
                                                                       \end{array}
                                                                     \right)
|\partial^{l - l'} L_{k0}^{-1}| |\partial^{l'} L_{k0}|,
\end{eqnarray*}
inductively, we deduce that
\begin{eqnarray}
\nonumber |\partial_\lambda^l \partial_y^\imath \partial_z^\jmath  L_{k0}^{-1}|_{D(s)\times\tilde{\Lambda}_+} &\leq& c |k|^{|l|+ |\imath|+ |\jmath|} |L_{k0}^{-1}|^{|l|+ |\imath|+ |\jmath|+1}\\
\label{N3} &\leq& \frac{|k|^{(|l|+ |\imath|+ |\jmath| +1)\tau + |l|+ |\imath|+ |\jmath|}}{\gamma^{|l|+ |\imath|+ |\jmath|+1}}.
\end{eqnarray}
It follows from $(\ref{5Eq7})$, $(\ref{N3})$ and Cauchy's estimate that
\begin{eqnarray}
\nonumber |\partial_\lambda^l \partial_y^\imath \partial_z^\jmath f_{k00}|_{D(s)\times {\tilde{\Lambda}}_+} &\leq& \delta |\partial_\lambda^l \partial_y^\imath \partial_z^\jmath (L_k^{-1} p_{k00})|_{D(s)\times {\tilde{\Lambda}}_+}\\
\nonumber &\leq&\frac{\delta |k|^{\chi_1}}{\gamma^{|l|+|\imath|+|\jmath|+1}} \frac{\gamma^{l_0 +9} s^2 \mu e^{-|k| r}}{\eta_+^{l_0}}\\
\label{EQ11} &\leq& c\frac{\delta s^2 \mu |k|^{\chi_1 } e^{- |k|r}}{\eta_+^{l_0}}.
\end{eqnarray}

Recall $L_{k1} = \breve{L}_{k1} + \sqrt{-1} \langle k, \Delta_1\rangle I_{m+ 2m_0}.$ Then, according to the basic property of Hermitian matrix (\cite{Horn}), on $D(s)\times {\tilde{\Lambda}}_+$,
\begin{eqnarray}
\nonumber L_{k1}^* L_{k1} &=& \breve{L}_{k1}^* \breve{L}_{k1}+ \sqrt{-1} \langle k, \Delta_1\rangle \big( \breve{L}_{k1}^* - \breve{L}_{k1} + (\sqrt{-1} \langle k, \Delta_1\rangle I_{m+ 2m_0})^*\big)~~~~~\\
\nonumber &\geq& \frac{\gamma}{|k|^{\tau}} I_{m + 2m_0} - c \max \{s, \mu^{\frac{1}{6l_0^2}}\}  K_+ I_{m + 2m_0}\\
&\geq& \frac{\gamma}{2 |k|^{\tau}} I_{m + 2m_0}.
\end{eqnarray}
Therefore,
\begin{eqnarray*}
|\det L_{k1}^* L_{k1}|_{D(s)\times \tilde{\Lambda}_+} &=&  |\overline{\det  L_{k1}} \det  L_{k1}|_{D(s)\times \tilde{\Lambda}_+}\\
 &=& (|\det  L_{k1}|_{D(s)\times \tilde{\Lambda}_+})^2\\
 &\geq& (\frac{\gamma}{2|k|^\tau})^{m + 2m_0}.
\end{eqnarray*}
 Inductively,
\begin{eqnarray*}
|\partial_\lambda^l \partial_y^ \imath \partial_z^ \jmath (\frac{1}{\det {L}_{k1}})|_{D(s)\times \tilde{\Lambda}_+} &\leq& |k|^{(m + 2m_0) (|l| + |\imath| + |\jmath|)} |\frac{1}{\det L_{k1}}|^{|l| + |\imath|+ |\jmath|+1}\\
&\leq& \frac{|k|^{(m + 2m_0)(|l| + |\imath| + |\jmath|) + \tau (m + 2m_0)(|l| + |\imath| + |\jmath| + 1)}}{\gamma ^{(m + 2m_0)(|l| + |\imath| + |\jmath|+1)}}.
\end{eqnarray*}
Hence
\begin{eqnarray*}
&~&|\partial_\lambda^ l \partial_y^\imath \partial_z^\jmath \left(
                                                          \begin{array}{c}
                                                            f_{k10} \\
                                                            f_{k01} \\
                                                          \end{array}
                                                        \right)|_{D(s)\times \tilde{\Lambda}_+}\\ &=&  |\partial_\lambda^ l \partial_y^\imath \partial_z^\jmath \big(L_{k1}^{-1} ( \left(
                                                                                                                                                    \begin{array}{c}
                                                                                                                                                      p_{k10} \\
                                                                                                                                                      p_{k01} \\
                                                                                                                                                    \end{array}
                                                                                                                                                  \right) + \delta \left(
                                                                                                                                                                     \begin{array}{c}
                                                                                                                                                                       M_{12}J \partial_z f_{k00} \\
                                                                                                                                                                       M_{22}J \partial_z f_{k00} \\
                                                                                                                                                                     \end{array}
                                                                                                                                                                   \right)
                                                        )\big)|_{D(s)\times \tilde{\Lambda}_+}\\
&=&  |\partial_\lambda^ l \partial_y^\imath \partial_z^\jmath \big( \frac{adj L_{k1}}{\det L_{k1}} ( \left(
                                                                                                                                                    \begin{array}{c}
                                                                                                                                                      p_{k10} \\
                                                                                                                                                      p_{k01} \\
                                                                                                                                                    \end{array}
                                                                                                                                                  \right) + \delta \left(
                                                                                                                                                                     \begin{array}{c}
                                                                                                                                                                       M_{12}J \partial_z f_{k00} \\
                                                                                                                                                                       M_{22}J \partial_z f_{k00} \\
                                                                                                                                                                     \end{array}
                                                                                                                                                                   \right)
                                                        )\big)|_{D(s)\times \tilde{\Lambda}_+}\\
&\leq& \frac{|k|^{\chi_1}}{\gamma^{(m + 2m_0)(|l| + |\imath| + |\jmath| +1)}} \big( \frac{\delta \gamma^{l_0 + 9} s \mu e^{-|k|r}}{\eta^{l_0}} + \frac{s^2 \mu |k|^{\chi_1} e^{-|k|r}}{\eta^{l_0}}\big)\\
&\leq& \frac{\delta s \mu |k|^{3\chi_1} e^{-|k|r}}{\eta_+^{l_0}}.
\end{eqnarray*}

Similarly, on $D(s)\times {\tilde{\Lambda}}_+$,
\begin{eqnarray}
\nonumber L_{k2}^* L_{k2} &=& \breve{L}_{k2}^* \breve{L}_{k2}+ \sqrt{-1} \langle k, \Delta\rangle \big( \breve{L}_{k2}^* - \breve{L}_{k2} + (\sqrt{-1} \langle k, \Delta\rangle I_{m^2+ 2m m_0 + 4m_0^2})^*\big)\\
&\geq& \frac{\gamma}{2 |k|^{\tau}} I_{m^2 + 2m m_0 + 4m_0^2}.
\end{eqnarray}
Hence
\begin{eqnarray*}
&~&|\partial_\lambda^ l \partial_y^\imath \partial_z^\jmath \left(
                                                          \begin{array}{c}
                                                            f_{k20} \\
                                                            f_{k11} \\
                                                            f_{k02} \\
                                                          \end{array}
                                                        \right)|_{D(s)\times \tilde{\Lambda}_+}\\ &=&  |\partial_\lambda^ l \partial_y^\imath \partial_z^\jmath \left(
                                                                                                                                  \begin{array}{c}
                                                                                                                                    T(f_{k20}) \\
                                                                                                                                    T(f_{k11}) \\
                                                                                                                                    T(f_{k02}) \\
                                                                                                                                  \end{array}
                                                                                                                                \right)
                                                        |_{D(s)\times \tilde{\Lambda}_+}\\
&=&  |\partial_\lambda^ l \partial_y^\imath \partial_z^\jmath \big( \frac{adj L_{k2}}{\det L_{k2}} ( \left(
                                                                                                                                                    \begin{array}{c}
                                                                                                                                                     T(p_{k20}) \\
                                                                                                                                                     T(p_{k11}) \\
                                                                                                                                                     T(p_{k02}) \\
                                                                                                                                                    \end{array}
                                                                                                                                                  \right) \\
                                                                                                                                                  &~&+ \delta  \left(
                                                                                                                                                                     \begin{array}{c}
                                                                                                                                                                       T(M_{12}J \partial_z (f_{k10})^T) \\
                                                                                                                                                                       T(M_{12} J \partial_z (f_{k01})^T+ (M_{22}J\partial_z(f_{k00})^T)^T)\\
                                                                                                                                                                       T(M_{22}J \partial_z (f_{k01})^T) \\
                                                                                                                                                                     \end{array}
                                                                                                                                                                   \right)
                                                        )\big)|_{D(s)\times \tilde{\Lambda}_+}\\
&\leq& \frac{|k|^{\chi_1}}{\gamma^{(m^2 + 2m m_0+ 4m_0^2)(|l| + |\imath| + |\jmath| +1)}} \big( \frac{\delta \gamma^{l_0 + 9} \mu e^{-|k|r}}{\eta^{l_0}} + \frac{ \mu |k|^{\chi_1} e^{-|k|r}}{\eta^{l_0}}\big)\\
&\leq& \frac{\delta \mu |k|^{3\chi_1} e^{-|k|r}}{\eta_+^{l_0}}.
\end{eqnarray*}
Now, we finish the proof of part $\bf{(1)}$.

For part $\bf{(2)}$, by part $\bf{(1)}$ and directly differentiating to $(\ref{709})$, we have, on $\hat{D}(s) \times {{\tilde{\Lambda}}}_+$,
\begin{eqnarray*}
|\partial_\lambda^l\partial_x^i \partial_{(y,z)}^j F|_{\hat{D}(s) \times {{\tilde{\Lambda}}}_+}&\leq &\sum\limits_{ 0<|k|\leq K_+} |k|^{|i|} \big(|\partial_\lambda^l\partial_{(y,z)}^j f_{k00}|+ |\partial_\lambda^l \partial_{(y,z)}^j f_{k10}|s^{1-sgn|j|}\\
&~&+ |\partial_\lambda^l\partial_{(y,z)}^j f_{k01}|s^{1-sgn|j|}+|\partial_\lambda^l\partial_{(y,z)}^j f_{k20}|s^{1-sgn(|j|-1)}\\
&~&~~+|\partial_\lambda^l\partial_{(y,z)}^j f_{k02}|s^{1-sgn(|j|-1)}\\
&~& +|\partial_\lambda^l\partial_{(y,z)}^j f_{k11}|s^{1-sgn(|j|-1)}\big) e^{-|k|(r_+ + \frac{7}{8} (r - r_+))}\\
&\leq& c\frac{\delta \mu s^{a_j}}{\eta_+^{l_0}} \sum\limits_{0< |k|\leq K_+} |k|^{3\chi_1} e^{- \frac{|k|( r - r_+)}{8}}\\
&=& c \frac{\delta  \mu s^{a_j} \Gamma(r - r_+)}{\eta_+^{l_0}}.
\end{eqnarray*}
\end{proof}

Similar to $\bf{Lemma~ 3.6}$ of \cite{Li}, here, $F$ can also be smoothly extended to functions of H\"{o}lder class $C^{l_0 + \sigma_0 + 1, l_0 -1 + \sigma_0} (\hat{D}(s_0) \times \tilde{\Lambda}_0)$, where $0<\sigma_0< 1$ is fixed. Moreover, there is a constant $c$ such that
\begin{eqnarray*}
|F|_{C^{l_0 + \sigma_0 + 1, l_0 -1 + \sigma_0} (\hat{D}(\beta_0) \times \tilde{\Lambda}_0)} \leq c\delta \mu \Gamma(r - r_+).
\end{eqnarray*}

\begin{lemma}
Assume
\begin{itemize}
\item[\bf{(H5)}] $c \mu \Gamma( r- r_+) < \frac{1}{8} (r -r_+),$
\item[\bf{(H6)}] $c \mu \Gamma ( r- r_+) < \frac{1}{8}\alpha.$
\end{itemize}
Then the following hold$:$
\begin{itemize}
\item[\bf{1)}] For all $0\leq t\leq1$,
\begin{eqnarray}\label{712}
\phi_F^t&:& D_{\frac{1}{4}\alpha} \rightarrow D_{\frac{1}{2}\alpha},\\
\phi&:&  D_{\frac{1}{8}\alpha} \rightarrow D_{\frac{1}{4}\alpha}
\end{eqnarray}
are well defined, real analytic and depend smoothly on $\lambda \in {\Lambda}_+$$;$
\item[\bf{2)}]There is a constant $c$ such that for all $0\leq t\leq 1$, $|l|\leq l_0$, $|j|\leq2$, $|i|\leq l_0$,
\begin{eqnarray*}
|\partial_\lambda^l\partial_x^i \partial_{(y,z)}^j( \phi_F^t\circ \phi - id )|_{D_{\frac{1}{4} \alpha}\times {\tilde{\Lambda}}_+} \leq c\frac{ \mu \Gamma(r - r_+)}{\eta_+^{l_0}}.
\end{eqnarray*}
\end{itemize}
\end{lemma}
\begin{proof}

Let $\phi_F^t  = (\phi_1^t,\phi_2^t,\phi_3^t )^T$, where $\phi_1^t$, $\phi_2^t$ and $\phi_3^t$ are components of $\phi_F^t$ in $x-$, $y-$ and $z-$coordinate, respectively. Obviously, $\phi_F^t = id + \int_0^t X_F \circ \phi_F^s ds,$ where $X_F = (\partial_y F, -\partial_x F, J \partial_z F)^T$. Let $(x,y,z)$ be any point in $D_{\frac{\alpha}{4}}$ and let $t_* =\sup \{t\in [0,1]: \phi_F^t(x,y,z)\in D_{\alpha}\}.$ Then, for $t\in [0, t_*]$, $\lambda\in \Lambda_+$, with $\bf{(H5)}$ and $\bf{(H6)}$,
\begin{eqnarray*}
|\phi_1^t(x,y,z) - x|_{D_{\frac{\alpha}{4}}} &\leq& \int_0^t |F_y \circ \phi_F^s|_{D_{\alpha}} ds \leq |F_y|_{\hat{D}(s)} \leq \delta \mu \Gamma \leq \frac{1}{8} (r-r_+),\\
|\phi_2^t(x,y,z) - y|_{D_{\frac{\alpha}{4}}} &\leq& \int_0^t |F_x \circ \phi_F^s|_{D_{\alpha}} ds \leq |F_x|_{\hat{D}(s)} \leq \delta \mu s^2 \Gamma \leq \frac{\alpha s}{8},\\
|\phi_3^t(x,y,z) - z|_{D_{\frac{\alpha}{4}}} &\leq& \int_0^t |F_z \circ \phi_F^s|_{D_{\alpha}} ds \leq |F_z|_{\hat{D}(s)} \leq \delta \mu s \Gamma \leq \frac{\alpha s}{8},\\
\end{eqnarray*}
which implies $|\phi_1^t(x,y,z)|< r_+ + \frac{3}{8}(r- r_+)$, $|\phi_2^t(x,y,z)|<\frac{\alpha s}{2}$, $|\phi_3^t(x,y,z)|<\frac{\alpha s}{2}$, i.e. $\phi_F^t (x,y,z)\in D_{\frac{\alpha}{2}}$. Using $(\ref{E_20})$ and $\bf{(H6)}$, $\phi: D_{\frac{1}{8}\alpha} \rightarrow D_{\frac{1}{4}\alpha}$ is obvious.

The proof of $\bf{2)}$ follows from Lemma \ref{Lema1}.

\end{proof}

\subsubsection{New perturbation}

Here we will estimate the new perturbation $P_+$ on the domain $D_+ \times \Lambda_+$, where $D_+ = D_{\frac{\alpha}{8}}$.

\begin{lemma}
Assume
 \begin{itemize}
\item[\bf{(H7)}] $\mu^{\frac{1}{12}} \Gamma^3 (r - r_+) \leq \gamma_+^{l_0+ 9}$.
\end{itemize}
Then
\begin{eqnarray*}
|\partial_ \lambda^ l  P_+|_{D_+ \times {\tilde{\Lambda}}_+}\leq c\frac{\delta \gamma_+^{l_0 + 9} s_+^{2} \mu_+}{\eta_+^{l_0}}.
\end{eqnarray*}
\end{lemma}

\begin{proof}
Directly,
\begin{eqnarray*}
| R' |_{D_{\frac{\alpha}{4}} \times {\Lambda}_+} &\leq& c\delta^2 s^{3} \mu \Gamma(r -r_+).
\end{eqnarray*}
Denote $\partial^{i,j} = \partial_x^i\partial_{(y,z)}^j$ for $|j| \leq 2$, $|i| \leq l_0$. Then
\begin{eqnarray*}
|\partial^{i,j}( \int_0^1 \{R_t, F\} \circ \phi_F^t dt \circ \phi)|_{D_{\frac{\alpha}{4}} \times {{\Lambda}}_+} &\leq& c\delta s^{a_j} \mu^2 \Gamma^3(r - r_+),\\
|\partial^{i,j} ( P- R)\circ \phi_F^1\circ\phi|_{D_{\frac{\alpha}{4}} \times {{{\Lambda}}}_+} &\leq& c\delta \gamma^{l_0 + 9}s^{a_j} \mu^{2}\Gamma(r -r_+),\\
|\partial^{i,j} R' \circ \phi|_{D_{\frac{\alpha}{4}} \times {{\Lambda}}_+} &\leq& c\delta^2 s^{a_j+1} \mu \Gamma(r -r_+),\\
|\partial^{i,j} \langle \left(
                            \begin{array}{c}
                              y \\
                              z \\
                            \end{array}
                          \right), \left(
                                     \begin{array}{cc}
                                       p_{020} & \frac{1}{2}p_{011} \\
                                       \frac{1}{2}p_{011}^T & p_{002} \\
                                     \end{array}
                                   \right)\left(
                            \begin{array}{c}
                              y_0 \\
                              z_0 \\
                            \end{array}
                          \right)
\rangle|_{D_{\frac{\alpha}{8}} \times {\Lambda}_+} &\leq& c \delta \gamma^{l_0+9} s^{a_j} \mu^2.
\end{eqnarray*}
Further, by $(\ref{N5})$, we have
\begin{eqnarray*}
|\partial_\lambda^l  P_+|_{D_+ \times {\tilde{\Lambda}}_+} \leq c\frac{\delta s^{2} \mu^{2}\Gamma^3(r - r_+)}{\eta_+^{l_0}}.
\end{eqnarray*}
Here we use the fact that $s = c  \mu \mu_0^{-\frac{3}{4}} s_0$ and $\delta \mu_0^{-\frac{3}{4}} s_0 = o(c).$ (According to the construction of $s_\nu$ and $\mu_\nu$, obviously, $s = c  \mu \mu_0^{-\frac{3}{4}} s_0$.)
Using assumption $\bf{(H7)}$, we finish the proof of this lemma.

\end{proof}

\subsubsection{The preservation of frequencies}
Combining the argument in subsections \ref{Homological_equations} and \ref{Estimate}, if $M(\lambda)$ is nonsingular, there is a transformation $(\ref{E_21})$ such that all the frequencies are preserved after a KAM step. However, when $M(\lambda)$ is singular, $(\ref{E_1})$ is not solvable, i.e. there is no transformation such that all frequencies are preserved after a KAM step. To show the part preservation of frequency,  we give a simple property.

\begin{lemma}\label{Pro2}
For an $n\times n$ symmetrical matrix $A$ with $rank (A) = m$, there is an invertible matrix $T$ that corresponds to a linear transformation, under which only some rows of $A$ exchange, such that
          $$T^{-1} A T = \left(
                           \begin{array}{cc}
                             B & C \\
                             D & E \\
                           \end{array}
                         \right),
          $$
          where $B$ is an $m \times m$ nonsingular minor.
\end{lemma}

\begin{proof}
Rewrite
\begin{eqnarray*}
A = \left(
      \begin{array}{c}
        a_1 \\
        a_2 \\
        \vdots \\
        a_n \\
      \end{array}
    \right) = (b_1, b_2, \cdots, b_n),
\end{eqnarray*}
where $a_i$ is $i-$th row of $A$ and $b_i$ is $i-$th column of $A$, $i=1, \cdots, n.$
Since $A$ is symmetrical, $a_i = b_i^T$, $i = 1, \cdots, n$, which means that there is a same linear relation between $a_i$ and $b_i$, $i=1, \cdots, n$. Because $rank (A) = m$, there are $m$ linearly independent rows (columns) of $A$. Then there is an invertible matrix $T$, which corresponds to a linear transformation that exchange some rows of $A$, such that
\begin{eqnarray*}
T \left(
    \begin{array}{c}
      a_1 \\
      a_2 \\
      \vdots \\
      a_n \\
    \end{array}
  \right) = \left(
              \begin{array}{c}
                a_1^1 \\
                a_2^1 \\
                \vdots \\
                a_m^1  \\
                \vdots \\
                a_n^1  \\
              \end{array}
            \right),
\end{eqnarray*}
where $a_1^1, ~\cdots,~ a_m^1$ are linearly independent. Since $T^{-1} = T$ and $T^{-1}$ does not change the linear relation among $b_1$, $\cdots$, $b_m$, we get
\begin{eqnarray*}
T^{-1} A T &=& \left(
              \begin{array}{c}
                a_1^1 \\
                a_2^1 \\
                \vdots \\
                a_m^1  \\
                \vdots \\
                a_n^1  \\
              \end{array}
            \right) T\\
         &=& \left(
               \begin{array}{cc}
                 B & C \\
                 D & E \\
               \end{array}
             \right),
\end{eqnarray*}
where $B$ is an $m \times m$ nonsingular minor.

\end{proof}

Combining assumption $\bf{(A2)}$ and $\bf{Lemma~\ref{Pro2}}$, there is an invertible matrix $T$, which corresponds to a transformation only exchanging columns or rows, such that
\begin{eqnarray*}
T^{-1} \left(
    \begin{array}{cc}
      M_{11} & M_{12} \\
      M_{21} & M_{22} \\
    \end{array}
  \right) T = \left(
              \begin{array}{cc}
                C_{11} & C_{12} \\
                C_{21} & C_{22} \\
              \end{array}
            \right),
\end{eqnarray*}
where $(C_{11}, C_{12})_{(n+2m_0)\times(m + 2m_0)}$ is a matrix with $rank (C_{11}, C_{12}) = n + 2m_0$ and $(C_{21}, C_{22})_{(m-n)\times(m + 2m_0)}$ is the complements. Moreover, $(C_{11})_{(n+2m_0)\times (n+2m_0)}$ is nonsingular. Denote
$
\left(
         \begin{array}{c}
           y_1 \\
           y_2 \\
         \end{array}
       \right) = T^{-1} \left(
                          \begin{array}{c}
                           y_0 \\
                           z_0 \\
                          \end{array}
                        \right),
\left(
  \begin{array}{c}
    p_1 \\
    p_2 \\
  \end{array}
\right)= T^{-1} \left(
                                                 \begin{array}{c}
                                                   p_{010} \\
                                                   p_{001} \\
                                                 \end{array}
                                               \right)
,$
where $p_1, ~y_1 = (y_3, z_0)^T\in R^{n+2m_0}$, $y_2,$ $p_2$ $\in R^{m-n}$, $p_{010}, ~y_0 = (y_3, y_2)^T\in R^{m}$, $p_{001}, ~z_0\in R^{2m_0}$.  Then (\ref{E_1}) is changed to $:$
\begin{eqnarray}\label{933}
\frac{\delta}{2} \left(
                   \begin{array}{cc}
                     C_{11} & C_{12} \\
                     C_{21} & C_{22} \\
                   \end{array}
                 \right)
 \left(
                     \begin{array}{c}
                       y_1 \\
                       y_2 \\
                     \end{array}
                   \right)+ \delta \left(
                                     \begin{array}{c}
                                       \partial_{y_1} h(y_0, z_0, \lambda) \\
                                       \partial_{y_2} h(y_0, z_0, \lambda) \\
                                     \end{array}
                                   \right) = -\left(
                                                 \begin{array}{c}
                                                   p_1 \\
                                                   p_2 \\
                                                 \end{array}
                                               \right).
\end{eqnarray}
Since $ rank (C_{11}, C_{12}) =rank \left(
                                  \begin{array}{cc}
                                    C_{11} & C_{12} \\
                                    C_{21} & C_{22} \\
                                  \end{array}
                                \right),
$
there is an invertible matrix $T_1$ that only exchange columns or rows such that
$
T_1 \left(
                                  \begin{array}{cc}
                                    C_{11} & C_{12} \\
                                    C_{21} & C_{22} \\
                                  \end{array}
                                \right) = \left(
                                  \begin{array}{cc}
                                    C_{11} & C_{12} \\
                                    0 & 0 \\
                                  \end{array}
                                \right),
$
which is equivalent to the fact that the rows of $(C_{21}, C_{22})$ is linearly dependent on the rows of $(C_{11}, C_{12}).$
Obviously, $T_1$ is a matrix with the following form $\left(
                        \begin{array}{cc}
                          I & 0 \\
                          D_1 & I \\
                        \end{array}
                      \right),
$
where $D_1$ is determined by the linear relation among the rows of $(C_{21}, C_{22})$ and $(C_{11}, C_{12})$. Then
\begin{eqnarray*}
T_1 \left(
                                     \begin{array}{c}
                                       \partial_{y_1} h(y_0, z_0, \lambda) \\
                                       \partial_{y_2} h(y_0, z_0, \lambda) \\
                                     \end{array}
                                   \right) &=& \left(
                                     \begin{array}{c}
                                       \partial_{y_1} h(y_0, z_0, \lambda) \\
                                       D_1 \partial_{y_1} h(y_0, z_0, \lambda)+ \partial_{y_2} h (y_0, z_0, \lambda)\\
                                     \end{array}
                                   \right),\\
T_1 \left(
                                                 \begin{array}{c}
                                                   p_1 \\
                                                   p_2 \\
                                                 \end{array}
                                               \right) &=& \left(
                                                 \begin{array}{c}
                                                   p_1 \\
                                                   D_1 p_1 + p_2 \\
                                                 \end{array}
                                               \right).
\end{eqnarray*}
Consider the following equation
\begin{eqnarray}\label{934}
\frac{\delta}{2} \left(
                   \begin{array}{cc}
                     C_{11} & C_{12} \\
                     0 & 0 \\
                   \end{array}
                 \right)
 \left(
                     \begin{array}{c}
                       y_1 \\
                       y_2 \\
                     \end{array}
                   \right)+ \delta \left(
                                     \begin{array}{c}
                                       \partial_{y_1} h(y_0, z_0, \lambda) \\
                                       0 \\
                                     \end{array}
                                   \right) = -\left(
                                                 \begin{array}{c}
                                                   p_1 \\
                                                   0 \\
                                                 \end{array}
                                               \right),
\end{eqnarray}
where $C_{11}$ is nonsingular. Obviously, $(y_1, y_2)^T  = (y_1, 0)^T$ is a specific solution of $(\ref{934})$, i.e., with assumption $\bf{(A2)}$ there is a symplectic transformation such that part of the frequencies are preserved.
\begin{remark}
If $M$ is singular,  some of the frequencies are preserved and the others drift. Moreover, the drift depends on $D_1 p_1 + p_2$ and $D_1 \partial_{y_1} h(y_0, z_0, \lambda)+ \partial_{y_2} h (y_0, z_0, \lambda)$ and the estimate on drift is showed by \emph{(\ref{drift})}.
\end{remark}

Consider $:$
\begin{eqnarray}
\label{935}  \langle \omega, y_0\rangle + \frac{\delta}{2} \langle \left(
                                                                    \begin{array}{c}
                                                                      y_0 \\
                                                                      z_0 \\
                                                                    \end{array}
                                                                  \right), M \left(
                                                                               \begin{array}{c}
                                                                                 y_0 \\
                                                                                 z_0 \\
                                                                               \end{array}
                                                                             \right)
\rangle + p_{000}+ \langle \left(
                             \begin{array}{c}
                               p_{010} \\
                               p_{001} \\
                             \end{array}
                           \right), \left(
                                      \begin{array}{c}
                                        y_0 \\
                                        z_0 \\
                                      \end{array}
                                    \right)
\rangle&~&\\
\nonumber~~~+ \langle\left(
          \begin{array}{c}
            y_0 \\
            z_0 \\
          \end{array}
        \right), \left(
                   \begin{array}{cc}
                     p_{020} & \frac{1}{2}p_{011} \\
                     \frac{1}{2}p_{011}^T  & p_{002} \\
                   \end{array}
                 \right) \left(
                           \begin{array}{c}
                             y_0 \\
                             z_0 \\
                           \end{array}
                         \right)\rangle + \delta h(y_0, z_0, \lambda)&=& 0,\\
\label{936} \frac{\delta M}{2} \left(
                     \begin{array}{c}
                       y_0 \\
                       z_0 \\
                     \end{array}
                   \right)+ \delta \left(
                                     \begin{array}{c}
                                       \partial_y h(y_0, z_0, \lambda) \\
                                       \partial_z h(y_0, z_0, \lambda) \\
                                     \end{array}
                                   \right) + \left(
                                                 \begin{array}{c}
                                                   p_{010} \\
                                                   p_{001} \\
                                                 \end{array}
                                               \right) - t \left(
                                                                   \begin{array}{c}
                                                                     \omega \\
                                                                     0 \\
                                                                   \end{array}
                                                                 \right)
                                                &=& 0.~~~~~~~~
\end{eqnarray}

If $M$ is nonsingular, according to $\bf{(A3)}$ and the continuity of determinant, we have $\det \left(
  \begin{array}{cc}
    M & \bar{\omega}_1 \\
    \bar{\omega}_2 & 0 \\
  \end{array}
\right) \neq 0,$ where $\omega_1 = (\omega, 0)^T \in R^{n+ 2m_0}$, $\omega_2 = (p_{010}+ \omega, p_{001})$. Then, combining $(\ref{935})$ and $(\ref{936})$, with implicit theorem we get $(y_0, z_0, t),$ i.e., we construct a transformation such that on the same energy surface the ratios of the frequencies are preserved after a KAM step.
 \begin{remark}
 If $M$ is nonsingular, the condition $\det \left(
  \begin{array}{cc}
    M & \bar{\omega}_1 \\
    \bar{\omega}_1 & 0 \\
  \end{array}
\right) \neq 0$ is a generalization of the isoenergetically nondegenerate condition given by V. I. Arnold (\cite{Arnold1}) to the persistence of lower dimensional invariant tori on a given energy surface, where $\omega_1 = (\omega, 0)^T $.
 \end{remark}

Assume $M$ is singular and conditions $\bf{(A2)}$ and $\bf{(A3)}$ hold. Denote $\tilde{\omega}_1$ by the first $n+ 2m_0$ components of $T_1T^{-1} (\omega, 0)^T$, which is equal to the first $n+ 2m_0$ components of $T^{-1} (\omega, 0)^T$. In fact, $$T_1T^{-1} \left(
                                                                                                                             \begin{array}{c}
                                                                                                                               \omega \\
                                                                                                                               0 \\
                                                                                                                             \end{array}
                                                                                                                           \right)
= T_1 \left(
                                                                                                                             \begin{array}{c}
                                                                                                                               \tilde{\omega}_1 \\
                                                                                                                               \omega_4 \\
                                                                                                                             \end{array}
                                                                                                                           \right)=\left(
                                                                                                                                     \begin{array}{cc}
                                                                                                                                       I & 0 \\
                                                                                                                                       D_1 & I \\
                                                                                                                                     \end{array}
                                                                                                                                   \right)\left(
                                                                                                                             \begin{array}{c}
                                                                                                                               \tilde{\omega}_1 \\
                                                                                                                               \omega_4 \\
                                                                                                                             \end{array}
                                                                                                                           \right)=
                                                                                                                           \left(
                                                                                                                             \begin{array}{c}
                                                                                                                               \tilde{\omega}_1 \\
                                                                                                                              D_1\tilde{\omega}_1+ \omega_4 \\
                                                                                                                             \end{array}
                                                                                                                           \right),
                                                                                                                           $$
where $\omega = (\omega_3, \omega_4)^T\in R^{m},$ $\tilde{\omega}_1 = (\omega_3, 0)^T\in R^{n+ 2m_0}.$ Similarly, combining $(\ref{934})$, we have
\begin{eqnarray*}
\frac{\delta}{2} \left(
                   \begin{array}{cc}
                     C_{11} & C_{12} \\
                     0 & 0 \\
                   \end{array}
                 \right)
 \left(
                     \begin{array}{c}
                       y_1 \\
                       y_2 \\
                     \end{array}
                   \right)+ \delta \left(
                                     \begin{array}{c}
                                       \partial_{y_1} h(y_0, z_0, \lambda) \\
                                       0 \\
                                     \end{array}
                                   \right)- t\left(
                                               \begin{array}{c}
                                                 \tilde{\omega}_1 \\
                                                 0 \\
                                               \end{array}
                                             \right)
 = -\left(
                                                 \begin{array}{c}
                                                   p_1 \\
                                                   0 \\
                                                 \end{array}
                                               \right).
\end{eqnarray*}
Assume
\begin{eqnarray}\label{937}
det \left(
  \begin{array}{cc}
   C_{11} & \tilde{\omega}_1 \\
    \tilde{\omega}_2 & 0 \\
  \end{array}
\right) \neq 0,
\end{eqnarray}
where $\tilde{\omega}_2 $ is the first $n+ 2m_0$ components of $(p_{010}+ \omega, p_{001}) T$.
Then there is a $(y_{i_1}^0, \cdots, y_{i_n}^0, 0, \cdots, 0, z_1^0, \cdots, z_{2m_0}^0, t)$ such that
\begin{eqnarray*}
\langle \omega, y_0\rangle + \frac{\delta}{2} \langle \left(
                                                                    \begin{array}{c}
                                                                      y_0 \\
                                                                      z_0 \\
                                                                    \end{array}
                                                                  \right), M \left(
                                                                               \begin{array}{c}
                                                                                 y_0 \\
                                                                                 z_0 \\
                                                                               \end{array}
                                                                             \right)
\rangle + p_{000}+ \langle \left(
                             \begin{array}{c}
                               p_{010} \\
                               p_{001} \\
                             \end{array}
                           \right), \left(
                                      \begin{array}{c}
                                        y_0 \\
                                        z_0 \\
                                      \end{array}
                                    \right)
\rangle&~&\\
\nonumber~~~+ \langle\left(
          \begin{array}{c}
            y_0 \\
            z_0 \\
          \end{array}
        \right), \left(
                   \begin{array}{cc}
                     p_{020} & \frac{1}{2}p_{011} \\
                     \frac{1}{2}p_{011}^T  & p_{002} \\
                   \end{array}
                 \right) \left(
                           \begin{array}{c}
                             y_0 \\
                             z_0 \\
                           \end{array}
                         \right)\rangle + \delta h(y_0, z_0, \lambda)&=& 0,\\
\frac{\delta}{2} \left(
                   \begin{array}{cc}
                     C_{11} & C_{12} \\
                     0 & 0 \\
                   \end{array}
                 \right)
 \left(
                     \begin{array}{c}
                       y_1 \\
                       y_2 \\
                     \end{array}
                   \right)+ \delta \left(
                                     \begin{array}{c}
                                       \partial_{y_1} h(y_0, z_0, \lambda) \\
                                       0 \\
                                     \end{array}
                                   \right)- t\left(
                                               \begin{array}{c}
                                                 \tilde{\omega}_1 \\
                                                 0 \\
                                               \end{array}
                                             \right) &=& -\left(
                                                 \begin{array}{c}
                                                   p_1 \\
                                                   0 \\
                                                 \end{array}
                                               \right).
\end{eqnarray*}
Finally, combining $\bf{(A2)}$, $\bf{(A3)}$, $\bf{Property ~\ref{Pro2}}$ and the continuity of determinant, assumption $(\ref{937})$ holds. Therefore, on a given energy surface there is a transformation such that ratios of frequencies between the unperturbed torus and the perturbed are preserved.
\begin{remark}
Assume $\bf{(A2)}$ and $\bf{(A3)}$. For a given energy, $n$ coordinates of the frequency $\omega_{+}$ coincide with $n$ coordinates of $t \omega$, where $t\rightarrow 0$ as $\varepsilon \rightarrow 0$. Simultaneously, the other frequencies slightly drift and the drift depend on $D_1 p_1 + p_2$ and $D_1 \partial_{y_1} h(y_0, z_0, \lambda)+ \partial_{y_2} h (y_0, z_0, \lambda)$.
\end{remark}

\subsection{Iteration Lemma}\label{727}
Let $r_0$, $\gamma_0$, $s_0$, $\eta_0$, $\Lambda_0$, $H_0$, $N_0$, $e_0$, $P_0$ be given as above and denote $\hat{D}_0 = D(r_0, \beta_0)$. For any $\nu = 0,1, \cdots,$  denote
\begin{eqnarray*}
r_\nu &=& r_0 (1 - \sum_{i=1}^\nu \frac{1}{2^{i+1}}), ~~~\gamma_\nu = \gamma_0 (1 - \sum_{i=1}^\nu \frac{1}{2^{i+1}}),~~~\alpha_\nu = \mu_\nu^{\frac{1}{3}}\\
\eta_\nu &=&\mu_\nu^{\frac{1}{6 l_0^2}},~~~\mu_\nu = 64c_0 \mu_{\nu-1}^{\frac{13}{12}}, ~~~K_\nu = ([\log\frac{1}{\mu_{\nu-1}}]+1)^{3 {\eta}},\\
D_\nu &=& D(r_\nu, s_\nu),~~~~~~~\hat{D}_\nu = D(r_\nu+ \frac{7}{8} (r_{\nu-1} - r_\nu)),~~~s_\nu = \frac{1}{8} \alpha_{\nu-1} s_{\nu-1},\\
\Lambda_\nu &=& \{\lambda\in \Lambda_{\nu-1}: |\breve{L}_{k0, \nu}| > \frac{\gamma_\nu}{|k|^\tau}, \breve{L}_{k1, \nu}^* \breve{L}_{k1, \nu} > \frac{\gamma_\nu}{|k|^\tau}  I_{m + 2m_0},  \\
&~&\breve{L}_{k2, \nu}^* \breve{L}_{k2, \nu} > \frac{\gamma_\nu}{|k|^\tau}  I_{m^2 + 2m m_0+ 4m_0^2},  ~for~all~0< |k|\leq K_\nu\},\\
\tilde{\Lambda}_\nu &=& \{\lambda\in \mathds{C}^m, |\lambda - \Lambda_\nu| \leq 4\eta_\nu\}.
\end{eqnarray*}

We have the following Iteration Lemma.
\begin{lemma}\label{balala}
Assume $(\ref{714})$ hold. Then the KAM step described in Section $\ref{KAM}$ is valid for all $\nu = 0,1,\cdots$, and the following facts hold for all $\nu = 1,2,\cdots.$
 \begin{itemize}
\item[\bf{(1)}] $P_\nu$ is real analytic in $(x,y,z)\in D_\nu$, smooth in $(x,y,z)\in \hat{D}_\nu$ and smooth in $\lambda \in {\Lambda}_\nu$, and moreover,
\begin{eqnarray*}
|\partial_\lambda^l  P_\nu |_{D_\nu \times {\tilde{\Lambda}}_\nu} \leq c\frac{\delta\gamma_\nu^{l_0 + 9} s_\nu^{2} \mu_\nu}{\eta_\nu^{l_0}}, ~|l| \leq l_0;
\end{eqnarray*}
\item[\bf{(2)}] $\Phi_\nu = \phi_F^t\circ \phi: \hat{D} \times {\Lambda}_0 \rightarrow \hat{D}_{\nu - 1}, D_\nu \times {\Lambda}_\nu \rightarrow D_{\nu-1}$, is symplectic for each $\lambda \in {\Lambda}_0$, and is of class $C^{l_0 + 1 + \sigma_0, l_0 -1 +\sigma_0}$, $C^{\alpha, l_0}$, respectively, where $\alpha$ stands for real analyticity and $0<\sigma_0<1$ is fixed. Moreover,
\begin{eqnarray*}
{H}_\nu = H_{\nu-1}\circ \Phi_\nu = N_\nu+ {P}_\nu,
\end{eqnarray*}
on $\hat{D} \times {\Lambda}_\nu$, and
\begin{eqnarray*}
|\Phi_\nu - id |_{C^{l_0+ 1 + \sigma_0,l_0 -1 + \sigma_0} (\hat{D} \times \tilde{\Lambda}_0)} \leq c_0 \delta \frac{\mu_0}{2^\nu};
\end{eqnarray*}
\item[\bf{(3)}] $\Lambda_\nu = \{\lambda\in \Lambda_{\nu-1}: |\breve{L}_{k0, \nu}| > \frac{\gamma_\nu}{|k|^\tau}, \breve{L}_{k1, \nu}^* \breve{L}_{k1, \nu} > \frac{\gamma_\nu}{|k|^\tau}  I_{m + 2m_0},  \\
   ~~~~~~~ \breve{L}_{k2, \nu}^* \breve{L}_{k2, \nu} > \frac{\gamma_\nu}{|k|^\tau}  I_{m^2 + 2m m_0+ 4m_0^2},  ~for~all~0< |k|\leq K_\nu\}$.
\end{itemize}
\end{lemma}

\begin{proof}
The proof of this lemma is to verify conditions $\bf{(H1)} - \bf{(H7)}$. Those are standard and we place the detail on Appendix \ref{B}.

\end{proof}

\subsection{Convergence and measure estimate}\setcounter{equation}{0}
Let $ \Psi^{\nu } = \Phi_1 \circ \Phi_2 \circ \cdots \circ \Phi_\nu, ~~ \nu = 1,2, \cdots.$ Then $\Psi^{\nu}: \tilde{D}_{\nu} \times \Lambda _0(g,G) \rightarrow \tilde{D}_0$, and
  \begin{eqnarray}
   \nonumber H_0 \circ \Psi^{\nu} &=& H_{\nu} = N_{\nu}+ P_{\nu},\\
   \nonumber  N_{\nu}&=& e_{\nu} + \langle \omega_\nu , y\rangle+ h_\nu (y, \omega),~~\nu= 0,1,\cdots,
  \end{eqnarray}
  where $\Psi _0 = id $.

  Standardly, $N_\nu$ converges uniformly to $N_\infty$, $P_\nu$ converges uniformly to $P_\infty$
 and $\partial _y^i \partial_z^j P_\infty = 0,~|i|+|j|\leq 2.$

Hence for each $\lambda\in \Lambda_{\infty}$, $T^d \times \{0\} \times \{0\}$ is an analytic invariant torus of $H_\infty$ with the toral frequency $\omega_\infty$, which for all $k\in Z^m \backslash \{0\},~1\leq q\leq n$, by the definition of $\Lambda_\nu$ and Lemma \ref{balala} (2), satisfies the following facts
\begin{itemize}
  \item [\bf{(1)}]  if $\bf{(A1)}$ holds and $M$ is nonsingular, then $\omega_\infty  \equiv \omega_0,~~~|\langle k, \omega_\infty \rangle| > \frac{\gamma}{2|k|^\tau};$
  \item [\bf{(2)}]if $\bf{(A1)}$ and $\bf{(A3)}$ hold and $M$ is nonsingular, then on a given energy surface $\omega_\infty  \equiv t \omega_0, ~~|\langle k, \omega_\infty \rangle| > \frac{\gamma}{2|k|^\tau};$
  \item [\bf{(3)}]if $\bf{(A1)}$ and $\bf{(A2)}$ hold, then $ (\omega_\infty)_{i_q}  \equiv (\omega_0)_{i_q}, q= 1, \cdots, n, ~~|\langle k, \omega_\infty \rangle| > \frac{\gamma}{2|k|^\tau};$
  \item [\bf{(4)}]if $\bf{(A1)}$, $\bf{(A2)}$ and $\bf{(A3)}$ hold, then $(\omega_\infty)_{i_q}  \equiv t (\omega_0)_{i_q}, q= 1, \cdots, n, ~~|\langle k, \omega_\infty \rangle| > \frac{\gamma}{2|k|^\tau}.$
\end{itemize}

Following the Whitney extension of $\Psi^\nu,$ all $e_\nu,$ $\omega_\nu,$ $h_\nu,$ $P_\nu,$ $(\nu = 0,1,\cdots)$ admit uniform $C^{l_0-1 +\sigma_0}$ extensions in $\lambda \in \Lambda_0$ with derivatives in $\lambda$ up to order $l_0-1$. Thus, $e_\infty$, $\omega_\infty$, $h_\infty$, $P_\infty$ are $C^{l_0-1}$ Whitney smooth in $\lambda \in \Lambda_{\infty}$, and the derivatives of $e_\infty -e_0$, $\omega_\infty -\omega_0$, $h_\infty -h_0$ satisfy similar estimates. Consequently, the perturbed tori form a $C^{l_0-1}$ Whitney smooth family on $\Lambda_{\infty}(g,G)$.

The measure estimate is the same as ones in \cite{LLL,Qian,Qian2,Qian3} and for the sake of completeness we place details on Appendix \ref{C}. Now we have finished the proof of Theorem \ref{shengluede}.

\section{Proof of Theorem \ref{dingli11}}\label{074}

For $d$-dimensional manifold $\mathcal{M}$ with a global coordinate, there is a bounded closed region $\Lambda\in R^{m}$ and a $C^{l_0}$ diffeomorphism $I: \Lambda\rightarrow \mathcal{M}$ such that $\mathcal{M} = I(\Lambda)$. Under the transformation $I \mapsto I+ I(\lambda)$, Hamiltonian system $(\ref{005})$ is changed to
\begin{eqnarray}\label{EQ17}
H(I, \theta, \lambda,\varepsilon) = e + \langle \omega(\lambda), I\rangle + \frac{1}{2}\langle I, \partial_I^2 H_0(\lambda) I\rangle + O(|I|^3) + \varepsilon P(I,\theta,\lambda,\varepsilon),
\end{eqnarray}
where $e = H_0(I(\lambda))$, $\omega(\lambda) = \partial_I H (I(\lambda))$. Let
\begin{eqnarray*}
\Gamma = K_0^T \partial_{I}^2 H_0 (\lambda) K_0 = \left(
                                              \begin{array}{cc}
                                                \Gamma_{11} & \Gamma_{12} \\
                                                \Gamma_{21} & \Gamma_{22} \\
                                              \end{array}
                                            \right),
\end{eqnarray*}
where $\Gamma_{11},$ $\Gamma_{12}$, $\Gamma_{21},$ $\Gamma_{22}$ are $m\times m,$ $m\times m_0$, $m_0\times m,$ $m_0\times m_0$ matrices, respectively, $\Gamma_{12} = \Gamma_{21}^T$, $\Gamma_{22} = {K'}^T \partial_{I}^2 H_0 (\lambda) K'$, and $m_0 = d - m$. Denote $\omega^{*}(\lambda) = K_{*}^T \omega(\lambda) \in \widetilde{\Lambda} (g,\Lambda),$ where $\widetilde{\Lambda}(g, \Lambda) = \{ \lambda\in \Lambda: \langle k, \omega(\lambda)\rangle = 0, k \in g\}$ and $\hat{\Lambda} (g, \Lambda) = \{\omega^*(\lambda) = K_*^T\omega\in R^m, \lambda \in \widetilde{\Lambda}(g, \Lambda)\}$. Recall $p=(y, v),$ $q= (x,u)$, where $y = (p_1, \cdots ,p_m)^T,$ $v = (p_{m+1}, \cdots ,p_d)^T,$ $x = (q_1, \cdots ,q_m)^T,$ $u = (q_{m+1}, \cdots ,q_d)^T.$ For any $\lambda \in \widetilde{\Lambda}(g, \Lambda)$, with the following coordinate transformation $I  = K_0 p$, $q = K_0 ^ T \theta$, Hamiltonian (\ref{EQ17}) is changed to
\begin{eqnarray}
\nonumber H(x,y,u,v)&=& \langle\omega^{*} , y\rangle + \frac{1}{2}\langle \left(
                                                                \begin{array}{c}
                                                                  y \\
                                                                  v \\
                                                                \end{array}
                                                              \right)
 , \Gamma(\lambda) \left(
                                                                \begin{array}{c}
                                                                  y \\
                                                                  v \\
                                                                \end{array}
                                                              \right)\rangle\\
\label{qq}  &~& + O(|K_0 \left(
                                                                \begin{array}{c}
                                                                  y \\
                                                                  v \\
                                                                \end{array}
                                                              \right)|^3) + \varepsilon  \bar{P} (x,y,u,v, \varepsilon)
\end{eqnarray}
up to a constant, where $\bar{P}(x, y, u, v, \varepsilon) = P( K_0 \left(
             \begin{array}{c}
               y \\
               v \\
             \end{array}
           \right), (K_0 ^T)^{-1} \left(
          \begin{array}{c}
            x \\
            u \\
          \end{array}
        \right),\varepsilon).$ By the following symplectic transformation:
 \begin{eqnarray}\label{Eq1p}
 \left(
   \begin{array}{c}
     y \\
     v \\
   \end{array}
 \right)
  \rightarrow \varepsilon^{\frac{1}{4}} \left(
   \begin{array}{c}
     y \\
     v \\
   \end{array}
 \right), ~\left(
             \begin{array}{c}
               x \\
               u \\
             \end{array}
           \right)
   \rightarrow \left(
             \begin{array}{c}
               x \\
               u \\
             \end{array}
           \right), ~H \rightarrow\varepsilon^{-\frac{1}{4}}H,
 \end{eqnarray}
 Hamiltonian (\ref{qq}) is changed to
 \begin{eqnarray}\label{N1}
\nonumber H(x,y,u,v)&=& \langle\omega^{*} , y\rangle + \frac{\varepsilon^{\frac{1}{4}}}{2}\langle \left(
                                                                \begin{array}{c}
                                                                  y \\
                                                                  v \\
                                                                \end{array}
                                                              \right), \Gamma(\lambda) \left(
                                                                \begin{array}{c}
                                                                  y \\
                                                                  v \\
                                                                \end{array}
                                                              \right) \rangle\\
 &~& + \varepsilon^{\frac{1}{2}}O(|K_0 \left(
                                                                \begin{array}{c}
                                                                  y \\
                                                                  v \\
                                                                \end{array}
                                                              \right)|^3) + \varepsilon^{\frac{3}{4}}  \bar{P} (x,y,u,v, \varepsilon).
 \end{eqnarray}

In order to use Theorem $\ref{shengluede}$, we should reduce Hamiltonian system $(\ref{N1})$ to $(\ref{model33})$. But the traditional method fails due to high degeneracy of perturbation, which does not guarantee that the perturbation is sufficiently small. Hence we have to proceed a program, finite quasilinear KAM steps, to improve the order of the perturbation. To fix thought, we only give an outline.

Let $\epsilon = \varepsilon^{\frac{1}{4}}$. Rewrite Hamiltonian system (\ref{N1}) with the following form:
\begin{eqnarray} \label{xiaoming}
H_1(x,y,u,v) &=& N_1(y,v) + \epsilon^2 {P_1} (x,y,u,v,\epsilon),
\end{eqnarray}
where $N_1= \langle\omega_1, y\rangle + \hat{h}_1,$ $\hat{h}_1=\frac{\epsilon}{2} \langle \left(
                                                                                   \begin{array}{c}
                                                                                     y \\
                                                                                     v \\
                                                                                   \end{array}
                                                                                 \right)
, \breve{M}_1 \left(
           \begin{array}{c}
             y \\
             v \\
           \end{array}
         \right)
\rangle+ \epsilon^2 O(|K_0 \left(
                                                                       \begin{array}{c}
                                                                         y \\
                                                                         v \\
                                                                       \end{array}
                                                                     \right)|^3),$\\ $P_1(x,y,u,v) = \epsilon P(x,y,u,v).$
Rewrite $\breve{M}_1 = \left(
                      \begin{array}{cc}
                        \breve{M}_{11,1} & \breve{M}_{12,1} \\
                        \breve{M}_{21,1} & \breve{M}_{22,1} \\
                      \end{array}
                    \right)
,$ where $\breve{M}_{11,1},$ $\breve{M}_{12,1}$, $\breve{M}_{21,1},$ $\breve{M}_{22,1}$ are $m\times m,$ $m\times m_0$, $m_0\times m,$ $m_0\times m_0$ matrices, respectively. Let $z =  (u,v)$ and $M_1 = \left(
                               \begin{array}{ccc}
                                 \breve{M}_{11,1} & \tilde{M}_{1,1} & \breve{M}_{12,1} \\
                                \tilde{M}_{2,1}  & \tilde{M}_{3,1} & \tilde{M}_{4,1} \\
                                 \breve{M}_{21,1} & \tilde{M}_{5,1} & \breve{M}_{22,1} \\
                               \end{array}
                             \right)
$, where $\tilde{M}_{1,1} =0$, $\tilde{M}_{2,1} =0,$  $\tilde{M}_{3,1}=0, $ $\tilde{M}_{4,1}= 0,$ $\tilde{M}_{5,1} =0$ with obvious dimension. Choose $\epsilon = \delta$, $\gamma = \delta^{\frac{1}{4(9+l_0)}}$, $s = \delta^{\frac{1}{4}}$, $\mu = \delta^{\frac{1}{4}}$. Then $(\ref{xiaoming})$ is changed to
\begin{eqnarray} \label{xiaoming1}
H_1(x,y,u,v,\lambda) &=& N_1(y,v,\lambda) + \delta^2 {P_1} (x,y,u,v,\lambda,\varepsilon),
\end{eqnarray}
where $N_1= \langle\omega_1(\lambda), y\rangle + \hat{h}_1,$ $\hat{h}_1=\frac{\delta}{2} \langle \left(
                                                                                   \begin{array}{c}
                                                                                     y \\
                                                                                     z \\
                                                                                   \end{array}
                                                                                 \right)
, M_1(\lambda) \left(
           \begin{array}{c}
             y \\
             z \\
           \end{array}
         \right)
\rangle+ \delta^2 h_1,$ $|{P_1}(x,y,u,v,\lambda)| \leq \gamma^{l_0+9} s^2 \mu.$
Here, $h_1$ is a polynomial of $K_0 \left(
                                                                       \begin{array}{c}
                                                                         y \\
                                                                         v \\
                                                                       \end{array}
                                                                     \right)$ from the third order term.
Let ${M}_{11,1} = \breve{M}_{11,1}$, ${M}_{12,1} = (\tilde{M}_{1,1}, \breve{M}_{11,1})$, ${M}_{21,1} = \left(
                                                                                                   \begin{array}{c}
                                                                                                     \tilde{M}_{2,1} \\
                                                                                                     \breve{M}_{21,1} \\
                                                                                                   \end{array}
                                                                                                 \right)
$, ${M}_{22,1} = \left(
                     \begin{array}{cc}
                       \tilde{M}_{3,1} & \tilde{M}_{4,1} \\
                      \tilde{ M}_{5,1} & \breve{M}_{22,1} \\
                     \end{array}
                   \right).
$

Write, for $|i|+ |j| \leq 2$,
\begin{eqnarray*}
{P}_1&=&  \sum_{k} p_{kij}y^i z^j e^{\sqrt{-1} \langle k, x\rangle},\\
R_1&=&  \sum_{|k|\leq K_1} p_{kij}y^i z^j e^{\sqrt{-1} \langle k, x\rangle},\\
{{P}_1} - R_1 &=&  \sum_{|k|> K_1} p_{kij} y^i z^j e^{\sqrt{-1} \langle k, x\rangle},\\
\end{eqnarray*}
where $K_1$ is specified in Section $\ref{normal form}$.

Next, we are going to improve the order of ${{P}_1}$ by the symplectic transformation $\Phi_{F_1}^1$, the time$-1$ map generated by the vector field $J \nabla F_1$ with $J = \left(
      \begin{array}{cccc}
        0 & I_{m} & 0 & 0 \\
        -I_{m} & 0 & 0 & 0 \\
        0 & 0 & 0 & I_{m_0} \\
        0 & 0 & -I_{m_0} & 0 \\
      \end{array}
    \right)$,
    where $F_1(x,y,z,\lambda) = \sum\limits_{0<|k|\leq K_1} f_{kij} y^i z^j e^{\sqrt{-1} \langle k, x\rangle}$ that satisfies
\begin{eqnarray}\label{tongdiao31}
\{N_1,F_1\} + \delta^2 (R_1- [R_1]) - R_1' = 0,\\
\nonumber R_1' = \partial_z h_1 J \partial_z F_1+  \langle y, {M}_{12,1}J {\Delta}_0 \rangle+ \langle z, {M}_{22,1}J \Delta_0 \rangle,\\
\nonumber \Delta_0 = \langle y, \partial_z f_{k20} y\rangle + \langle y, \partial_z f_{k11} z\rangle+ \langle z, \partial_z f_{k02}z\rangle,\\
\nonumber~[R_1](y,z,\lambda,\varepsilon)= \int_{T^m} {R_1}(x,y,z,\lambda,\varepsilon)dx.
\end{eqnarray}

Using (\ref{tongdiao31}) and comparing coefficients, we obtain the following quasilinear homological equations
\begin{eqnarray}
 {L}_{k0, 1} f_{k00} &=& p_{k00},\\
 {L}_{k1, 1 } \left(
         \begin{array}{c}
           f_{k10} \\
           f_{k01} \\
         \end{array}
       \right) &=& \left(
         \begin{array}{c}
           p_{k10} \\
           p_{k01} \\
         \end{array}
       \right) +\delta \left(
         \begin{array}{c}
            {M}_{12,1} J \partial_z f_{k00} \\
            {M}_{22,1} J \partial_z f_{k00} \\
         \end{array}
       \right),\\
 {L}_{k2,1} \left(
         \begin{array}{c}
           T (f_{k20})\\
           T (f_{k11})\\
           T(f_{k02}) \\
         \end{array}
       \right) &=& \left(
                   \begin{array}{c}
                     T(p_{k20}) \\
                     T(p_{k11}) \\
                     T(p_{k02}) \\
                   \end{array}
                 \right) + \delta \left(
                               \begin{array}{c}
                                 T({M}_{12,1} J \partial_z (f_{k10})^T) \\
                                 T(\check{M}_{1}) \\
                                 T({M}_{22,1} J \partial_z(f_{k01})^T) \\
                               \end{array}
                             \right),~~~~~~~
\end{eqnarray}
where $\check{M}_1 = {M}_{12,1} J \partial_z (f_{k01})^T + ({M}_{22,1} J \partial_z (f_{k00})^T)^T$, which are uniquely solvable on the following domain
\begin{eqnarray*}
\Lambda_1 &=& \{\lambda\in \Lambda_0: |\breve{L}_{k0,1}| > \frac{\gamma_1}{|k|^\tau}, \breve{L}_{k1,1}^* \breve{L}_{k1,1} > \frac{\gamma_1}{|k|^\tau}  I_{m + 2m_0},  \\
&~&\breve{L}_{k2,1}^* \breve{L}_{k2,1} > \frac{\gamma_1}{|k|^\tau}  I_{m^2 + 2m m_0+ 4m_0^2},  ~for~all~0< |k|\leq K_1\}.
\end{eqnarray*}
By (\ref{tongdiao31}), we have
\begin{eqnarray*}
  \bar{H}_2= H_1 \circ \Phi _{F_1}^1 = N_2(y,u,v,\lambda) + \delta^2 \bar{P}_2 (x,y,u,v,\lambda,\varepsilon),
\end{eqnarray*}
where
\begin{eqnarray*}
N_2 &=& N_1 + \delta^2 [{R}_1],\\
\bar{P}_2 &=& \frac{1}{\delta^2}(R_1' + \int_0^1 \{R_{1,t}, F_1\}\circ \Phi_{F_1}^t dt +\delta^2 ({\bar{P}_1} - R_1) \circ\Phi_{F_1}^1),\\
R_{1,t} &=&  t\delta^2 R_1 + (1-t)R_1' + (1-t)\delta^2 [R_1].
\end{eqnarray*}
It is easy to see that $[{R}_1]$ has critical point on $u$, due to the $T^{m_0}-$periodicity in $u$.
Consider the following transformation
$$\phi: ~~x \rightarrow x, ~~ y \rightarrow y+ y_0, ~~v \rightarrow v+ v_0,~~u\rightarrow u,$$
where $y_0$ and $v_0$ are determined by the following equation$:$
\begin{eqnarray*}
\delta \breve{M}_1 \left(
           \begin{array}{c}
             y_0 \\
             v_0 \\
           \end{array}
         \right) + \delta^2 \left(
                              \begin{array}{c}
                                \partial_y h(y_0, v_0) \\
                                \partial_v h(y_0, v_0) \\
                              \end{array}
                            \right) = \delta^2 \left(
                                                 \begin{array}{c}
                                                   \partial_y [{R}_1] \\
                                                    \partial_v [{R}_1]\\
                                                 \end{array}
                                               \right).
\end{eqnarray*}
Here and below, denote $[R_i]_2 = O(|\left(
               \begin{array}{c}
                 y \\
                 u\\
                 v \\
               \end{array}
             \right)
|^2).$ Then
\begin{eqnarray}\label{jiayou}
 H_2&=& N_2(y,u,v,\lambda) + \delta^2 {P}_2 (x,y,u,v,\lambda,\varepsilon),
\end{eqnarray}
where
\begin{eqnarray*}
N_2 &=& \langle\omega_2, y\rangle + \frac{\delta}{2} \langle \left(
                                                                                   \begin{array}{c}
                                                                                     y \\
                                                                                     v \\
                                                                                   \end{array}
                                                                                 \right)
, \breve{M}_2 \left(
           \begin{array}{c}
             y \\
             v \\
           \end{array}
         \right)
\rangle+ \delta^2h_2+ \delta^2 [{R}_1]_2,\\
\omega_2 & =& \omega+\delta \breve{M}_1 \left(
                        \begin{array}{c}
                          y_0 \\
                          v_0 \\
                        \end{array}
                      \right) + \delta^2 \left(
                                           \begin{array}{c}
                                             \partial_y h (y_0, v_0) \\
                                             \partial_v h (y_0, v_0) \\
                                           \end{array}
                                         \right)+ \delta^2 \left(
                                                             \begin{array}{c}
                                                               \partial_y [{R}_1] \\
                                                               \partial_v [{R}_1] \\
                                                             \end{array}
                                                           \right),\\
\breve{M}_2 &=&\breve{ M}_1 + \delta^2 \partial_{(y,v)}^2  h_1,\\
h_2 &=& O(|K_0 \left(
                                                                       \begin{array}{c}
                                                                         y \\
                                                                         v \\
                                                                       \end{array}
                                                                     \right)|^3),\\
P_2&=&\bar{P}_2\circ\phi + \langle \left(
                                     \begin{array}{c}
                                       y \\
                                       v \\
                                     \end{array}
                                   \right), \partial_{(y,v)}^2 [{R}_1] \left(
                                                                             \begin{array}{c}
                                                                               y_0 \\
                                                                               v_0 \\
                                                                             \end{array}
                                                                           \right)
\rangle.
\end{eqnarray*}
Moreover,
 \begin{eqnarray*}
 | P_2 |\leq c \delta^{\frac{89}{48}}.
 \end{eqnarray*}

Here and below, we denote $c$ the positive constant independent of the iteration process. Generally, the $\kappa-$th KAM step state as follows, where $\kappa$ is a given constant. After $\kappa$ KAM steps, we get
\begin{eqnarray}\label{kappa-1}
 H_\kappa&=& N_\kappa(y,u,v,\lambda) + \delta^2 {P}_\kappa (x,y,u,v,\lambda,\varepsilon),\\
\nonumber N_\kappa &=& \langle\omega_\kappa, y\rangle + \frac{\delta}{2} \langle \left(
                                                                                   \begin{array}{c}
                                                                                     y \\
                                                                                     v \\
                                                                                   \end{array}
                                                                                 \right)
, \breve{M}_\kappa \left(
           \begin{array}{c}
             y \\
             v \\
           \end{array}
         \right)
\rangle + \delta^2h_\kappa+ \delta^2 [{R}_1]_2+ \cdots+\delta^2 [{R}_{\kappa}]_2,
\end{eqnarray}
Denote $\breve{M}_\kappa = \left(
                      \begin{array}{cc}
                        \breve{M}_{11,\kappa} & \breve{M}_{12,\kappa} \\
                        \breve{M}_{21,\kappa} & \breve{M}_{22,\kappa} \\
                      \end{array}
                    \right)
,$ where $\breve{M}_{11,\kappa},$ $\breve{M}_{12,\kappa}$, $\breve{M}_{21,\kappa},$ $\breve{M}_{22,\kappa}$ are $m\times m,$ $m\times m_0$, $m_0\times m,$ $m_0\times m_0$ matrices, respectively. Let $\tilde{M}_\kappa = \left(
                               \begin{array}{ccc}
                                 \breve{M}_{11,\kappa} & \tilde{M}_{1,\kappa} & \breve{M}_{12,\kappa} \\
                                \tilde{M}_{2,\kappa}  & \tilde{M}_{3,\kappa} & \tilde{M}_{4,\kappa} \\
                                 \breve{M}_{21,\kappa} & \tilde{M}_{5,\kappa} & \breve{M}_{22,\kappa} \\
                               \end{array}
                             \right)
$, where $\tilde{M}_{1,\kappa} =0$, $\tilde{M}_{2,\kappa} =0,$  $\tilde{M}_{3,\kappa}=0, $ $\tilde{M}_{4,\kappa}= 0,$ $\tilde{M}_{5,\kappa} =0$ with obvious dimension. Let $M_\kappa = \left(
                  \begin{array}{cc}
                    M_{11,\kappa} & M_{12,\kappa}  \\
                    M_{21,\kappa} & M_{22,\kappa}  \\
                  \end{array}
                \right),
$ where $M_{11, \kappa} = \breve{M}_{11,\kappa} + \delta \partial_y^2 ( [{R}_1]_2+ \cdots+ [{R}_{\kappa}]_2),$ $M_{12, \kappa} = (\tilde{M}_{1,\kappa}, \breve{M}_{12,\kappa}) + \delta \partial_y \partial_z ( [{R}_1]_2+ \cdots+ [{R}_{\kappa}]_2),$ $M_{21, \kappa} = \left(
                                                                                                                                                         \begin{array}{c}
                                                                                                                                                           \tilde{M}_{1,\kappa} \\
                                                                                                                                                           \breve{M}_{12,\kappa} \\
                                                                                                                                                         \end{array}
                                                                                                                                                       \right)
 + \delta\partial_z \partial_y ( [{R}_1]_2+ \cdots+ [{R}_{\kappa}]_2),$ $M_{22, \kappa} = \left(
                                                                                              \begin{array}{cc}
                                                                                                \tilde{M}_{3,\kappa} & \tilde{M}_{4,\kappa} \\
                                                                                                \tilde{M}_{5,\kappa} &  \breve{M}_{22,\kappa} \\
                                                                                              \end{array}
                                                                                            \right)
 + \delta \partial_z^2 ( [{R}_1]_2+ \cdots+ [{R}_{\kappa}]_2).$ Rewrite $(\ref{kappa-1})$ as follows:
 \begin{eqnarray}\label{kappa}
 H_\kappa&=& N_\kappa(y,u,v,\lambda) + \delta^2 {P}_\kappa (x,y,u,v,\lambda,\varepsilon),\\
\nonumber N_\kappa &=& \langle\omega_\kappa, y\rangle + \frac{\delta}{2} \langle \left(
                                                                                   \begin{array}{c}
                                                                                     y \\
                                                                                     z \\
                                                                                   \end{array}
                                                                                 \right)
, M_\kappa \left(
           \begin{array}{c}
             y \\
             z \\
           \end{array}
         \right)
\rangle + \delta^2h_\kappa.
\end{eqnarray}

Write, for $|i|+ |j| \leq 2$,
\begin{eqnarray*}
{P}_\kappa&=&  \sum_{k} p_{kij}y^i z^j e^{\sqrt{-1} \langle k, x\rangle},\\
\label{R_1} R_\kappa&=& \sum_{|k|\leq K_{\kappa}} p_{kij}y^i z^j e^{\sqrt{-1} \langle k, x\rangle},\\
\label{I_1}{{P}_\kappa} - R_\kappa &=&  \sum_{|k|> K_{\kappa}} p_{kij}y^i z^j e^{\sqrt{-1} \langle k, x\rangle}.\\
\end{eqnarray*}
Improve the order of ${{P}_\kappa}$ by the symplectic transformation $\Phi_{F_\kappa}^1$, where
 \begin{eqnarray} \label{F}
 F_\kappa(x,y,z,\lambda) = \sum\limits_{\substack{0<|k|\leq K_{\kappa}\\|i|+ |j| \leq2}} f_{kij}y^i z^j e^{\sqrt{-1} \langle k, x\rangle}
 \end{eqnarray}
that satisfies
\begin{eqnarray}\label{tongdiao511}
\{N_\kappa,F_\kappa\} + \delta^2 (R_\kappa- [R_\kappa] )- R_\kappa' = 0,
\end{eqnarray}
\begin{eqnarray*}
R_\kappa' &=& \partial_z h_\kappa J \partial_z F_\kappa +  \langle y, {M}_{12, \kappa}J  \Delta_0 \rangle+ \langle z, {M}_{22,\kappa}J \Delta_0 \rangle,\\
\nonumber \Delta_0 &=& \langle y, \partial_z f_{k20} y\rangle + \langle y, \partial_z f_{k11} z\rangle+ \langle z, \partial_z f_{k02}z\rangle,\\
~[R_i]&=& \int_{T^m} {R_i}(x,y,z,\lambda,\varepsilon)dx, ~~~~1\leq i \leq \kappa.
\end{eqnarray*}

Using (\ref{tongdiao511}) and comparing coefficients, we obtain the following nonlinear homological equations
\begin{eqnarray}
 {L}_{k0, \kappa} f_{k00} &=& p_{k00},\\
 {L}_{k1, \kappa} \left(
         \begin{array}{c}
           f_{k10} \\
           f_{k01} \\
         \end{array}
       \right) &=& \left(
         \begin{array}{c}
           p_{k10} \\
           p_{k01} \\
         \end{array}
       \right) +\delta \left(
         \begin{array}{c}
            {M}_{12, \kappa} J \partial_z f_{k00} \\
            {M}_{22, \kappa} J \partial_z f_{k00} \\
         \end{array}
       \right),\\
 {L}_{k2, \kappa} \left(
         \begin{array}{c}
           T (f_{k20})\\
           T (f_{k11})\\
           T(f_{k02}) \\
         \end{array}
       \right) &=& \left(
                   \begin{array}{c}
                     T(p_{k20}) \\
                     T(p_{k11}) \\
                     T(p_{k02}) \\
                   \end{array}
                 \right) + \delta \left(
                               \begin{array}{c}
                                 T({M}_{12, \kappa} J \partial_z (f_{k10})^T) \\
                                 T(\check{M}_\kappa) \\
                                 T({M}_{22, \kappa} J \partial_z(f_{k01})^T) \\
                               \end{array}
                             \right),~~~~~~
\end{eqnarray}
where $\check{M}_\kappa = {M}_{12, \kappa} J \partial_z (f_{k01})^T + ({M}_{22, \kappa} J \partial_z (f_{k00})^T)^T$, which are uniquely solvable on the following domain
\begin{eqnarray*}
\Lambda_{\kappa} &=& \{\lambda\in \Lambda_{\kappa-1}: |\breve{L}_{k0, \kappa}| > \frac{\gamma_\nu}{|k|^\tau}, \breve{L}_{k1, \kappa}^* \breve{L}_{k1, \kappa} > \frac{\gamma_\nu}{|k|^\tau}  I_{m + 2m_0},  \\
&~&\breve{L}_{k2, \kappa}^* \breve{L}_{k2, \kappa} > \frac{\gamma_\nu}{|k|^\tau}  I_{m^2 + 2m m_0+ 4m_0^2},  ~for~all~0< |k|\leq K_{\kappa}\}.
\end{eqnarray*}
Let $u_0$ be the critical point of $\tilde{[R]} = [{R}_1]_2+ \cdots+ [{R}_{\kappa+1}]_2$. Consider the following transformation
$$\phi: ~~x \rightarrow x, ~~ y \rightarrow y+ y_0, ~~v \rightarrow v+ v_0,~~u\rightarrow u,$$
where $y_0$ and $v_0$ are determined by the following equation$:$
\begin{eqnarray*}
\delta \breve{M}_\kappa \left(
           \begin{array}{c}
             y_0 \\
             v_0 \\
           \end{array}
         \right) + \delta^2 \left(
                              \begin{array}{c}
                                \partial_y h(y_0, v_0) \\
                                \partial_v h(y_0, v_0) \\
                              \end{array}
                            \right) = \delta^2 \left(
                                                 \begin{array}{c}
                                                   \partial_y [{R}_1] \\
                                                    \partial_v [{R}_1]\\
                                                 \end{array}
                                               \right).
\end{eqnarray*}

 Then
\begin{eqnarray*}
  H_{\kappa+1}&=& H_\kappa \circ \Phi _{F_\kappa}^1 \circ\phi = N_{\kappa+1}(y,u,v,\lambda) + {P}_{\kappa+1} (x,y,u,v,\lambda,\varepsilon),
\end{eqnarray*}
where
\begin{eqnarray*}
N_{\kappa+1} &=& \langle\omega_{\kappa+1}, y\rangle + \frac{\delta}{2} \langle \left(
                                                                                   \begin{array}{c}
                                                                                     y \\
                                                                                     v \\
                                                                                   \end{array}
                                                                                 \right)
, {M}_{\kappa+1} \left(
           \begin{array}{c}
             y \\
             v \\
           \end{array}
         \right)
\rangle+ \delta^2h_\kappa, \\
\omega_{\kappa+1} & =& \omega_\kappa+\delta \breve{M}_\kappa \left(
                        \begin{array}{c}
                          y_0 \\
                          v_0 \\
                        \end{array}
                      \right) + \delta^2 \left(
                                           \begin{array}{c}
                                             \partial_y h_\kappa (y_0, v_0) \\
                                             \partial_v h_\kappa (y_0, v_0) \\
                                           \end{array}
                                         \right)+ \delta^2 \left(
                                                             \begin{array}{c}
                                                               \partial_y [{R}_\kappa] \\
                                                               \partial_v [{R}_\kappa] \\
                                                             \end{array}
                                                           \right),\\
{M}_{\kappa+1} &=& {M}_\kappa + \delta^2 \partial_{(y,z)}^2  h_\kappa+ \partial_{(y,z)}^2 \tilde{[R]},\\
{P}_{\kappa+1} &=& R_\kappa'\circ\phi + \int_0^1 \{R_{\kappa,t}, F_\kappa\}\circ \Phi_{F_\kappa}^t \circ\phi dt +({P_\kappa} - R_\kappa) \circ\Phi_{F_\kappa}^1\circ\phi\\
&~&+\langle \left(
                                     \begin{array}{c}
                                       y \\
                                       u \\
                                       v \\
                                     \end{array}
                                   \right), \partial_{(y,z)}^2 \tilde{[{R}]} \left(
                                                                             \begin{array}{c}
                                                                               y_0 \\
                                                                               0 \\
                                                                               v_0
                                                                             \end{array}
                                                                           \right)
\rangle,\\
R_{\kappa,t} &=&  t R_\kappa + (1-t)R_\kappa' + (1-t)[R_\kappa].
\end{eqnarray*}
Hence
 \begin{eqnarray*}
 | {P}_{\kappa+1}|\leq c \delta^{\frac{1}{2} + \frac{5}{4}(\frac{13}{12})^{\kappa+1} },~|l| \leq d.
 \end{eqnarray*}

Therefore, after $\kappa$ KAM steps, the new Hamiltonian reads as
 \begin{eqnarray} \label{youyong}
 H_{\kappa+1}  = N_{\kappa+1}+ \delta^{2} P_{\kappa+1},
\end{eqnarray}
where
\begin{eqnarray*}
N_{\kappa+1} &=& \langle\omega_{\kappa+1}, y\rangle + \frac{\delta}{2} \langle \left(
                                                                                   \begin{array}{c}
                                                                                     y \\
                                                                                     v \\
                                                                                   \end{array}
                                                                                 \right)
, \breve{M}_{\kappa+1} \left(
           \begin{array}{c}
             y \\
             v \\
           \end{array}
         \right)
\rangle+ \delta^2h_{\kappa+1} \\
&~&+ \delta^2 [{R}_1]_2+ \delta^2 [{R}_2]_2+ \cdots+  \delta^2 [{R}_\kappa]_2 ,\\
\omega_{\kappa+1} & =& \omega_\kappa+\delta \breve{M}_\kappa \left(
                        \begin{array}{c}
                          y_0 \\
                          v_0 \\
                        \end{array}
                      \right) + \delta^2 \left(
                                           \begin{array}{c}
                                             \partial_y h_\kappa (y_0, v_0) \\
                                             \partial_v h_\kappa (y_0, v_0) \\
                                           \end{array}
                                         \right)+ \delta^2 \left(
                                                             \begin{array}{c}
                                                               \partial_y [{R}_\kappa] \\
                                                               \partial_v [{R}_\kappa] \\
                                                             \end{array}
                                                           \right),\\
\breve{M}_{\kappa+1} &=& \breve{M}_\kappa + \delta^2 \partial_{(y,v)}^2  h_\kappa.
\end{eqnarray*}

Let
\begin{eqnarray*}
\bar{g}&=&\delta^2 [{R}_1]_2+ \delta^2 [{R}_2]_2 +\cdots +\delta^{2} [{R}_{\kappa}]_2\\
 &=& \delta^{\frac{185}{48}} [\bar{R}_1]_2+ \delta^{\frac{2285}{576}} [\bar{R}_2]_2 +\cdots +\delta^{\frac{5}{2}+ \frac{5}{4}(\frac{13}{12})^{\kappa+1} } [\bar{R}_{\kappa}]_2\\
 &=&\sum\limits_{j_1} \delta^{\frac{185}{48}+ j_1} [\bar{R}_1]_2^{(j_1)}+ \sum\limits_{j_2}\delta^{\frac{2285}{576}+ j_2} [\bar{R}_2]_2^{(j_2)} +\cdots\\
 &~& +\sum\limits_{j_\kappa}\delta^{\frac{5}{2}+ \frac{5}{4}(\frac{13}{12})^{\kappa+1} + j_\kappa} [\bar{R}_{\kappa}]_2^{(j_\kappa)}.
\end{eqnarray*}

\begin{definition}\label{HOP}
If the following two hold:

(1)At critical points of $\bar{g}$, $(y_0, u_0, v_0)$,
\begin{eqnarray*}
~\det~ \partial_u^2 \delta^{-a+1} \bar{g} = 0;
\end{eqnarray*}

(2) At critical points of $\bar{g}$, $(y_0, u_0, v_0)$, there is a constant $\bar{\sigma}_0 >0$, such that
\begin{eqnarray*}
|\det \partial_u^2 ~\delta^{-a} \bar{g} |\geq \bar{\sigma}_0,
\end{eqnarray*}
then $\bar{g}$ is called $a-$order nondegenerate at $(y_0, u_0, v_0)$.
\end{definition}

\begin{remark}
Since $\tilde{P}(x,y,u,v)$ is $\kappa-$order nondegenerate, at relative critical point $(y_0, u_0, v_0)$ $\det \partial_u^2 [\tilde{P}_{\kappa}]( y_0,u_0,v_0,0) \neq 0$, which implies that $\bar{g}$ is $a-$order nondegenerate, where $0< a\leq \kappa$. And since $\bar{g}$ is $T^{m_0}$ periodic in $u$, it has $2^{m_0}$ critical points via the high order nondegeneracy and Morse theory \emph{(\cite{Milnor})}.
\end{remark}

\begin{remark}
Assumption $\emph{(2)}$ in definition \ref{HOP} is equivalent to the following $\bf{(\mathfrak{S}1)}$.

\begin{itemize}
\item[$\bf{(\mathfrak{S}1)}$] At critical point of $\bar{g}$, $(y_0, u_0, v_0)$, there exists a constant $ c > 0$ such that the minimum $\lambda_{min}^{\varepsilon} (\omega)$ among absolute values of all eigenvalues of $\partial_u^2 \bar{g}$ satisfies $|\lambda_{min}^{\varepsilon} |\geq c \varepsilon ^a $ for all $\omega \in \Lambda(g,G)$.
\end{itemize}

\end{remark}

At the critical point of $\bar{g}$, $(y_0,u_0,v_0)$, rewrite Hamiltonian system (\ref{youyong}) as follows
\begin{eqnarray}\label{new221}
 H(x,y,u,v)&=& N(y,u,v) +{\delta^{a +1}} \tilde{P}(x,y,u,v,\varepsilon),
\end{eqnarray}
where
\begin{eqnarray*}
N&=&\langle\omega_{\kappa+1}, y\rangle + \frac{\delta}{2} \langle \left(
                                                                                   \begin{array}{c}
                                                                                     y \\
                                                                                     v \\
                                                                                   \end{array}
                                                                                 \right)
, \breve{M}_{\kappa+1} \left(
           \begin{array}{c}
             y \\
             v \\
           \end{array}
         \right)
\rangle+ \delta^2h_\kappa + \frac{\delta^{a}}{2} \langle u, Vu\rangle + \delta^{a}O(|u|^3),\\
\delta^{a + 1} \tilde{P}&=& \delta^{\kappa + 1} P(x,y,u,v,\varepsilon)+ O(\delta^{a+1}),
\end{eqnarray*}
 $x \in T^m$, $y \in R^m$, $u, v \in R ^{m_0}$, $1 \leq a \leq \kappa$. In the above, all $\lambda-$dependence is of class $C^{l_0}$ for some $l_0 \geq d$.

Next we should raise the order of $\tilde{P}$ by performing finite times quasilinear KAM steps. Let $\tilde{\tau}$ be the smallest integer such that $[\frac{5}{2}+ \frac{5}{4}(\frac{13}{12})^{\tilde{\tau}}  ]\geq\frac{3a+1}{2}$, where $a$ is a constant.  After $\tilde{\tau}$ KAM steps mentioned as above, at each critical point, we obtain the following
\begin{eqnarray}
\nonumber H_{\tilde{\tau}}(x,y,u,v) &=& \langle\omega_{\tilde{\tau}}, y\rangle + \frac{\delta}{2} \langle \left(
                                                                                   \begin{array}{c}
                                                                                     y \\
                                                                                     v \\
                                                                                   \end{array}
                                                                                 \right)
, \breve{M}_{\tilde{\tau}} \left(
           \begin{array}{c}
             y \\
             v \\
           \end{array}
         \right)
\rangle+ \delta^2h_{\tilde{\tau}}+ \frac{\delta^{a}}{2}{\langle u , V_{\tilde{\tau}_1}(\lambda) u\rangle} \\
\label{keyileba}&~&~ + \delta^{a}\hat{u}_{\tilde{\tau}_1}(u)+\delta ^{\frac{{3a +1}}{2}} \hat{P}(x,y,u,v,\delta),~~
\end{eqnarray}
up to a constant,  where
\begin{eqnarray*}
V_{\tilde{\tau}_1} &=& V+ \partial_u^2 \tilde{h},~~~~~\hat{u}_{\tilde{\tau}}(u) = \hat{u} + (\tilde{h} - \langle\partial_u^2 \tilde{h} u, u\rangle),\\
\tilde{h} &=& \delta^{\frac{5}{2}+ \frac{5}{4}(\frac{13}{12})^{\kappa+1} + j_\kappa} [\bar{R}_{\kappa+1}]+\cdots+ \delta^{\frac{5}{2}+ \frac{5}{4}(\frac{13}{12})^{\tilde{\tau}} + j_\kappa} [\bar{R}_{\tilde{\tau}}],\\
\hat{P} &=& \delta P(x,y,u,v,\delta), ~~~~~1\leq a\leq\kappa,
\end{eqnarray*}
 with nonsingular $V_{\tilde{\tau}}$. But in each KAM step we have a similar hypothesis in form, $\delta K^{\tau+1} = o(\gamma)$. And the assumption obviously holds for finite times KAM steps. Consider re-scaling $x\rightarrow x$, $y\rightarrow \delta^{\frac{{a -1}}{2}}y$, $u\rightarrow u$, $v \rightarrow \delta^{\frac{{a -1}}{2}} v$, $H\rightarrow \delta^{\frac{{-a+1}}{2}}H$. Then the re-scaled Hamiltonian reads
\begin{eqnarray*}
H_{\tau_1}(x,y, u,v) &=& \langle\omega_{\tilde{\tau}}, y\rangle + \frac{\delta^{\frac{a+1}{2}}}{2} \langle \left(
                                                                                   \begin{array}{c}
                                                                                     y \\
                                                                                     v \\
                                                                                   \end{array}
                                                                                 \right)
, \breve{M}_{\tilde{\tau}} \left(
           \begin{array}{c}
             y \\
             v \\
           \end{array}
         \right)
\rangle+ \delta^2h_{\tilde{\tau}}\\
&~&~+ \frac{\delta^{\frac{a+1}{2}}}{2}{\langle u , V_{\tilde{\tau}}(\lambda)u\rangle}+ \delta^{\frac{a+1}{2}}\hat{u}_{\tilde{\tau}}(u)+\delta ^{a+1} \tilde{P}(x,y,u,v).
\end{eqnarray*}
Denote $\delta^{\frac{a+1}{2}} = \delta$.  Then we have
\begin{eqnarray}\label{model31}
 H(x,y,u,v) =N(y,u,v) + P(x,y,u,v),
\end{eqnarray}
with
\begin{eqnarray*}
N &=& \langle\omega_{\tilde{\tau}}, y\rangle + \frac{\delta}{2} \langle \left(
                                                                                   \begin{array}{c}
                                                                                     y \\
                                                                                     v \\
                                                                                   \end{array}
                                                                                 \right)
, \breve{M}_{\tilde{\tau}} \left(
           \begin{array}{c}
             y \\
             v \\
           \end{array}
         \right)
\rangle+ \delta^2h_{\tilde{\tau}}+ \frac{\delta}{2}{\langle u , V_{\tilde{\tau}} (\lambda)u\rangle}+ \delta \hat{u}_{\tilde{\tau}}(u),\\
 P&=&\delta^2 \tilde{P}(x,y,u,v),~~~~\hat{u}(u)= O(|u|^3),
\end{eqnarray*}
 where $x \in T^m$, $y \in R^m$, $u, v \in R ^{m_0}$. In the above, all $\lambda-$dependence is of class $C^{l_0}$ for some $l_0 \geq d$.

Applying Theorem \ref{shengluede} to (\ref{model31}), the system admits a family of invariant tori. By Morse theory, there are $2^{m_0}$ critical points, and consequently it has $2^{m_0}$ families of resonant torus. This completes the proof of Theorem \ref{dingli11}.

\section{Example}\label{example}
Here we give two examples to show how the program mentioned in section $\ref{074}$ work.
\begin{example} \label{071}
Consider the following Hamiltonian system
\begin{eqnarray}
\nonumber H(\tilde{x}, \tilde{y}) &=& \langle \tilde{\omega}, \tilde{y} \rangle+ \frac{\varepsilon}{2} \langle \tilde{y}, M
\tilde{y} \rangle + \varepsilon^3 \cos(-\frac{ x_2}{2})\\
\label{EQ41}&& + \varepsilon^2 \cos(-\frac{ x_2}{2})\sin (-2 x_1 + x_2) e^{-y_1 - 2y_2},
\end{eqnarray}
where $\tilde{x}= (x_1, x_2)^T$, $\tilde{y}= (y_1, y_2)^T$, $\tilde{\omega} = (\omega_1, 2\omega_1)^T$,  $x_1, x_2 \in T^1$, $y_1$, $y_2$ $ \in R^1$, $\omega_1 \in R\setminus \{0\}$ and $M = \left(
                                                  \begin{array}{cc}
                                                    \frac{1}{4} & 0 \\
                                                    0 & 0 \\
                                                  \end{array}
                                                \right)$.

Consider transformation $\tilde{\phi}_g: \left(
                                                                                                                                                                        \begin{array}{c}
                                                                                                                                                                          y_1  \\
                                                                                                                                                                          y_2 \\
                                                                                                                                                                        \end{array}
                                                                                                                                                                      \right)
\mapsto \left(
          \begin{array}{c}
            y \\
            v \\
          \end{array}
        \right), \left(
                                     \begin{array}{c}
                                       x_1 \\
                                       x_2 \\
                                     \end{array}
                                   \right)
\mapsto \left(
          \begin{array}{c}
            x \\
            u \\
          \end{array}
        \right),
$
where $\left(
          \begin{array}{c}
            y_1 \\
            y_2 \\
          \end{array}
        \right) = \left(
                    \begin{array}{cc}
                      -2 & 0 \\
                      1 & -\frac{1}{2} \\
                    \end{array}
                  \right) \left(
                            \begin{array}{c}
                              v \\
                              y \\
                            \end{array}
                          \right)$, $\left(
          \begin{array}{c}
            x_1 \\
            x_2 \\
          \end{array}
        \right) =
\left(
          \begin{array}{cc}
            -\frac{1}{2} & -1 \\
            0 & -2 \\
          \end{array}
        \right) \left(
                  \begin{array}{c}
                    x \\
                    u \\
                  \end{array}
                \right)$. Denote $\omega = -\omega_1$. Then Hamiltonian (\ref{EQ41}) is changed to
\begin{eqnarray*}
H(x,y,u,v) &=& \omega y + \frac{\varepsilon}{2}v^2 + \varepsilon^3 \cos u + \varepsilon^2~ \cos u ~ \sin x ~e^y,
\end{eqnarray*}
which means that previous works do not apply to this system, since, first, $\left(
                                                  \begin{array}{cc}
                                                    \frac{1}{4} & 0 \\
                                                    0 & 0 \\
                                                  \end{array}
                                                \right)$ is degenerate and, second, the perturbation $P_1 = \varepsilon^3 \cos u + \varepsilon^2~ \cos u ~ \sin x ~e^y$ is $2-$order nondegenerate perturbation.

Next, we will improve the order of ${{P}_1}$ by the symplectic transformation $\Phi_{F_1}^1$, where $ F_1(x,y,u,v) =a_1(y,u,v) \sin x + b_1(y,u,v) \cos x$
satisfies
\begin{eqnarray}\label{063}
\{N,F_1\} + P_1- [P_1] - P_1' = 0,
\end{eqnarray}
\begin{eqnarray*}
P_1' &=& \partial_u N \partial_v F_1 - \partial_v N \partial_u F_1,\\
N &=& \omega y + \frac{\varepsilon}{2}v^2.
\end{eqnarray*}
Take $F_1(x,y,u) = \frac{- \varepsilon^2 ~ \cos u ~ e^y ~ \cos x}{\omega}.$ Then
\begin{eqnarray*}
H_2(x,y,u,v) = N_2(y,u) + P_1'(x,y,u,v,\varepsilon)+ \int_0^1 \{(1-t)\{N, F_1\}+ P_1, F_1\}\circ \phi_{F_1}^t dt,
\end{eqnarray*}
where
\begin{eqnarray*}
N_2(y,u) &=& N(y,u) + \varepsilon^3 \cos u,\\
P_1'(x,y,u,v) &=& \frac{- \varepsilon^3 ~ v~ \sin u~ e^y~ \cos x}{\omega},\\
P_2&=&\int_0^1 \{(1-t)\{N, F_1\}+ P_1, F_1\}\circ \phi_{F_1}^t dt = O(\varepsilon^4).
\end{eqnarray*}
In fact,
\begin{eqnarray*}
R_t &=& (1-t) \{N, F_1\} + P_1\\
&=& (1-t) (- \varepsilon^2 \cos u e^y \sin x - \frac{\varepsilon^3 v\sin u e^y \cos x}{\omega})\\
&~&+ \varepsilon^3 \cos u + \varepsilon^2 \cos u \sin x e^y,\\
\{R_t, F_1\} &=& \frac{\partial R_t}{\partial x} \frac{\partial F_1}{\partial y} - \frac{\partial R_t }{\partial y} \frac{\partial F_1}{\partial x} + \frac{\partial R_t}{\partial u} \frac{\partial F_1}{\partial v} - \frac{\partial R_t }{\partial v} \frac{\partial F_1}{\partial u}\\
&=& \frac{\varepsilon^4 t \cos^2 u e^{2y} \cos 2x}{\omega} + \varepsilon^5 (1-t) \frac{\sin^2 u e^{2y}(\cos 2x +1)}{2 \omega^2}.
\end{eqnarray*}
Let $F_2 = \frac{ \varepsilon^3~ v~ \sin u ~e^y~ \sin x}{\omega^2}.$
Then $\{N_2, F_2\} + P_2 - [P_2] - P_2' = 0,$
where $P_2' = \partial _u N_2 \partial_v F_2 - \partial_v N_2 \partial_u F_2.$
With the help of $\Phi_{F_2}^1$, we have
\begin{eqnarray*}
H_3(x,y,u,v) = N_2(y,u) + P_3(x,y,u,v),
\end{eqnarray*}
where $N_2 =\omega y + \frac{\varepsilon}{2}v^2 + \varepsilon^3 \cos u,$ $P_3= O(\varepsilon^4).$
Therefore, using $\textbf{Theorem \emph{\ref{shengluede}}}$, there are  two families of invariant tori for the Hamiltonian (\ref{EQ41}) associated with relative critical points $(y, u,v) = (y_0,0,0),~ (y_0,\pi,0).$
\end{example}

\begin{example}\label{072}
Consider the following Hamiltonian system
\begin{eqnarray}
\nonumber H(x, y) &=& \langle \tilde{\omega}, \tilde{y} \rangle + \frac{\varepsilon}{2} \langle \tilde{y}, M \tilde{y} \rangle + \varepsilon^4 \cos (-\frac{x_2}{2} + \frac{\iota \pi}{4})\\
\label{EQ44}&&+ \varepsilon^2 \sin (-\frac{x_2}{2}) \sin (-2 x_1 + x_2) e^{-y_1 - 2 y_2},
\end{eqnarray}
where $\tilde{\omega} = (\omega_1, \omega_2)^T$, $\tilde{x} = (x_1, x_2)^T$, $\tilde{y} = (y_1, y_2)^T$, $x_1, x_2\in T^1$, $y_1$, $y_2$  $\in R^1$, $\omega_1 \in R^1\setminus \{0\}$ and $M = \left(
                                                   \begin{array}{cc}
                                                     \frac{1}{4} & 0 \\
                                                     0 & 0 \\
                                                   \end{array}
                                                 \right)$.

Consider the following transformation: $\left(
                                          \begin{array}{c}
                                            y_1 \\
                                            y_2 \\
                                          \end{array}
                                        \right) = \left(
                                                    \begin{array}{cc}
                                                      -2 & 0 \\
                                                      1 & -\frac{1}{2} \\
                                                    \end{array}
                                                  \right)\left(
                                                           \begin{array}{c}
                                                             v \\
                                                             y \\
                                                           \end{array}
                                                         \right)
$, $\left(
      \begin{array}{c}
        x_1 \\
        x_2 \\
      \end{array}
    \right) = \left(
                \begin{array}{cc}
                  -\frac{1}{2} & -1 \\
                  0 & -2 \\
                \end{array}
              \right) \left(
                        \begin{array}{c}
                          x \\
                          u \\
                        \end{array}
                      \right).
$
Denote $\omega = -\omega_1$. Hamiltonian system (\ref{EQ44}) is transformed to
\begin{eqnarray*}
H(x,y,u,v) = \omega y + \frac{\varepsilon}{2} v^2 + \varepsilon^4 \cos (u + \frac{\iota \pi}{4}) + \varepsilon^2 \sin u \sin x e^y,
\end{eqnarray*}
which means the perturbation $P_1 = \varepsilon^4 \cos (u + \frac{\iota \pi}{4}) + \varepsilon^2 \sin u \sin x e^y$ is $3-$order nondegenerate, i.e. previous works do not apply to this system. Let us prove the persistence of resonant tori for Hamiltonian system (\ref{EQ44}) using $\textbf{Theorem \emph{\ref{shengluede}}}$.

Denote $F_1(x,y,u) = \frac{- \varepsilon^2  \sin u  e^y  \cos x}{\omega}.$
Then
\begin{eqnarray*}
\{N_1,F_1\} + P_1- [P_1] - P_1' = 0,
\end{eqnarray*}
where
\begin{eqnarray*}
P_1' &=& \partial_u N_1 \partial_v F_1 - \partial_v N_1 \partial_u F_1,\\
~[P_1]&=& \int_0^{2 \pi} {P_1}(x,y,u,\varepsilon)dx,\\
N_1 &=& \omega y + \frac{\varepsilon}{2}v^2.
\end{eqnarray*}
Therefore, under the symplectic transformation $\Phi_{F_1}^1$, we have
\begin{eqnarray*}
H_2(x,y,u,v) = N_2(y,u) + P_2'(x,u,v,\varepsilon)+ \bar{P}_3(x,y,u,v,\varepsilon),
\end{eqnarray*}
where
\begin{eqnarray*}
N_2(y,u) &=& \omega y + \frac{\varepsilon}{2}v^2 + \varepsilon^4 \cos (u + \frac{\iota\pi}{4}),\\
P_2'(x,u,v) &=& \frac{ \varepsilon^3 ~ v~ \cos u~ e^y~ \cos x}{\omega},\\
\bar{P}_3 &=& \int_0^1 \{R_t, F_1\}\circ \phi_{F_1}^t dt,\\
R_t &=& (1-t) \{N, F_1\}+ P_1.
\end{eqnarray*}
Moreover,
\begin{eqnarray*}
R_t &=& (1-t) \{N, F_1\} +P \\
&=& t \varepsilon^2 \sin u e^y \sin x + (1-t) \frac{\varepsilon^3 \nu \cos u e^y \cos x}{\omega} + \varepsilon^4 \cos (u+ \frac{\iota\pi}{4}),\\
\{R_t, F_1\} &=& \frac{\partial R_t}{\partial x} \frac{\partial F_1}{\partial y} - \frac{\partial R_t }{\partial y} \frac{\partial F_1}{\partial x} + \frac{\partial R_t}{\partial u} \frac{\partial F_1}{\partial v} - \frac{\partial R_t }{\partial v} \frac{\partial F_1}{\partial u}\\
&=& -(t \varepsilon^2 \sin u e^y \cos x- (1-t)\frac{\varepsilon^3 \nu \cos u e^y \sin x}{\omega} )\frac{\varepsilon^2 \sin u e^y \cos x}{\omega}\\
&~& - (t \varepsilon^2 \sin u e^y \sin x+ (1-t)\frac{\varepsilon^3 \nu \cos u e^y \cos x}{\omega} )\frac{\varepsilon^2 \sin u e^y \sin x}{\omega}\\
&~&+ (1- t) \frac{\varepsilon^3 \cos u e^y \cos x}{\omega} \frac{\varepsilon^2 \cos u e^y \cos x}{\omega}\\
&=&-\frac{t \varepsilon^4 \sin ^2 u e^{2y}}{\omega} + (1- t) \frac{\varepsilon^5 \cos^2 u e^{2y} \cos^2 x}{\omega^2}.
\end{eqnarray*}
Hence $| \bar{P}_3| = -\frac{ \varepsilon^4 \sin ^2 u e^{2y}}{2\omega} + O(\varepsilon^5).$
Set $F_2 = \frac{- \varepsilon^3 v \cos u e^y \sin x}{\omega^2}.$
Then
\begin{eqnarray*}
\{N_2, F_2\} + P_2 - [P_2] - P_2' = o(\varepsilon^3),
\end{eqnarray*}
where
\begin{eqnarray*}
[P_2] &=& \int_0^{2\pi} P_2(x,u,v) dx,\\
P_2' &=& \partial _u N_2 \partial_v F_2 - \partial_v N_2 \partial_u F_2\\
&=&-\frac{\varepsilon^7 \sin (u + \frac{\iota\pi}{4}) \cos u e^y \sin x }{\omega^2} + \varepsilon^4 \frac{v^2 \sin u  e^y \sin x }{\omega^2}.
\end{eqnarray*}
With the aid of $\Phi_{F_2}^1$, we have
\begin{eqnarray*}
H_3(x,y,u,v) = N_3(y,u) + P_3(x,y,u,v),
\end{eqnarray*}
where
\begin{eqnarray*}
N_3 &=& \omega y + \frac{\varepsilon}{2}v^2 + \varepsilon^4 \cos (u+\frac{\iota\pi}{4}) - \frac{\varepsilon^4 ~\sin^2u~e^{2y}}{2 \omega},\\
P_3&=&  \varepsilon^4 \frac{v^2 \sin u  e^y \sin x }{\omega^2}+ O(\varepsilon^5).
\end{eqnarray*}
Let $F_3 = \frac{ -\varepsilon^4~ v^2~ \sin u ~e^y~ \cos x}{\omega^3}.$ Then
\begin{eqnarray*}
\{N_3, F_3\} + P_3 - [P_3] - P_3' = 0,
\end{eqnarray*}
where
\begin{eqnarray*}
P_3' &=& \partial _u N_3 \partial_v F_3 - \partial_v N_3 \partial_u F_3\\
&=& \frac{\varepsilon^5 v^3 \cos u e^y \cos x}{\omega^3}\\
 &&+ \varepsilon^8 (\frac{2v \sin (u + \frac{\iota \pi}{4}) \sin u \cos x e^y \omega + 2v \sin^2u\cos ue^{3y}\cos x}{\omega^4}).
\end{eqnarray*}
With the help of $\Phi_{F_3}^1$, we have
\begin{eqnarray*}
H_4(x,y,u,v) = N_4(y,u) + P_4(x,y,u,v),
\end{eqnarray*}
where
\begin{eqnarray*}
N_4 =\omega y + \frac{\varepsilon}{2}v^2 + \varepsilon^4 \cos (u+\frac{\iota\pi}{4}) - \frac{\varepsilon^4 ~\sin^2u~e^{2y}}{2 \omega},~~~~P_4= O(\varepsilon^5).
\end{eqnarray*}
Therefore, using $\textbf{Theorem \emph{\ref{shengluede}}}$, there are  two families of invariant tori for the Hamiltonian $(\ref{EQ44})$.
\end{example}
\begin{remark}
The survival resonant tori are closely related to relative critical points. Relative critical points maybe drift when we do KAM iteration. In detail, both $\cos (u + \frac{\iota\pi}{4})$ and $g = - \cos (u + \frac{\iota\pi}{4}) - \frac{2\sin^2u e^{2y}}{2\omega}$ have two relative critical points. These critical points of $\cos (u + \frac{\iota\pi}{4})$ are $u = - \frac{\iota\pi+ 4\pi}{4} $ and $- \frac{\iota\pi+ 8\pi}{4}$, which are not relative critical point of $g$ when $\iota = 1$, since $\partial_u g(-\frac{\pi}{4}+ \pi) = \sqrt{2} \cos \pi =-1$ and $\partial_u g(-\frac{\pi}{4}+ 2\pi) = \sqrt{2} \cos 2\pi =1$.
\end{remark}

\section*{Acknowledgement}

We sincerely thank the anonymous referee for their most useful comments, which allowed us to vastly improve the exposition of our result.
\appendix
\section{Some Properties}\label{A}

\begin{pro}
 Coordinate transformation $\phi_g$: $I - I_0 = K_0 p$, $q = K_0 ^ T \theta$ is symplectic.
\end{pro}

\begin{proof} In fact,
\begin{eqnarray*}
\left(
  \begin{array}{c}
    \theta \\
   I - I_0 \\
  \end{array}
\right) = \left(
            \begin{array}{cc}
              (K_0^T)^{-1} & 0 \\
              0 & K_0 \\
            \end{array}
          \right)\left(
                    \begin{array}{c}
                      q \\
                      p \\
                    \end{array}
                  \right).
\end{eqnarray*}
Then
\begin{eqnarray*}
&~&\left(
  \begin{array}{cc}
    ((K_0^T)^{-1})^T & 0 \\
    0 & K_0^T \\
  \end{array}
\right)
\left(
  \begin{array}{cc}
    0 & I \\
    -I & 0 \\
  \end{array}
\right)
\left(
  \begin{array}{cc}
    (K_0^T)^{-1} & 0 \\
    0 & K_0 \\
  \end{array}
\right)\\
&=& \left(
  \begin{array}{cc}
    K_0^{-1} & 0 \\
    0 & K_0^T \\
  \end{array}
\right)
\left(
  \begin{array}{cc}
    0 & I \\
    -I & 0 \\
  \end{array}
\right)
\left(
  \begin{array}{cc}
    (K_0^T)^{-1} & 0 \\
    0 & K_0 \\
  \end{array}
\right)\\
&=& \left(
      \begin{array}{cc}
        0 & I \\
        -I & 0 \\
      \end{array}
    \right),
\end{eqnarray*}
which means that the coordinate transformation is symplectic.

\end{proof}

\begin{pro}
Transformation \emph{(\ref{Eq1p})} is symplectic.
\end{pro}

\begin{proof}
Let $p =\left(
   \begin{array}{c}
     y \\
     v \\
   \end{array}
   \right)$ and $q= \left(
             \begin{array}{c}
               x \\
               u \\
             \end{array}
           \right)$. With $p \rightarrow \varepsilon^{\frac{1}{4}} p$, $q\rightarrow q$, the motion equation of Hamiltonian system (\ref{qq}) is changed to $
\left\{
  \begin{array}{ll}
    \varepsilon^{\frac{1}{4}} \dot{p} = \frac{\partial H_1(\varepsilon^{\frac{1}{4}}p, q)}{\partial q},  \\
     \dot{q} = - \frac{\partial H_1(\varepsilon^{\frac{1}{4}}p, q)}{\varepsilon^{\frac{1}{4}}\partial p}.
  \end{array}
\right.$
Since $H(p, q) = \varepsilon^{-\frac{1}{4}} H_1(\varepsilon^{\frac{1}{4}}p, q),$  $
\left\{
  \begin{array}{ll}
   \dot{p} = \frac{\partial H(p, q)}{\partial q},  \\
     \dot{q} = - \frac{\partial H(p, q)}{\partial p},
  \end{array}
\right.$ i.e., the symplectic structure is preserved under transformation \emph{(\ref{Eq1p})}.

\end{proof}

\section{Proof of Lemma \ref{balala}} \label{B}

Directly,
\begin{eqnarray}
\nonumber \mu_\nu &=& 64 C_0 \big( 64 C_0\mu_{\nu-2}^{\frac{13}{12}} \big)^{\frac{13}{12}}\\
\nonumber&=& 64 C_0 \big( 64 C_0 (64 C_0 \mu_{\nu-3}^{\frac{13}{12}})^{\frac{13}{12}}  \big)^{\frac{13}{12}}\\
\nonumber&&\cdots\\
\nonumber&=&(64 C_0)^{1+ \frac{13}{12} + \cdots+ (\frac{13}{12})^\nu} \mu_0^{(\frac{13}{12})^\nu}\\
\label{EQ2}&=& (64 C_0)^{ 12((\frac{13}{12})^\nu -1)} \mu_0 ^{(\frac{13}{12})^\nu},\\
\label{EQ24}\alpha_\nu &=& (64 C_0)^{ 4((\frac{13}{12})^\nu -1)} \mu_0 ^{\frac{(\frac{13}{12})^\nu}{3}},\\
\nonumber s_\nu &=& \frac{1}{8} \mu_{\nu-1}^{\frac{1}{3}} \frac{1}{8} \mu_{\nu-2}^{\frac{1}{3}} s_{\nu-2}\\
\nonumber&=&\frac{1}{8} \mu_{\nu-1}^{\frac{1}{3}} \frac{1}{8} \mu_{\nu-2}^{\frac{1}{3}} \frac{1}{8} \mu_{\nu-3}^{\frac{1}{3}} s_{\nu-3}\\
\nonumber&=&(\frac{1}{8})^\nu (\mu_{\nu-1} \mu_{\nu-2} \cdots \mu_0)^{\frac{1}{3}} s_0\\
\nonumber&=&\frac{1}{8^\nu} \big( (64 C_0)^{ \frac{(\frac{13}{12})^{\nu-1} -1}{\frac{1}{12}}} \mu_0 ^{(\frac{13}{12})^{\nu-1}} (64 C_0)^{ \frac{(\frac{13}{12})^{\nu-2} -1}{\frac{1}{12}}} \mu_0 ^{(\frac{13}{12})^{\nu-2}} \cdots \mu_0 \big)^{\frac{1}{3}} s_0\\
\nonumber&=&\frac{1}{8^\nu} \big( (64 C_0)^{ \frac{(\frac{13}{12})^{\nu-1} -1}{\frac{1}{12}} +  \frac{(\frac{13}{12})^{\nu-2} -1}{\frac{1}{12}} + \frac{(\frac{13}{12}) -1}{\frac{1}{12}}}  \mu_0 ^{(\frac{13}{12})^{\nu-1}+ (\frac{13}{12})^{\nu-2} + \cdots +1}  \big)^{\frac{1}{3}} s_0\\
\label{EQ3}&=&\frac{1}{8^\nu} \big( (64 C_0)^{4 (13((\frac{13}{12})^{\nu-1} -1)) } \mu_0 ^{ 4((\frac{13}{12})^{\nu}-1) }\big) s_0.
\end{eqnarray}
Then
\begin{eqnarray}
\nonumber K_+ &=& \big(\big[- \frac{(\frac{13}{12})^\nu -1}{\frac{1}{12}} \log(64 C_0) - (\frac{13}{12})^\nu \log \mu_0\big] +1 \big)^{3 \eta}\\
\nonumber &=& \big(\big[- 12(\frac{13}{12})^\nu  \log(64 C_0) +12 \log(64 C_0) - (\frac{13}{12})^\nu \log \mu_0\big]+1 \big)^{3 \eta}\\
\label{EQ5} &=& \big(\big[(\frac{13}{12})^\nu( -12 \log(64 C_0) - \log \mu_0 ) + 12 \log(64 C_0) \big]+1 \big)^{3 \eta}\\
\nonumber &\geq& (\frac{13}{12})^{3\nu} (\log \frac{1}{\mu_0})^3\\
\nonumber &\geq& 8 (m + l_0) 2^{\nu+2}.
\end{eqnarray}
Combining
\begin{eqnarray}\label{EQ21}
r_\nu - r_{\nu+1} = r_0 (1- \sum\limits_{i=0}^{\nu-1} \frac{1}{2^{i+1}}) - r_0 (1- \sum\limits_{i=0}^{\nu} \frac{1}{2^{i+1}})=\frac{r_0}{2^{\nu+1}},
\end{eqnarray}
we finish the verification of $\bf{(H1)}$ for all $\nu =1,2,\cdots$.

According to $(\ref{EQ5})$, we have
\begin{eqnarray*}
K_{\nu+1}^{2\chi_1} \leq 2^{2\chi_1} (\frac{13}{12})^{3\nu (2\chi_1)} (\log \frac{1}{\mu_0} )^{3\big(2\chi_1\big)}.
\end{eqnarray*}
Then, for small enough $\mu_0$,
\begin{eqnarray*}
s_\nu K_{\nu+1}^{2\chi_1} &\leq& \frac{1}{8^\nu}\frac{1}{(64 C_0)^{\nu-1}} (64 C_0)^{52 ((\frac{13}{12})^{\nu-1} -1)} \mu_0 ^{4 ((\frac{13}{12})^\nu -1)}s_0 \big( 2 (\frac{13}{12})^{3\nu} (\log \frac{1}{\mu_0} )^{3}\big)^{(2\chi_1)}\\
&\leq& \frac{1}{8^\nu} \mu_0 ^{2((\frac{13}{12})^ \nu -1)} s_0 \big( 2 (\frac{13}{12})^{3\nu} (\log \frac{1}{\mu_0} )^{3}\big)^{(2\chi_1)}\\
&\leq& 2^{(2\chi_1)} ( \frac{(\frac{13}{12})^{3(2\chi_1)}}{8})^\nu \mu_0^{(\frac{13}{12})^\nu -1} \mu_0^{(\frac{13}{12})^\nu -1} s_0 (\log \frac{1}{\mu_0})^{3(2\chi_1)}\\
&\leq& 2^{(2\chi_1)} ( \frac{(\frac{13}{12})^{3(2\chi_1)}}{8})^\nu (\mu_0^{\frac{1}{12}})^\nu \mu_0^{(\frac{13}{12})^\nu -1} s_0 (\log \frac{1}{\mu_0})^{3(2\chi_1)}\\
&\leq& \frac{\gamma_0}{2^\nu},\\
\mu_\nu^{\frac{1}{6l_0^2}} K_{\nu+1}^{2\chi_1} &\leq& (64 C_0)^{\frac{12}{6l_0^2}((\frac{13}{12})^\nu -1)} \mu_0 ^{\frac{1}{6l_0^2}(\frac{13}{12})^\nu} 2^{2\chi_1} (\frac{13}{12})^{3\nu (2\chi_1)} (\log \frac{1}{\mu_0} )^{3\big(2\chi_1\big)}\\
&\leq& \frac{\gamma_0}{2^\nu}.
\end{eqnarray*}
Here, we use the fact, for constant $a>0$, $b>0$, $\mu_0^a (\log \frac{1}{\mu_0})^b \rightarrow 0$, as $\mu_0 \rightarrow 0$, which could be verified using finite times L'Hopital's rule. Combining
\begin{eqnarray}
\label{EQ25}\gamma_\nu = \gamma_0 (1- \sum\limits_{i=0}^{\nu-1} \frac{1}{2^{i+1}}) = \frac{\gamma_0}{2^{\nu}},
\end{eqnarray} we verity $\bf{(H4)}$ for $\nu = 1,2, \cdots$.

Since
\begin{eqnarray}
 \nonumber \Gamma _\nu &=& \Gamma_\nu (r_\nu - r_{\nu-1})\\
\nonumber &\leq&\int _1^{\infty} {t}^{\chi}  e^{-\frac{t(r_{\nu} - r_{\nu+1})}{8}} dt\\
 \nonumber  &\leq&   (  2^{\nu+6} e^{-\frac{1}{ 2^{\nu+6}}} + 2^{2(\nu+6)} \chi e^{-\frac{1}{ 2^{\nu+6}}} + \cdots +   2^{(\nu+6)\chi} \chi! e^{-\frac{1}{  2^{\nu+6}}} )\\
\nonumber &\leq& c 2^{(\nu+6)\chi} e^{-\frac{1}{ (2^{\nu+6})}},
\end{eqnarray}
it is clear that
\begin{eqnarray}\label{EQ20}
 \mu_{\nu}\Gamma_{\nu}^3 < (64 C_0)^{\frac{1}{1- \lambda_0} (\frac{13}{12})^\nu -1} \mu_0^{(\frac{13}{12})^\nu} (  2^{(\nu+6)\chi} e^{-\frac{1}{ 2^{\nu+6}}})^3.
\end{eqnarray}
Combining $(\ref{EQ21})$ and $(\ref{EQ20})$, assumption $\bf{(H5)}$ holds for $\nu = 1,2,\cdots$. Using $(\ref{EQ24})$ and $(\ref{EQ20})$, we finish the proof of $\bf{(H6)}$ for $\nu = 1,2,\cdots$. With $(\ref{EQ25})$ and $(\ref{EQ20})$, we verify $\bf{(H7)}$ for $\nu = 1,2,\cdots$. Combining (\ref{EQ30}), (\ref{EQ31}), $(\ref{EQ32})$ and $(\ref{EQ2})$, yield
\begin{eqnarray*}
|\partial_{(y,z)}^j(\hat{h}_\nu - \hat{h}_0)| \leq \sum\limits_{\nu} \mu_\nu \leq \mu_0^{\frac{13}{24}},
\end{eqnarray*}
which implies $\bf{(H3)}$ hold for $\nu = 1,2,\cdots$.

\section{Measure Estimate} \label{C}

\begin{theorem}
Let $\Lambda_* = \bigcap\limits_{\nu=0}^{\infty} \Lambda_\nu$. Assume $(A2)$ hold. Then, for sufficiently small $\delta$,
\begin{eqnarray*}
|\Lambda_0\setminus \Lambda_*| \rightarrow 0 ~as ~\gamma_0 \rightarrow 0.
\end{eqnarray*}
\end{theorem}

\begin{proof}
Let
\begin{eqnarray*}
R_{\nu+1} &=& \{\lambda\in \Lambda_{\nu} (\lambda): |\breve{L}_{k0, \nu}| \leq \frac{\gamma_\nu}{|k|^\tau}, \breve{L}_{k1, \nu}^* \breve{L}_{k1, \nu}  \leq \frac{\gamma_\nu}{|k|^\tau}I_{m+2m_0},\\
 &~&\breve{L}_{k2, \nu}^* \breve{L}_{k2, \nu}\leq \frac{\gamma_\nu}{|k|^\tau}I_{m^2+2m m_0+ 4m_0^2}, for~all~K_{\nu}< |k|\leq K_{\nu+1} \}\\
 &\subset& S_1 \bigcup S_2\bigcup S_3,
\end{eqnarray*}
where $S_1 = \{\lambda\in \Lambda_{\nu}: 0\leq |\breve{L}_{k0,\nu}|\leq \frac{\gamma_\nu}{|k|^\tau}, K_{\nu}< |k|\leq K_{\nu+1}\}$, $S_2 = \{\lambda\in\Lambda_{\nu}: 0\leq \breve{L}_{k1,\nu}^* \breve{L}_{k1, \nu}\leq \frac{\gamma_\nu}{|k|^\tau}, K_{\nu}< |k|\leq K_{\nu+1}\}$, $S_3 = \{\lambda\in\Lambda_{\nu}: 0\leq \breve{L}_{k2,\nu}^* \breve{L}_{k2, \nu}\leq \frac{\gamma_\nu}{|k|^\tau}, K_{\nu}< |k|\leq K_{\nu+1}\}$.

Let $\varsigma = \frac{k}{|k|}\in S^m$, where $S^m$ is a $m$-dimensional ball. For given $\lambda_0\in \Lambda_\nu$, denote $\Omega_\nu (\lambda_0) = \big( \omega_\nu(\lambda_0), \cdots, \partial_\lambda^\alpha \omega_\nu(\lambda_0), \int_0^1 (1-t)^{|\alpha+1|} \partial_\lambda^ {\alpha+1}\omega_\nu(\lambda_0+ t \lambda) \big)$, $\hat{\lambda} = \lambda - \lambda_0 = (\hat{\lambda}_1, \cdots, \hat{\lambda}_m),$ $\tilde{\lambda}=(1, \hat{\lambda}, \cdots, \hat{\lambda}^\alpha, \hat{\lambda}^{\alpha+1})$. Using Taylor series,
\begin{eqnarray*}
\breve{L}_{k0, \nu}= |k|\varsigma^T \Omega_\nu (\lambda_0) \tilde{\lambda}.
\end{eqnarray*}
Let $Q_{\lambda_0, \nu}= (q_{ij})_{\breve{\iota}\times \breve{\iota}}$, where $q_{1\tau_{1}} = q_{2\tau_{2}} = \cdots= q_{m\tau_{m}} = 1$, $\tau_p \neq \tau_q$, $1\leq p,q \leq \breve{\iota}$ and other elements of $Q_{\lambda_0, \nu}$ are equal to 0.
Since $rank \Omega_\nu(\lambda_0) = m$ for $\lambda_0\in \Lambda_\nu\subset \Lambda$, i.e. condition $(A2)$, there is an matrix $Q_{\lambda_0,\nu}= (q_{ij})_{\breve{\iota}\times \breve{\iota}}$ such that $\Omega_\nu(\lambda_0)Q_{\lambda_0,\nu} = \big(A_\nu (\lambda_0), B_\nu(\lambda_0)\big)$, where $A_\nu(\lambda_0) = (a_{ij})_{m\times m}$ is nonsingular. Denote $\Lambda_{\lambda_0,\nu}$ the neighborhood of $\lambda_0$ and $\bar{\Lambda}_{\lambda_0,\nu}$ the closure of $\Lambda_{\lambda_0,\nu}$. Then $\det A_\nu(\lambda) \neq 0$ for $\lambda \in \bar{\Lambda}_{\lambda_0, \nu}$. Therefore, there is an orthogonal matrix $Q_{\lambda_0,\nu}$ such that $\Omega_\nu(\lambda) Q_{\lambda_0, \nu} = (A_{\nu}(\lambda), B_{\nu}(\lambda))$ for $\lambda\in \bar{\Lambda}_{\lambda_0, \nu}$, where $\det A_{\nu}(\lambda) \neq 0$ on $\bar{\Lambda}_{\lambda_0, \nu}.$
Denote the eigenvalues of $(A_{\nu}(\lambda)A_{\nu}^* (\lambda)+ B_{\nu}(\lambda)B_{\nu}^* (\lambda))$ by $\check{\lambda}_{1, \nu}\leq \cdots \leq\check{\lambda}_{n,\nu}$. Since $rank (A_{\nu}(\lambda)A_{\nu}^* (\lambda)+ B_{\nu}(\lambda)B_{\nu}^* (\lambda)) = rank (A_{\nu}(\lambda), B_{\nu}(\lambda))$ (\cite{Horn}), there is a unitary $U_{\nu}$ and a real diagonal $V_{\nu} = diag(\check{\lambda}_{1,\nu}, \cdots, \check{\lambda}_{n,\nu})$ such that $(A_{\nu}(\lambda)A_{\nu}^* (\lambda)+ B_{\nu}(\lambda)B_{\nu}^* (\lambda)) = U_{\nu} V_{\nu} U_{\nu}^*$. Therefore, using Poincar\'{e} separation theorem,
\begin{eqnarray*}
\varsigma^* (A_{\nu}(\lambda)A_{\nu}^* (\lambda)+ B_{\nu}(\lambda)B_{\nu}^* (\lambda)) \varsigma &=& \varsigma^*U_{\nu}^* U_{\nu} (A_{\nu}(\lambda)A_{\nu}^* (\lambda)+ B_{\nu}(\lambda)B_{\nu}^* (\lambda))U_{\nu}^* U_{\nu} \varsigma\\
 &=& \varsigma^*U_{\nu}^*  diag(\check{\lambda}_{1,\nu}, \cdots, \check{\lambda}_{n,\nu}) U_{\nu} \varsigma\\
 &\geq& \varsigma^*U_{\nu}^*  \check{\lambda}_{1,\nu} I_m U_{\nu}\varsigma\\
 &\geq& \check{\lambda}_{1, \nu}.
\end{eqnarray*}
Since the nonzero eigenvalues of $\left(
                                    \begin{array}{cc}
                                      A_{\nu}^T (\lambda) \varsigma\varsigma^T A_{\nu}(\lambda) & A_{\nu}^T (\lambda) \varsigma\varsigma^T B_{\nu}(\lambda) \\
                                      B_{\nu}^T (\lambda) \varsigma\varsigma^TA_{\nu}(\lambda) & B_{\nu}^T(\lambda) \varsigma\varsigma^T B_{\nu}(\lambda) \\
                                    \end{array}
                                  \right)
$ and\\ $\varsigma^T (A_{\nu}(\lambda)A_{\nu}^T (\lambda)+ B_{\nu}(\lambda)B_{\nu}^T (\lambda))\varsigma$ are the same, there is an unitary matrix $U_{\nu}(\lambda)$ such that
\begin{eqnarray*}
\left(
                                    \begin{array}{cc}
                                      A_{\nu}^T (\lambda) \varsigma\varsigma^T A_{\nu}(\lambda) & A_{\nu}^T (\lambda) \varsigma\varsigma^T B_{\nu}(\lambda) \\
                                      B_{\nu}^T (\lambda) \varsigma\varsigma^TA_{\nu}(\lambda) & B_{\nu}^T(\lambda) \varsigma\varsigma^T B_{\nu}(\lambda) \\
                                    \end{array}
                                  \right) = U_{\nu}(\lambda) diag (0, \cdots,0, \check{\lambda}_\nu) U_{\nu}^*(\lambda),
\end{eqnarray*}
where $\check{\lambda}_\nu = \varsigma^*U_{\nu}^*  diag(\check{\lambda}_{1,\nu}, \cdots, \check{\lambda}_{n,\nu}) U_{\nu} \varsigma.$ Denote $(U_{\nu}(\lambda) Q_{\lambda_0,\nu} )_i$ the $i-$th row of $U_{\nu}(\lambda) Q_{\lambda_0,\nu}$. Therefore, $|(U_{\nu}(\lambda) Q_{\lambda_0,\nu}^{-1} \tilde{\lambda})_i| = |(U_{\nu}(\lambda) Q_{\lambda_0, \nu}^{-1} )_i \tilde{\lambda} | \geq (\min\limits_{1\leq j\leq m} |\hat{\lambda}_j|)^{2N+2}$. Hence
\begin{eqnarray*}
|\breve{L}_{k0, \nu}^* \breve{L}_{k0, \nu}| &=& |k|^2 |\tilde{\lambda}^T Q_{\lambda_0, \nu}Q_{\lambda_0, \nu}^{-1} \Omega^T \varsigma\varsigma^T \Omega Q_{\lambda_0, \nu}  Q_{\lambda_0, \nu}^{-1} \tilde{\lambda}|\\
&=& |k|^2 |\tilde{\lambda}^T Q_{\lambda_0, \nu} \left(
                                          \begin{array}{cc}
                                            A_{\nu}^T (\lambda) \varsigma \varsigma^T A_{\nu}(\lambda) & A_{\nu}^T (\lambda) \varsigma\varsigma^T B_{\nu}(\lambda) \\
                                            B_{\nu}^T (\lambda) \varsigma\varsigma^T A_{\nu}(\lambda) & B_{\nu}^T(\lambda) \varsigma\varsigma^T B_{\nu}(\lambda) \\
                                          \end{array}
                                        \right)Q_{\lambda_0,\nu}^{-1} \tilde{\lambda}|\\
&=&|k|^2 |\tilde{\lambda}^T Q_{\lambda_0, \nu}U_{\nu}^*(\lambda) diag (0,\cdots, 0, \check{\lambda}_{\nu}) U_{\nu}(\lambda)Q_{\lambda_0, \nu}^{-1} \bar{\lambda}|\\
&\geq& |k|^2 \check{\lambda}_1 |(U_{\nu}(\lambda) Q_{\lambda_0,\nu}^{-1} \tilde{\lambda})_{i}|\\
&\geq& |k|^2 \check{\lambda}_1 (\min\limits_{1\leq j \leq m} |\hat{\lambda}_i|)^{2N+2}.
\end{eqnarray*}
 Then
\begin{eqnarray*}
~&~& |\{ \lambda\in \Lambda_\nu \bigcap \bar{\Lambda}_{\lambda_0, \nu}: |\breve{L}_{k0, \nu}^* \breve{L}_{k0, \nu}| \leq \frac{\gamma_\nu^2}{|k|^{2\tau}},K_{\nu}< |k|\leq K_{\nu+1} \}| \\
&<& |\{ \lambda\in \Lambda_\nu \bigcap \bar{\Lambda}_{\lambda_0, \nu}: \check{\lambda}_1(\lambda) \big(\min\limits_{j}|\hat{\lambda}_j|\big)^{2N + 2} \leq \frac{\gamma_\nu^2}{|k|^{2(\tau-1)}}, K_{\nu}< |k|\leq K_{\nu+1} \}|\\
 &\leq& c \frac{1}{\check{\lambda}_1} D^{m-1} \frac{\gamma^{\frac{1}{N+1}}}{|k|^{\frac{\tau-1}{N+1}}},
\end{eqnarray*}
where $D$ is the exterior diameter of $\bar{\Lambda}_{\lambda_0,\nu}$ with respect to the maximum norm, $m$ is the dimension of $\Lambda_\nu$. Further, there are finite sets, $\bar{\Lambda}_{\lambda_i, \nu}$, $1\leq i \leq \tilde{\iota}$, such that $\Lambda_\nu\subset \bigcup\limits_{i=1}^{\tilde{\iota}} \bar{\Lambda}_{\lambda_i, \nu}$ and
\begin{eqnarray*}
|\breve{L}_{k0, \nu}^* (\lambda)\breve{L}_{k0, \nu}(\lambda)| > |k|^2 \check{\lambda}_1^{\lambda_i}(\lambda) \big(\min\limits_{j}|\hat{\lambda}_j|\big)^{2N + 2} ~for ~\lambda\in \bar{\Lambda}_{\lambda_i, \nu},
\end{eqnarray*}
where  $\check{\lambda}_1^{\lambda_i}(\lambda)$ the minimum eigenvalue of $\Omega_\nu^*(\lambda) \Omega_\nu(\lambda)$ on $\bar{\Lambda}_{\lambda_i, \nu}$. Therefore,
\begin{eqnarray*}
|S_1| = |\{\lambda\in \Lambda_\nu: |\breve{L}_{k0,\nu}| \leq \frac{\gamma_\nu}{|k|^\tau}, K_{\nu}< |k|\leq K_{\nu+1}\}|< c D^{m-1} \frac{\gamma_0^{\frac{1}{N+1}}}{|k|^{\frac{\tau-1}{N+1}}},
\end{eqnarray*}
where $c$ depends on $\Lambda$, $D$, $m$ and $\check{\lambda}_1^{\lambda_i}$, $1\leq i\leq \tilde{\iota}$.

Denote
\begin{eqnarray*}
B_\nu= \left(
     \begin{array}{ccc}
       0 & (JM_{21, \nu})^T \otimes I_m & 0 \\
       0 & -(M_{22, \nu}J) \otimes I_m & -I_{2m_0}\otimes (2M_{12, \nu}J) \\
       0 & 0 & -(M_{22, \nu}J )\otimes I_{2m_0} - I_{2m_0} \otimes(M_{22, \nu}J) \\
     \end{array}
   \right).
\end{eqnarray*}
Let $\tilde{B}_\nu= - \sqrt{-1}\langle k, \omega_\nu\rangle I_{m^2 + 2m m_0 + 4m_0^2} B_\nu + B_\nu^* \sqrt{-1}\langle k, \omega_\nu\rangle I_{m^2 + 2m m_0 + 4m_0^2} + \delta B_\nu^* B_\nu.$
Combining Poincar\'{e} separation theorem and eigenvalue perturbation theorem (\cite{Horn}), for sufficiently small $\delta$, we have
\begin{eqnarray*}
\breve{L}_{k2, \nu}^*\breve{L}_{k2, \nu} &=& |k|^2 \big( \langle \varsigma, \omega_\nu\rangle\big)^2 I_{m^2 + 2m m_0 + 4m_0^2} + \delta \tilde{B}_\nu\\
&\geq& \frac{|k|^2 }{2} \check{\lambda}_1^{\lambda_i} \big(\min\limits_{j} |\hat{\lambda}_j|\big)^{2N + 2} I_{m^2 + 2m_0 m + 4m_0^2} + \delta \tilde{B}_\nu\\
&\geq& \frac{|k|^2 }{4} \check{\lambda}_1^{\lambda_i} \big(\min\limits_{j} |\hat{\lambda}_j|\big)^{2N + 2} I_{m^2 + 2m_0 m + 4m_0^2}.
\end{eqnarray*}
Therefore,
\begin{eqnarray*}
|S_3| &=& |\{\lambda\in \Lambda_\nu: \breve{L}_{k2,\nu}^* \breve{L}_{k2,\nu} \leq \frac{\gamma_\nu}{|k|^{\tau}} I_{m^2+ + 2m m_0 + 4m_0^2}, K_{\nu}< |k|
\leq K_{\nu+1}\}|\\
&<& c D^{m-1} \frac{\gamma_0^{\frac{1}{2(N+1)}}}{|k|^{\frac{\tau-2}{2(N+1)}}}.
\end{eqnarray*}
Similarly,
\begin{eqnarray*}
|S_2| &=& |\{\lambda\in \Lambda_\nu: \breve{L}_{k1,\nu}^* \breve{L}_{k1,\nu} \leq \frac{\gamma_\nu}{|k|^{\tau}} I_{m + 2m_0}, K_{\nu}< |k|\leq K_{\nu+1}\}|\\
&<& c D^{m-1} \frac{\gamma_0^{\frac{1}{2(N+1)}}}{|k|^{\frac{\tau-2}{2(N+1)}}}.
\end{eqnarray*}

Obviously, $|R_{\nu+1}| \leq  \frac{\gamma_0^{\frac{1}{2(N+1)}}}{|k|^{\frac{\tau-2}{2(N+1)}}}.$ Thus
\begin{eqnarray*}
|\bigcup\limits_{\nu = 0}^{\infty} \bigcup\limits_{K_{\nu}\leq |k| \leq K_{\nu+1}} R_{\nu+1}| \leq c\sum\limits_{\nu=0}^{\infty} \sum\limits_{K_\nu\leq|k|\leq K_{\nu+1}} \frac{\gamma_0^{\frac{1}{2(N+1)}}}{|k|^{\frac{\tau-2}{2(N+1)}}} \rightarrow 0 ~as~\gamma_0\rightarrow 0.
\end{eqnarray*}

\end{proof}

\end{document}